%% file: pthy_mod.tex
\documentclass{article}

\usepackage[numbers]{natbib}
\usepackage[bookmarkstype=toc,colorlinks=true,citecolor=blue,linkcolor=red]{hyperref}  
\usepackage{graphicx}
\usepackage{amsmath,amssymb,amsfonts,amsthm}

\usepackage{geometry}
\usepackage{enumerate}

\newtheorem{theorem}{Theorem}[section]
\newtheorem{lem}{Lemma}[section]
\newtheorem{prop}{Proposition}[section]
\newtheorem{cor}{Corollary}[section]
\theoremstyle{definition}
\newtheorem{Def}{Definition}[section]
\newtheorem*{rmk*}{Remark}
\newtheorem{rmk}{Remark}[section]

\makeatletter
    \renewcommand{\theequation}{
    \thesection.\arabic{equation}}
    \@addtoreset{equation}{section}
\makeatother

\makeatletter
    \renewcommand*{\section}{\@startsection{section}{1}{\z@}%
    {21pt}{12pt}{\reset@font\normalsize\bfseries}}
\makeatother
    
\makeatletter
    \renewcommand*{\subsection}{\@startsection{subsection}{2}{\z@}%
    {15pt}{6pt}{\reset@font\normalsize\mdseries\itshape}}
\makeatother
\if0
\makeatletter
\def\@seccntformat#1{\csname the#1\endcsname.\quad}
\makeatletter
\fi

\makeatletter
\def\@listi{\leftmargin\leftmargini
  \topsep=.5\baselineskip %0pt
  \partopsep=0pt \parsep=0pt \itemsep=0pt}
\let\@listI\@listi
\@listi
\def\@listii{\leftmargin\leftmarginii
  \labelwidth\leftmarginii \advance\labelwidth-\labelsep
  \topsep=0pt \partopsep=0pt \parsep=0pt \itemsep=0pt}
\def\@listiii{\leftmargin\leftmarginiii
  \labelwidth\leftmarginiii \advance\labelwidth-\labelsep
  \topsep=0pt \partopsep=0pt \parsep=0pt \itemsep=0pt}
\def\@listiv{\leftmargin\leftmarginiv
  \labelwidth\leftmarginiv \advance\labelwidth-\labelsep
  \topsep=0pt \partopsep=0pt \parsep=0pt \itemsep=0pt}
\makeatother

\newcommand{\indepe}{\mathop{\perp\!\!\!\perp}}

\title{Estimation of integrated covariances in the simultaneous presence of nonsynchronicity, microstructure noise and jumps}
\author{Yuta Koike
\thanks{University of Tokyo, Graduate School of Mathematical Sciences, 3-8-1 Komaba, Meguro-ku, Tokyo 153-8914, Japan, Email: kyuta@ms.u-tokyo.ac.jp}}

\begin{document}

\maketitle %\thispagestyle{plain}

%{\large\bf NOTE: This paper is not finished.}

\input{pthy/pthy_abstract}

\input{pthy/pthy_intro}
\input{pthy/pthy_review_mod}
%\input{refreshPHY/refreshPHY_construction_end}
\input{pthy/pthy_main}

\input{pthy/pthy_topics}

\input{pthy/pthy_anova_mod}

\input{pthy/pthy_sim_3}
\newpage
\appendix

\renewcommand*{\theequation}{\textup{\Alph{section}}.\arabic{equation}}

\section*{Appendix}

\input{pthy/pthy_proof_mod}

\input{pthy/pthy_propphy}

\input{pthy/pthy_dependent}

\input{pthy/pthy_acknowledgements}
{\small
\addcontentsline{toc}{section}{References}
%\bibliography{base}

\input{pthy_mod.bbl}
}
 
\end{document}

%% file: pthy/pthy_abstract.tex
\begin{abstract}

We propose a new estimator for the integrated covariance of two It\^o semimartingales observed at a high-frequency. This new estimator, which we call the pre-averaged truncated Hayashi-Yoshida estimator, enables us to separate the sum of the co-jumps from the total quadratic covariation even in the case that the sampling schemes of two processes are nonsynchronous and the observation data is polluted by some noise. It is the first estimator which can simultaneously handle these three issues, which are fundamental to empirical studies of high-frequency financial data. We also show the asymptotic mixed normality of this estimator under some mild conditions allowing infinite activity jump processes with finite variations, some dependency between the sampling times and the observed processes as well as a kind of endogenous observation errors. We examine the finite sample performance of this estimator using a Monte Carlo study. \vspace{3mm}

%We will focus on estimating the integrated covariance of two diffusion processes observed in a nonsynchronous manner. The observation data is contaminated by some noise, which is possibly correlated with the returns of the diffusion processes, while the sampling times also possibly depend on the observed processes. %This situation is much more realistic than those in which both of the noise and the sampling times are independent of the diffusion processes. 
%In a high-frequency setting, we consider a modified version of the pre-averaged Hayashi-Yoshida estimator, and we show that such a kind of estimators has the consistency and the asymptotic mixed normality, and attains the optimal rate of convergence. \vspace{3mm}

\noindent \textit{Keywords}: Hayashi-Yoshida estimator; Integrated covariance; Jumps; Market microstructure noise; Nonsynchronous observations; Pre-averaging; Threshold estimator.
\end{abstract}

\if0
\begin{keyword}
Endogenous noise; Hayashi-Yoshida estimator; Integrated covariance; Market microstructure noise; Nonsynchronous observations; Pre-averaging; Stable convergence; Strong predictability
\end{keyword}
\fi

%% file: pthy/pthy_intro.tex
%%%%%%%%%%%%%%%%%%%%%%%%%%%%%%%%%%%%%%%%%%%%%%%%%%%%%%%%%%%%%%%%%%%%%
%                   Introduction
%%%%%%%%%%%%%%%%%%%%%%%%%%%%%%%%%%%%%%%%%%%%%%%%%%%%%%%%%%%%%%%%%%%%%

\section{Introduction}

In the past years there has been a considerable development in statistical inferences for the quadratic covariations of semimartingales observed at a high frequency. This was mainly motivated by financial application because price processes need to follow a semimartingale under the no-arbitrage assumption (see \cite{DS1994} for instance) % Christensen et al. (2011)
and technological developments made high frequency data commonly available. % Ait-Sahalia et al. (2010)
In general the quadratic covariation of two semimartingales consists of two sources; the continuous martingale parts and the co-jumps of the semimartingales. Recently many authors have indicated that separating these two sources benefits various areas of finance such as volatility forecasting (\citet{ABD2007}), credit risk management (\citet{ContKan2011}),  the construction of a hedging portfolio (\citet{TB2010}) and so on. Motivated by these reasons, in this paper we focus on disentangling these two components of the quadratic covariations of two semimartingales by using high-frequency observation data. 

%the construction of a hedging portfolio, credit risk management, forecasting and so on. 

Let $Z^1$ and $Z^2$ be two It\^o semimartingales and let $(S^i)_{i\in\mathbb{Z}_+}$ be a sequence of stopping times that is increasing a.s., $S^i\uparrow\infty$, and $S^0=0$. Then it is well-known in the classic stochastic calculus that
\begin{equation}\label{rc}
\sum_{i:S^i\leq t}(Z^1_{S^i}-Z^1_{S^{i-1}})(Z^2_{S^i}-Z^2_{S^{i-1}})\to^p[Z^1,Z^2]_t
\end{equation}
for any $t>0$, provided $\sup_{i\in\mathbb{N}}(S^i\wedge t-S^{i-1}\wedge t)\to^p0$. Therefore, if we observe $Z^1$ and $Z^2$ at the time $S^i$ for every $i$, we can use the statistic in the left hand of the equation $(\ref{rc})$ (which is called the \textit{realized covariance}) as a consistent estimator of the quadratic covariation $[Z^1,Z^2]_t$ of $Z^1$ and $Z^2$. Since
\begin{equation}\label{qc}
[Z^1,Z^2]_t=\langle Z^{1,c},Z^{2,c}\rangle_t+\sum_{0\leq s\leq t}\Delta Z^1_s\Delta Z^2_s,
\end{equation}
our aim will be achieved by constructing an estimator for the quantity $\langle Z^{1,c},Z^{2,c}\rangle_t$ which we call the \textit{integrated covariance} of $Z^1$ and $Z^2$. In the present situation we have the observation data $(Z^1_{S^i}+Z^2_{S^i})_{i\in\mathbb{Z}_+}$ and $(Z^1_{S^i}-Z^2_{S^i})_{i\in\mathbb{Z}_+}$, so that the problem results in the univariate case due to the polarization identity $\langle Z^{1,c},Z^{2,c}\rangle=(\langle Z^{1,c}+Z^{2,c}\rangle-\langle Z^{1,c}-Z^{2,c}\rangle)/4$. As a consequence, we can benefit from a vast numbers of studies on detecting jumps in a single financial high-frequency data. For example, \citet{BNS2004b} used such a method based on the \textit{bipower technique} introduced in \cite{BNS2004} and proposed an estimator called the \textit{realized bipower covariation}. On the other hand, there are several approaches which directly treat the multivariate data; see \citet{MG2012} and \citet{Boudt2011} for example. 

In real financial markets, however, some difficulties caused by the so-called \textit{market microstructure} confront us. In the present context there are two major topics related to them: One is the nonsynchronicity of observation times and the other is a kind of observation errors called \textit{microstructure noise}. In recent years the simultaneous treatment of these two problems, which is based on the combination of methods for dealing with each individual one, has been established by many authors in the case that jumps are absent. See \citet{BNHLS2011}, \citet{B2011a} and \citet{CKP2010} for example. Furthermore, in the presence of one of the above issues the methodologies of detecting jumps have been also studied in the literature. In the case that observation times are nonsynchronous, \citet{MG2012} combined the \textit{Hayashi-Yoshida method} proposed in \citet{HY2005} to deal with the nonsynchronoicity with the \textit{thresholding technique} proposed independently in \citet{M2001} and \citet{S2002} to detect jumps and constructed a consistent estimator for the integrated covariance. That estimator, which we call the \textit{truncated Hayashi-Yoshida estimator}, was also studied in \cite{Koike2012}. On the other hand, in \citet{PV2009} they proposed a new method for dealing with microstructure noise and introduced a class of bipower-type statistics which goes well in the presence of microstructure noise. Their method is now called the \textit{pre-averaging method} and has been further investigated in \cite{JLMPV2009} and \cite{PV2009SPA} for example. \citet{FW2007} proposed another approach for detecting jumps in the presence of microstructure noise, where wavelet methods were applied to detect jumps. Their approach has further been developed by \cite{BV2012a}. In contrast, we remark that relatively few papers are so far available in a parametric setting with respect to these topics; we refer to \citet{OY2012} for interested readers.  

In the present article we investigate the methodology accommodated to the situation that all of the above problems are present simultaneously. That is, we consider two It\^o semimartingales which are observed at stopping times in a nonsynchronous manner and contaminated by noise. Then we develop a method for estimating their integrated covariance separately from the sum of their co-jumps. For this purpose, we combine the Hayashi-Yoshida method (to deal with the nonsynchronicity of the observation times) and the pre-averaging method (to remove the noise) with the threshold technique (to separate the jumps) and consider a class of statistics called the \textit{pre-averaged truncated Hayashi-Yoshida estimator}. We prove the consistency and the asymptotic mixed normality of the pre-averaged truncated Hayashi-Yoshida estimator under a very general situation allowing the presence of infinite activity jumps, some dependency between the observation times and the observed processes as well as a kind of endogenous noise.  
        
%On the other hand, recently it has become common recognition that at ultra high frequencies the financial data is contaminated by market microstructure noise. Also, it is common that the prices of two assets are observed nonsynchronously. Motivated by these reasons, we consider two semimartingales which are sampled at stopping times in a nonsynchronous manner and contaminated by noise and estimate their cumulative co-volatility separately from the sum of the co-jumps. For this purpose, we combine the Hayashi-Yoshida method (to deal with the nonsynchronicity) and the pre-averaging method (to exclude the noise) with the threshold technique (to separate the jumps) and consider a class of statistics called the pre-averaged truncated Hayashi-Yoshida estimator. We prove the consistency and the asymptotic mixed normality of the pre-averaged truncated Hayashi-Yoshida estimator under the very general situation allowing the presence of infinite activity jumps and dependence of the sampling times on the semimartingales.

This paper is organized as follows: In Section \ref{review} we briefly review on the results about the asymptotic properties of the pre-averaged Hayashi-Yoshida estimator in the continuous It\^o semimartingale setting. In Section \ref{main} we present the construction of our estimator and the main results in this paper. We discuss some topics for the statistical application to finance of our estimator in Section \ref{topics}, while Section \ref{simulation} provides some numerical experiments to illustrate the finite sample properties of our estimator. Most of the proofs are postponed to the Appendix.% \ref{proofthmFAnoise}--\ref{proofdepcon}. 

%% file: pthy/pthy_review_mod.tex
%%%%%%%%%%%%%%%%%%%%%%%%%%%%%%%%%%%%%%%%%%%%%%%%%%%%%%%%%%%%%%%%%
%                   Review of PHY estimator
%%%%%%%%%%%%%%%%%%%%%%%%%%%%%%%%%%%%%%%%%%%%%%%%%%%%%%%%%%%%%%%%%

\section{A brief review of the continuous case}\label{review}

We start by introducing an appropriate stochastic basis on which our observation data is defined. Let $\mathcal{B}^{(0)}=(\Omega^{(0)},\mathcal{F}^{(0)},\mathbf{F}^{(0)}=(\mathcal{F}^{(0)}_t)_{t\in\mathbb{R}_+} ,P^{(0)})$ be a stochastic basis. For any $t\in\mathbb{R}_+$ we have a transition probability $Q_t(\omega^{(0)},\mathrm{d}z)$ from $(\Omega^{(0)},\mathcal{F}^{(0)}_t)$ into $\mathbb{R}^2$, which satisfies
\begin{equation*}%\label{centered}
\int z Q_t(\omega^{(0)},\mathrm{d}z)=0.
\end{equation*}
We endow the space $\Omega^{(1)}=(\mathbb{R}^2)^{[0,\infty)}$ with the product Borel $\sigma$-field $\mathcal{F}^{(1)}$ and with the probability $Q(\omega^{(0)},\mathrm{d}\omega^{(1)})$ which is the product $\otimes_{t\in\mathbb{R}_+}Q_t(\omega^{(0)},\cdot)$. We also call $(\zeta_t)_{t\in\mathbb{R}_+}$ the ``canonical process'' on $(\Omega^{(1)},\mathcal{F}^{(1)})$ and the filtaration $\mathcal{F}^{(1)}_t=\sigma(\zeta_s;s\leq t)$. Then we consider the stochastic basis $\mathcal{B}=(\Omega,\mathcal{F},\mathbf{F}=(\mathcal{F}_t)_{t\in\mathbb{R}_+} ,P)$ defined as follows:
\begin{gather*}
\Omega=\Omega^{(0)}\times\Omega^{(1)},\qquad
\mathcal{F}=\mathcal{F}^{(0)}\otimes\mathcal{F}^{(1)},\qquad
\mathcal{F}_t=\cap_{s>t}\mathcal{F}^{(0)}_s\otimes\mathcal{F}^{(1)}_s,\\
P(\mathrm{d}\omega^{(0)},\mathrm{d}\omega^{(1)})=P^{(0)}(\mathrm{d}\omega^{(0)})Q(\omega^{(0)},\mathrm{d}\omega^{(1)}).
\end{gather*}
Any variable or process which is defined on either $\Omega^{(0)}$ or $\Omega^{(1)}$ can be considered in the usual way as a variable or a process on $\Omega$.

Next we introduce our observation data. There are two continuous semimartingales $X^{1}=(X^{1}_t)_{t\in\mathbb{R}_+}$ and $X^{2}=(X^{2}_t)_{t\in\mathbb{R}_+}$ on $\mathcal{B}^{(0)}$ with canonical decompositions
\begin{equation}\label{CSM}
X^l=A^l+M^l,\qquad l=1,2,
\end{equation}
where $A^1$ and $A^2$ are continuous $\mathbf{F}^{(0)}$-adapted processes with locally finite variations, while $M^1$ and $M^2$ are continuous $\mathbf{F}^{(0)}$-local martingales. We have two sequences of $\mathbf{F}^{(0)}$-stopping times $(S^i)_{i\in\mathbb{Z}_+}$ and $(T^j)_{j\in\mathbb{Z}_+}$ that are increasing a.s.,
\begin{equation}\label{increace}
S^i\uparrow\infty\qquad \textrm{and}\qquad T^j\uparrow\infty.
\end{equation}
As a matter of convenience we set $S^{-1}=T^{-1}=0$. These stopping times implicitly depend on a parameter $n\in\mathbb{N}$, which represents the frequency of the observations. Denote by $(b_n)$ a sequence of positive numbers tending to 0 as $n\to\infty$ (typically $b_n=n^{-1}$). Let $\xi'$ be a constant satisfying $0<\xi'<1$. In this paper, we will always assume that
\begin{equation}\label{A4}
r_n(t):=\sup_{i\in\mathbb{Z}_+}(S^i\wedge t-S^{i-1}\wedge t)\vee\sup_{j\in\mathbb{Z}_+}(T^j\wedge t-T^{j-1}\wedge t)=o_p(b_n^{\xi'})
\end{equation}
as $n\to\infty$ for any $t\in\mathbb{R}_+$.

The processes $X^1$ and $X^2$ are observed at the sampling times $(S^i)$ and $(T^j)$ with observation errors $(U^1_{S^i})_{i\in\mathbb{Z}_+}$ and $(U^2_{T^j})_{j\in\mathbb{Z}_+}$ respectively. In this paper, we assume that the observation errors have the following representations:
\begin{equation}\label{MSnoise}
U^1_{S^i}=b_n^{-1/2}(\underline{X}^1_{S^i}-\underline{X}^1_{S^{i-1}})+\zeta^1_{S^i},\qquad
U^2_{T^j}=b_n^{-1/2}(\underline{X}^2_{T^j}-\underline{X}^2_{T^{j-1}})+\zeta^2_{T^j}.
\end{equation}
Here, $\zeta_t=(\zeta^1_t,\zeta^2_t)$ for each $t$, while $\underline{X}^1$ and $\underline{X}^2$ are two continuous semimartingales on $\mathcal{B}^{(0)}$. After all, we have the observation data $\mathsf{X}^1=(\mathsf{X}^1_{S^i})_{i\in\mathbb{Z}_+}$ and $\mathsf{X}^2=(\mathsf{X}^2_{T^j})_{j\in\mathbb{Z}_+}$ of the form
\begin{equation*}
\mathsf{X}^1_{S^i}=X^1_{S^i}+U^1_{S^i},\qquad\mathsf{X}^2_{T^j}=X^2_{T^j}+U^2_{T^j}.
\end{equation*}

% construction of the estimator

Our aim is to estimate the integrated covariance $[X^1,X^2]_t$ of $X^1$ and $X^2$ at any time $t\in\mathbb{R}_+$ from the observation data $(\mathsf{X}^1_{S^i})_{i:S^i\leq t}$ and $(\mathsf{X}^2_{T^j})_{j:T^j\leq t}$. It is necessary to deal with both of the observation noise and the nonsynchronicity of the observation times simultaneously. As is mentioned in the introduction, we use the pre-averaging technique to remove the noise, while use the Hayashi-Yoshida method to deal with the nonsynchronicity. For the pre-averaging technique we introduce some notation. We choose a sequence $k_n$ of integers and a number $\theta\in(0,\infty)$ satisfying
\begin{equation}\label{window}
k_n=\theta b_n^{-1/2}+o(b_n^{-1/4})
\end{equation}
(for example $k_n=\lceil\theta b_n^{-1/2}\rceil$). We also choose a continuous function $g:[0,1]\rightarrow\mathbb{R}$ which is piecewise $C^1$ with a piecewise Lipschitz derivative $g'$ and satisfies
\begin{equation}\label{weight}
g(0)=g(1)=0,\qquad \psi_{HY}:=\int_0^1 g(x)\mathrm{d}x\neq 0
\end{equation}
(for example $g(x)=x\wedge(1-x)$). We associate the random intervals $I^i=[S^{i-1},S^i)$ and $J^j=[T^{j-1},T^j)$ with the sampling scheme $(S^i)$ and $(T^j)$ and refer to $\mathcal{I}=(I^i)_{i\in\mathbb{N}}$ and $\mathcal{J}=(J^j)_{j\in\mathbb{N}}$ as the sampling designs for $X^1$ and $X^2$. We introduce the \textit{pre-averaging observation data} of $X^1$ and $X^2$ based on the sampling designs $\mathcal{I}$ and $\mathcal{J}$ respectively as follows:
\begin{align*}
\overline{\mathsf{X}}^1(\mathcal{I})^i=\sum_{p=1}^{k_n-1}g\left (\frac{p}{k_n}\right)\left(\mathsf{X}^1_{S^{i+p}}-\mathsf{X}^1_{S^{i+p-1}}\right),\qquad
\overline{\mathsf{X}}^2(\mathcal{J})^j=\sum_{q=1}^{k_n-1}g\left (\frac{q}{k_n}\right)\left(\mathsf{X}^2_{T^{j+q}}-\mathsf{X}^2_{T^{j+q-1}}\right),\\
i,j=0,1,\dots.
\end{align*}

The following quantity was introduced in \citet{CKP2010} :
\begin{Def}[Pre-averaged Hayashi-Yoshida estimator]\label{Defphy}
The \textit{pre-averaged Hayashi-Yoshida estimator}, or \textit{pre-averaged HY estimator} of $\mathsf{X}^1$ and $\mathsf{X}^2$ associated with sampling designs $\mathcal{I}$ and $\mathcal{J}$ is the process
\begin{equation*}
PHY(\mathsf{X}^1,\mathsf{X}^2;\mathcal{I},\mathcal{J})^n_t
=\frac{1}{(\psi_{HY}k_n)^2}\sum_{\begin{subarray}{c}
i,j=0\\
S^{i+k_n}\vee T^{j+k_n}\leq t
\end{subarray}}^{\infty}\overline{\mathsf{X}}^1(\mathcal{I})^i\overline{\mathsf{X}}^2(\mathcal{J})^j 1_{\{[S^i,S^{i+k_n})\cap[T^j,T^{j+k_n})\neq\emptyset\}},\qquad t\in\mathbb{R}_+.
\end{equation*}
\end{Def}

For a technical reason explained in \cite{Koike2012phy}, we modify the above estimator as follows. The following notion was introduced to this area in \citet{BNHLS2011}:
\begin{Def}[Refresh time]
The first refresh time of sampling designs $\mathcal{I}$ and $\mathcal{J}$ is defined as $R^0=S^0\vee T^0$, and then subsequent refresh times as
\begin{align*}
R^k:=\min\{S^i|S^i>R^{k-1}\}\vee\min\{T^j|T^j>R^{k-1}\},\qquad k=1,2,\dots.
\end{align*}
\end{Def}

We introduce new sampling schemes by a kind of the next-tick interpolations to the refresh times. That is, we define $\widehat{S}^0:=S^0$, $\widehat{T}^0:=T^0$, and
\begin{align*}
\widehat{S}^k:=\min\{S^i|S^i>R^{k-1}\},\quad\widehat{T}^k:=\min\{T^j|T^j>R^{k-1}\},\qquad k=1,2,\dots.
\end{align*}
Then, we create new sampling designs as follows:
\begin{align*}
\widehat{I}^k:=[\widehat{S}^{k-1},\widehat{S}^k),\qquad\widehat{J}^k:=[\widehat{T}^{k-1},\widehat{T}^k),\qquad
\widehat{\mathcal{I}}:=(\widehat{I}^i)_{i\in\mathbb{N}},\qquad\widehat{\mathcal{J}}:=(\widehat{J}^j)_{j\in\mathbb{N}}.
\end{align*}
For the sampling designs $\widehat{\mathcal{I}}$ and $\widehat{\mathcal{J}}$ obtained in such a manner, we will consider the pre-averaged HY estimator $\widehat{PHY}(\mathsf{X}^1,\mathsf{X}^2)^n:=PHY(\mathsf{X}^1,\mathsf{X}^2;\widehat{\mathcal{I}},\widehat{\mathcal{J}})^n$.

Now we review the results related to the consistency and the asymptotic mixed normality of the estimator $\widehat{PHY}(\mathsf{X}^1,\mathsf{X}^2)^n$. We write the canonical decompositions of $\underline{X}^1$ and $\underline{X}^2$ as follows:
\begin{equation}\label{noiseCSM}
\underline{X}^l=\underline{A}^l+\underline{M}^l,\qquad l=1,2.
\end{equation}
Here, $\underline{A}^1$ and $\underline{A}^2$ are continuous $\mathbf{F}^{(0)}$-adapted processes with locally finite variations, while $\underline{M}^1$ and $\underline{M}^2$ are continuous $\mathbf{F}^{(0)}$-local martingales.
Next, let $N^n_t=\sum_{k=1}^{\infty}1_{\{R^k\leq t\}}$ for each $t\in\mathbb{R}_+$, and we introduce the following regularity conditions:
\begin{enumerate}[{[{C}1]}]
\item $b_nN^n_t=O_p(1)$ as $n\to\infty$ for every $t$.
\item  $A^1$, $A^2$, $\underline{A}^1$, $\underline{A}^2$, and $[V,W]$ for $V,W=X^1,X^2,\underline{X}^1,\underline{X}^2$ are absolutely continuous with locally bounded derivatives.
\end{enumerate}
Furthermore, for every $r\in[2,\infty)$ we introduce the following regularity condition for noise:
\begin{enumerate}
\item[{[N$^\flat_r$]}] $(\int |z|^rQ_t(\mathrm{d}z))_{t\in\mathbb{R}_+}$ is a locally bounded process.
\end{enumerate}
A sequence $(X^n)$ of stochastic processes is said to converge to a process $X$ \textit{uniformly on compacts in probability} (abbreviated \textit{ucp}) if, for each $t>0$, $\sup_{0\leq s\leq t}|X^n_s-X_s|\rightarrow^p0$ as $n\rightarrow\infty$. We then write $X^n\xrightarrow{ucp}X$. We have the following result about the consistency of the pre-averaged HY estimator:
\begin{theorem}[\protect\cite{Koike2012phy}, Theorem 5.1]\label{consistency}
Suppose $(\ref{A4})$, $[\mathrm{C}1]$-$[\mathrm{C}2]$ and $[\mathrm{N}^\flat_2]$ are satisfied. Then $$\widehat{PHY}(\mathsf{X}^1,\mathsf{X}^2)^n\xrightarrow{ucp}[X^1,X^2]$$ as $n\to\infty$, provided that $\xi'>1/2$.
\end{theorem}
\noindent The consistency of the pre-averaged HY estimator was first shown in \citet{CKP2010} in a simpler situation.

Next we review the results related to the asymptotic mixed normality of the pre-averaged HY estimator, which was first proven in \cite{CPV2011} when the sampling times are deterministic transformation of equidistant ones. In this paper we treat general sampling times, so that we review the result given in \cite{Koike2012phy}.

Let $N^{n,1}_t=\sum_{k=1}^{\infty}1_{\{\widehat{S}^k\leq t\}}$ and $N^{n,2}_t=\sum_{k=1}^{\infty}1_{\{\widehat{T}^k\leq t\}}$ for each $t\in\mathbb{R}_+$ and 
\begin{align*}
\Gamma^k=[R^{k-1},R^k),\qquad \check{I}^k:=[\check{S}^k,\widehat{S}^k),\qquad\check{J}^k:=[\check{T}^k,\widehat{T}^k)
\end{align*}
for each $k\in\mathbb{N}$. Here, for each $t\in\mathbb{R}_+$ we write $\check{S}^k=\sup_{S^i<\widehat{S}^k}S^i$ and $\check{T}^k=\sup_{T^j<\widehat{T}^k}T^j$. Note that $\check{S}^k$ and $\check{T}^k$ may not be stopping times.  

Let $\xi$ be a positive constant satisfying $\frac{1}{2}<\xi<1$. Furthermore, let $\mathbf{H}^n=(\mathcal{H}^n_t)_{t\in\mathbb{R}_+}$ be a sequence of filtrations of $\mathcal{F}$ to which $N^n$, $N^{n,1}$ and $N^{n,2}$ are adapted, and for each $n$ and each $\rho\geq0$ we define the processes $\chi^n$, $G(\rho)^n$, $F(\rho)^{n,1}$, $F(\rho)^{n,2}$ and $F(1)^{n,1* 2}$ by
\begin{gather*}
\chi^n_{s}=P(\widehat{S}^k=\widehat{T}^k\big|\mathcal{H}_{R^{k-1}}^n),\qquad
G(\rho)^n_s=E\left[\left(b_n^{-1}|\Gamma^{k}|\right)^\rho\big|\mathcal{H}_{R^{k-1}}^n\right],\\
F(\rho)^{n,1}_{s}=E\left[\left(b_n^{-1}|\check{I}^{k}|\right)^\rho\big|\mathcal{H}_{\widehat{S}^{k-1}}^n\right],\qquad
F(\rho)^{n,2}_{s}=E\left[\left(b_n^{-1}|\check{J}^{k}|\right)^\rho\big|\mathcal{H}_{\widehat{T}^{k-1}}^n\right],\\
F(1)^{n,1*2}_{s}=b_n^{-1}E\left[|\check{I}^{k}\cap\check{J}^k|+|\check{I}^{k+1}\cap\check{J}^k|+|\check{I}^{k}\cap\check{J}^{k+1}|\big|\mathcal{H}_{R^{k-1}}^n\right]
\end{gather*}
when $s\in\Gamma^k$. Here, $|\cdot|$ denotes the Lebesgue measure.

The following condition is necessary to compute the asymptotic variance of the estimation error of our estimator explicitly. For a sequence $(X^n)$ of c\`adl\`ag processes and a c\`adl\`ag process $X$, we write $X^n\xrightarrow{\text{Sk.p.}}X$ if $(X^n)$ converges to $X$ in probability for the Skorokhod topology.
\begin{enumerate}
\item[{[A1$'$]}]

(i) For each $n$, we have a c\`adl\`ag $\mathbf{H}^{n}$-adapted process $G^n$ and a random subset $\mathcal{N}^0_n$ of $\mathbb{N}$ such that $(\#\mathcal{N}^0_n)_{n\in\mathbb{N}}$ is tight, $G(1)^n_{R^{k-1}}=G^n_{R^{k-1}}$ for any $k\in\mathbb{N}-\mathcal{N}^0_n$, and there exists a c\`adl\`ag $\mathbf{F}^{(0)}$-adapted process $G$ satisfying that $G$ and $G_{-}$ do not vanish and that $G^n\xrightarrow{\text{Sk.p.}}G$ as $n\to\infty$.

(ii) There exists a constant $\rho\geq1/\xi'$ such that $\left(\sup_{0\leq s\leq t}G(\rho)^n_{s}\right)_{n\in\mathbb{N}}$ is tight for all $t>0$.

(iii) For each $n$, we have a c\`adl\`ag $\mathbf{H}^{n}$-adapted process $\chi^{\prime n}$ and a random subset $\mathcal{N}'_n$ of $\mathbb{N}$ such that $(\#\mathcal{N}'_n)_{n\in\mathbb{N}}$ is tight, $\chi^n_{R^{k-1}}=\chi^{\prime n}_{R^{k-1}}$ for any $k\in\mathbb{N}-\mathcal{N}'_n$, and there exists a c\`adl\`ag $\mathbf{F}^{(0)}$-adapted process $\chi$ such that $\chi^{\prime n}\xrightarrow{\text{Sk.p.}}\chi$ as $n\to\infty$.

(iv) For each $n$ and $l=1,2,1*2$, we have a c\`adl\`ag $\mathbf{H}^{n}$-adapted process $F^{n,l}$ and a random subset $\mathcal{N}^l_n$ of $\mathbb{N}$ such that $(\#\mathcal{N}^l_n)_{n\in\mathbb{N}}$ is tight, $F(1)^{n,l}_{R^{k-1}}=F^{n,l}_{R^{k-1}}$ for any $k\in\mathbb{N}-\mathcal{N}^l_n$, and there exists a c\`adl\`ag $\mathbf{F}^{(0)}$-adapted processes $F^l$ satisfying $F^{n,l}\xrightarrow{\text{Sk.p.}}F^l$ as $n\to\infty$.

(v) There exists a constant $\rho'\geq1/\xi'$ such that $\left(\sup_{0\leq s\leq t}F(\rho')^{n,l}_{s}\right)_{n\in\mathbb{N}}$ is tight for all $t>0$ and $l=1,2$.

\end{enumerate}

The following condition is a sufficient one for the condition [A1$'$]:
\begin{enumerate}
\item[{[A1$^{\prime\sharp}$]}]

(i) For every $\rho\in[0,1/\xi']$ there exists a c\`adl\`ag $\mathbf{F}^{(0)}$-adapted process $G(\rho)$ such that $G(\rho)^n\xrightarrow{\text{Sk.p.}}G(\rho)$ as $n\to\infty$. Furthermore, $G$ and $G_{-}$ do not vanish, where $G=G(1)$.

(ii) There exists a c\`adl\`ag $\mathbf{F}^{(0)}$-adapted process $\chi$ such that $\chi^{n}\xrightarrow{\text{Sk.p.}}\chi$ as $n\to\infty$.

(iii) For every $l=1,2$ and every $\rho'\in[0,1/\xi']$, there exists a c\`adl\`ag $\mathbf{F}^{(0)}$-adapted process $F(\rho)^l$ such that $F(\rho)^{n,l}\xrightarrow{\text{Sk.p.}}F(\rho)^l$ as $n\to\infty$.

(iv) There exists a c\`adl\`ag $\mathbf{F}^{(0)}$-adapted process $F(1)^{1*2}$ such that $F(1)^{n,1*2}\xrightarrow{\text{Sk.p.}}F(1)^{1*2}$ as $n\to\infty$.

\end{enumerate}

\begin{rmk}
An [A1$^{\prime\sharp}$] type condition appears in \cite{BNHLS2011} and \cite{HJY2011}, for example. The reason why we introduce a kind of exceptional sets $\mathcal{N}^l_n$ $(l=0,1,2,1*2,')$ is that the condition [A1$'$] without them is too local. To explain this, we focus on the univariate case. Note that in this case we have $R^k=S^k$ $(k=0,1,2,\dots)$. Let $\tau$ be a positive number and suppose that $(S^i)$ be a sequence of Poisson arrival times whose intensity is $\underline{\lambda}$ before the time $\tau$ and $\overline{\lambda}$ after $\tau$. Then the structure of the process $G(1)^n$ becomes very complex around the time $\tau$ (of course if $\underline{\lambda}\neq\overline{\lambda}$), so that it will be difficult to verify the convergence $G(1)^n\xrightarrow{\text{Sk.p.}}G$ because it requires a kind of uniformity. See Section 5.2 of \cite{Koike2012phy} for more precise discussion.
\end{rmk}

Next, we introduce the following strong predictability condition for the sampling designs, which is an analog to the condition [A2] in \cite{HY2011}.
\begin{enumerate}
\item[{[A2]}] For every $n,i\in\mathbb{N}$, $S^i$ and $T^i$ are $\mathbf{G}^{(n)}$-stopping times, where $\mathbf{G}^{(n)}=(\mathcal{G}^{(n)}_t)_{t\in\mathbb{R}_+}$ is the filtration given by $\mathcal{G}^{(n)}_t=\mathcal{F}_{(t-b_n^{\xi-1/2})_+}^{(0)}$ for $t\in\mathbb{R}_+$.
\end{enumerate}

%For real-valued functions $x$ on $\mathbb{R}_+$, the \textit{modulus of continuity} on $[0,T]$ is denoted by $w(x;\delta,T)=\sup\left\{|x(t)-x(s)|\big|s,t\in[0,T], |s-t|\leq \delta\right\}$ for $T,\delta>0$. Furthermore, if a process $V$ is (pathwise) absolutely continuous, we denote its density process by $V'$. 
The following conditions are analogs to the conditions [A3] and [A4] in \cite{HY2011}: 
\begin{enumerate}
%\item[[{A3]}] $[ M^X ]$, $[ M^Y]$, $[ M^X,M^Y]$ are absolutely continuous, and for the density process $f=[ M^X]', [ M^Y]'$ and $[ M^X,M^Y]'$, $w(f;h,t)=O_p(h^{\frac{1}{2}-\lambda})$ as $h\rightarrow 0$ for every $t,\lambda\in(0,\infty)$, and $f$ is adapted to $\mathbf{H}^n$.
\item[[{A3]}] For each $V,W=X^1,X^2,\underline{X}^1,\underline{X}^2$, $[V,W]$ is absolutely continuous with a c\`adl\`ag derivative, and for the density process $f=[V,W]'$ there is a sequence $(\sigma_k)$ of $\mathbf{F}^{(0)}$-stopping times such that $\sigma_k\uparrow\infty$ as $k\to\infty$ and for every $k$ and any $\lambda>0$ we have a positive constant $C_{k,\lambda}$ satisfying
\begin{equation}\label{eqA3}
E\left[|f^{\sigma_k}_{\tau_1}-f^{\sigma_k}_{\tau_2}|^2\big|\mathcal{F}_{\tau_1\wedge\tau_2}\right]\leq C_{k,\lambda}E\left[|\tau_1-\tau_2|^{1-\lambda}\big|\mathcal{F}_{\tau_1\wedge\tau_2}\right]
\end{equation}
for any bounded $\mathbf{F}^{(0)}$-stopping times $\tau_1$ and $\tau_2$, and $f$ is adapted to $\mathbf{H}^n$.
\item[{[A4]}] $\xi\vee\frac{9}{10}<\xi'$ and $(\ref{A4})$ holds for every $t\in\mathbb{R}_+$.
\end{enumerate}

The following conditions, which are analogs to the conditions [A5] and [A6] in \cite{HY2011}, are necessary to deal with the drift parts. For a (random) interval $I$ and a time $t\in\mathbb{R}_+$, we write $I(t)=I\cap[0,t)$.
\begin{enumerate}
%\item[{[A5]}] $A^{X}$ and $A^{Y}$ are absolutely continuous, and $w(f;h,t)=O_p(h^{\frac{1}{2}-\lambda})$ as $h\rightarrow 0$ for every $t\in\mathbb{R}_+$ and some $\lambda\in(0,3/8)$, for the density processes $f=(A^X)'$ and $(A^{Y})'$.
\item[[{A5]}] $A^{1}$, $A^{2}$, $\underline{A}^1$ and $\underline{A}^2$ are absolutely continuous with c\`adl\`ag derivatives, and there is a sequence $(\sigma_k)$ of $\mathbf{F}^{(0)}$-stopping times such that $\sigma_k\uparrow\infty$ as $k\to\infty$ and for every $k$ we have a positive constant $C_{k}$ and $\lambda_k\in(0,3/4)$ satisfying
\begin{equation}\label{eqA5}
E\left[|f^{\sigma_k}_t-f^{\sigma_k}_{\tau}|^2\big|\mathcal{F}_{\tau\wedge t}\right]\leq C_{k}E\left[|t-\tau|^{1-\lambda_k}\big|\mathcal{F}_{\tau\wedge t}\right]
\end{equation}
for every $t$ and any bounded $\mathbf{F}^{(0)}$-stopping time $\tau$, for the density processes $f=(A^1)'$, $(A^{2})'$, $(\underline{A}^1)'$ and $(\underline{A}^2)'$.
\item[{[A6]}] For each $t\in\mathbb{R}_+$, $b_n^{-1}H_n(t)=O_p(1)$ as $n\to\infty$, where $H_n(t)=\sum_{k=1}^{\infty}|\Gamma^k(t)|^2$.
\end{enumerate}

Let $r\in[2,\infty)$. The following condition is a regularity condition for the noise process:
\begin{enumerate}
%\item[{[N]}] $(\int |z|^8Q_t(\mathrm{d}z))_{t\in\mathbb{R}_+}$ is a locally bounded process and $w(\Psi^{ij};h,t)=O_p(h^{\frac{1}{2}-\lambda})$ as $h\rightarrow 0$ for every $i,j\in\{1,2\}$ and $t,\lambda\in(0,\infty)$, where
\item[{[N$_r$]}] $(\int |z|^rQ_t(\mathrm{d}z))_{t\in\mathbb{R}_+}$ is a locally bounded process, and the covariance matrix process
\begin{equation}\label{defPsi}
\Psi_t(\omega^{(0)})=\int zz^*Q_t(\omega^{(0)},\mathrm{d}z).
\end{equation}
is c\`adl\`ag and quasi-left continuous. Furthermore, there is a sequence $(\sigma^k)$ of $\mathbf{F}^{(0)}$-stopping times such that $\sigma^k\uparrow\infty$ as $k\to\infty$ and for every $k$ and any $\lambda>0$ we have a positive constant $C_{k,\lambda}$ satisfying
\begin{equation}\label{eqN}
E\left[|\Psi^{ij}_{\sigma^k\wedge t}-\Psi^{ij}_{\sigma^k\wedge (t-h)_+}|^2\big|\mathcal{F}_{(t-h)_+}\right]\leq C_{k,\lambda} h^{1-\lambda}
\end{equation}
for every $i,j\in\{1,2\}$ and every $t,h>0$.
\end{enumerate}

\begin{rmk}
The inequalities $(\ref{eqA3})$, $(\ref{eqA5})$ and $(\ref{eqN})$ are satisfied when $w(f;h,t)=O_p(h^{\frac{1}{2}-\lambda})$ as $h\to\infty$ for every $t,\lambda\in(0,\infty)$, for example. Here, for a real-valued function $x$ on $\mathbb{R}_+$, the \textit{modulus of continuity} on $[0,T]$ is denoted by $w(x;\delta,T)=\sup\{|x(t)-x(s)|;s,t\in[0,T],|s-t|\leq\delta\}$ for $T,\delta>0$. This is the original condition in \cite{HY2011}. Another such example is the case that there exist an $\mathbf{F}^{(0)}$-adapted process $B$ with a locally integrable variation and a locally square-integrable martingale $L$ such that $f=B+L$ and both of the predictable compensator of the variation process of $B$ and the predictable quadratic variation of $L$ are absolutely continuous with locally bounded derivatives. This type of condition is familiar in the context of the estimation of volatility-type quantities; see \cite{HJY2011} and \cite{JPV2010} for example. Furthermore, in both of the cases $f$ is c\`adl\`ag and quasi-left continuous.
\end{rmk}

We extend the functions $g$ and $g'$ to the whole real line by setting $g(x)=g'(x)=0$ for $x\notin[0,1]$. Then we put
\begin{gather*}
\kappa:=\int_{-2}^{2}\psi_{g,g}(x)^2\mathrm{d}x,\qquad
\widetilde{\kappa}:=\int_{-2}^{2}\psi_{g',g'}(x)^2\mathrm{d}x,\qquad
\overline{\kappa}:=\int_{-2}^{2}\psi_{g,g'}(x)^2\mathrm{d}.
\end{gather*}
Here, for each $f_1,f_2\in\{g,g'\}$ we define the function $\psi_{f_1,f_2}$ on $\mathbb{R}$ by $\psi_{f_1,f_2}(x)=\int_0^1\int_{x+u-1}^{x+u+1}f_1(u)f_2(v)\mathrm{d}v\mathrm{d}u$.

We denote by $\mathbb{D}(\mathbb{R}_+)$ the space of c\`adl\`ag functions on $\mathbb{R}_+$ equipped with the Skorokhod topology. A sequence of random elements $X^n$ defined on a probability space $(\Omega,\mathcal{F},P)$ is said to \textit{converge stably in law} to a random element $X$ defined on an appropriate extension $(\tilde{\Omega},\tilde{\mathcal{F}} ,\tilde{P})$ of $(\Omega,\mathcal{F},P)$ if $E[Yg(X^n)]\rightarrow E[Yg(X)]$ for any $\mathcal{F}$-measurable and bounded random variable $Y$ and any bounded and continuous function $g$. We then write $X^n\rightarrow^{d_s}X$. 

Now we are ready to state the result related to the asymptotic mixed normality of the pre-averaged HY estimator.
\begin{theorem}[\protect\cite{Koike2012phy}, Theorem 3.1]\label{AMN}
%\begin{enumerate}[(a)]

$(\mathrm{a})$ Suppose $[\mathrm{A}1'](\mathrm{i})$-$(\mathrm{iii})$, $[\mathrm{A}2]$-$[\mathrm{A}6]$ and $[\mathrm{N}_8]$ are satisfied. Suppose also $\underline{X}^1=\underline{X}^2=0$. Then
\begin{align*}
b_n^{-1/4}\{\widehat{PHY}(\mathsf{X}^1,\mathsf{X}^2)^n-[X^1,X^2]\}\to^{d_s}\int_0^\cdot w_s\mathrm{d}\widetilde{W}_s\qquad\mathrm{in}\ \mathbb{D}(\mathbb{R}_+)
\end{align*}
as $n\to\infty$, where $\tilde{W}$ is a one-dimensional standard Wiener process (defined on an extension of $\mathcal{B}$) independent of $\mathcal{F}$ and $w$ is given by
\begin{align}
w_s^2=\psi_{HY}^{-4}[&\theta\kappa\{[X^1]'_s[X^2]'_s+([X^1,X^2]'_s)^2\}G_s
+\theta^{-3}\widetilde{\kappa}\{\Psi^{11}_s\Psi^{22}_s+\left(\Psi^{12}_s\chi_s\right)^2\}G_s^{-1}\nonumber\\
&+\theta^{-1}\overline{\kappa}\{[X^1]'_s\Psi^{22}_s+[X^2]'_s\Psi^{11}_s+2[X^1,X^2]'_s\Psi^{12}_s\chi_s\}].\label{avar}
\end{align}

%\begin{enumerate}
%\item[$(\mathrm{b})$]
\noindent $(\mathrm{b})$ Suppose $[\mathrm{A}1']$, $[\mathrm{A}2]$-$[\mathrm{A}6]$ and $[\mathrm{N}_8]$ are satisfied. Then
\begin{align*}
b_n^{-1/4}\{\widehat{PHY}(\mathsf{X}^1,\mathsf{X}^2)^n-[X^1,X^2]\}\to^{d_s}\int_0^\cdot w_s\mathrm{d}\widetilde{W}_s\qquad\mathrm{in}\ \mathbb{D}(\mathbb{R}_+)
\end{align*}
as $n\to\infty$, where $\tilde{W}$ is as in the above and $w$ is given by
\begin{align}
w_s^2=\psi_{HY}^{-4}\bigg[&\theta\kappa\left\{[X^1]'_s[X^2]'_s+([X^1,X^2]'_s)^2\right\}G_s
+\theta^{-3}\widetilde{\kappa}\left\{\overline{\Psi}^{11}_s\overline{\Psi}^{22}_s+\left(\overline{\Psi}^{12}_s\right)^2\right\}G_s^{-1}\nonumber\\
&+\theta^{-1}\overline{\kappa}\left\{[X^1]'_s\overline{\Psi}^{22}_s+[X^2]'_s\overline{\Psi}^{11}_s+2[X^1,X^2]'_s\overline{\Psi}^{12}_s-\left([\underline{X}^1,X^2]'_s F^1_s-[X^1,\underline{X}^2]'_s F^2_s\right)^2 G_s^{-1}\right\}\Bigg],\label{avarend}
\end{align}
where
$\overline{\Psi}^{ll}_s=\Psi^{ll}_s+[\underline{X}^l]'_s F^l_s$ $(l=1,2)$ and
$\overline{\Psi}^{12}_s=\Psi^{12}_s\chi_s+[\underline{X}^1,\underline{X}^2]'_s F^{1* 2}_s.$
%\end{enumerate} 
\end{theorem}

%% file: pthy/pthy_main.tex
%%%%%%%%%%%%%%%%%%%%%%%%%%%%%%%%%%%%%%%%%%%%%%%%%%%%%%%%%%%%%%%%%%%
%                   Main results
%%%%%%%%%%%%%%%%%%%%%%%%%%%%%%%%%%%%%%%%%%%%%%%%%%%%%%%%%%%%%%%%%

\section{Main results}\label{main}

In this section we investigate the case that the latent processes possibly have jumps. Let $Z^{1}=(Z^{1}_t)_{t\in\mathbb{R}_+}$ and $Z^{2}=(Z^{2}_t)_{t\in\mathbb{R}_+}$ be two stochastic processes on $(\Omega^{(0)},\mathcal{F}^{(0)},P^{(0)})$. We have the observation data $\mathsf{Z}^1=(\mathsf{Z}^1_{S^i})_{i\in\mathbb{Z}_+}$ and $\mathsf{Z}^2=(\mathsf{Z}^2_{T^j})_{j\in\mathbb{Z}_+}$ of $Z^1$ and $Z^2$ contaminated by noise:
\begin{equation*}
\mathsf{Z}^1_{S^i}=Z^1_{S^i}+U^1_{S^i},\qquad\mathsf{Z}^2_{T^j}=Z^2_{T^j}+U^2_{T^j}.
\end{equation*}
Here, the observation noise $(U^1_{S^i})_{i\in\mathbb{Z}_+}$ and $(U^2_{T^j})_{j\in\mathbb{Z}_+}$ are given by $(\ref{MSnoise})$.

The idea for the construction of our estimator as follows. The pre-averaging procedure smooths the noise and thus we can expect the pre-averaged data $\bar{\mathsf{Z}}^{1}(\widehat{\mathcal{I}})^i$ and $\bar{\mathsf{Z}}^{2}(\widehat{\mathcal{J}})^j$ are small enough if they contain no jumps. This idea has already appeared in \citet{AJL2011} and \citet{PZ2010} in the univariate case and \citet{JLL2011} in the synchronous case. Following this idea, we introduce the following quantity:
\begin{Def}[Pre-averaged truncated Hayashi-Yoshida estimator]
The \textit{pre-averaged truncated Hayashi-Yoshida estimator}, or \textit{PTHY estimator} of two observation data $\mathsf{Z}^{1}$ and $\mathsf{Z}^{2}$ is the process
\begin{align*}
&\widehat{PTHY}(\mathsf{Z}^{1},\mathsf{Z}^{2})^n_t\\
=&\frac{1}{(\psi_{HY}k_n)^2}\sum_{i,j: \widehat{S}^{i+k_n}\vee \widehat{T}^{j+k_n}\leq t}\bar{\mathsf{Z}}^{1}(\widehat{\mathcal{I}})^i\bar{\mathsf{Z}}^{2}(\widehat{\mathcal{J}})^j \bar{K}^{i j} 1_{\{\bar{\mathsf{Z}}^{1}(\widehat{\mathcal{I}})^i|^2\leq\varrho_n^{1}(\widehat{S}^i),\bar{\mathsf{Z}}^{2}(\widehat{\mathcal{J}})^j|^2\leq\varrho_n^{2}(\widehat{T}^j)\}},\qquad t\in\mathbb{R}_+,
\end{align*}
where
\begin{align*}
\overline{\mathsf{Z}}^1(\widehat{\mathcal{I}})^i=\sum_{p=1}^{k_n-1}g\left (\frac{p}{k_n}\right)\left(\mathsf{Z}^1_{\widehat{S}^{i+p}}-\mathsf{Z}^1_{\widehat{S}^{i+p-1}}\right),\qquad
\overline{\mathsf{Z}}^2(\widehat{\mathcal{J}})^j=\sum_{q=1}^{k_n-1}g\left (\frac{q}{k_n}\right)\left(\mathsf{Z}^2_{\widehat{T}^{j+q}}-\mathsf{Z}^2_{\widehat{T}^{j+q-1}}\right),\\
i,j=0,1,\dots,
\end{align*}
$\bar{K}^{ij}=1_{\{[\widehat{S}^i,\widehat{S}^{i+k_n})\cap[\widehat{T}^j,\widehat{T}^{j+k_n})\neq\emptyset\}}$ and $(\varrho_n^{l}(t))_{n\in\mathbb{N}}$, $l=1,2$, are two sequences of positive-valued stochastic processes.
\end{Def}

We will write $\varrho_n^{1}[i]:=\varrho_n^{1}(\widehat{S}^i)$ and $\varrho_n^{2}[j]:=\varrho_n^{2}(\widehat{T}^j)$ for short. The above statistic was originally considered in \cite{Koike2012} without the refresh sampling modification. Recently \citet{WLL2013} also introduced such a statistic. In the present article we include this modification in order to obtain the central limit theorem for the estimator with a broad class of sampling schemes by using Theorem \ref{AMN}.   

%%%%%%%%%%%%%%%%%%%%%%%%%%%%%%%%%%%%%%%%%%%%%%%%%%%%%%%%%%%%%%%%
%               Finite activity case
%%%%%%%%%%%%%%%%%%%%%%%%%%%%%%%%%%%%%%%%%%%%%%%%%%%%%%%%%%%%%%%

\subsection{Finite activity jump case}

First we consider the case that the observed processes have at most finite jumps. We assume the following structural assumption:
\begin{enumerate}
\item[{[F]}] For each $l=1,2$ we have
$Z^{l}_t=X^{l}_t+\sum_{k=1}^{N^{l}_t}\gamma^{l}_k,$
where $X^l$ is a continuous semimartingale on $\mathcal{B}^{(0)}$ given by $(\ref{CSM})$, $N^{l}$ is a (simple) point process adopted to $\mathbf{F}^{(0)}$, and $(\gamma_k^{l})_{k\in\mathbb{N}}$ is a sequence of non-zero random variables.
\end{enumerate}
Moreover, we impose the following condition on the threshold processes:
\begin{enumerate}
\item[{[$\mathrm{T}$]}] $\xi'>1/2$, and for each $l=1,2$ we have $\varrho^{l}_n(t)=\alpha_n^{l}(t)\rho_n$, where
\begin{enumerate}[(i)]
\item $(\rho_n)_{n\in\mathbb{N}}$ is a sequence of (deterministic) positive numbers satisfying $\rho_n\rightarrow 0$ and
\begin{equation}\label{threshold2}
\frac{b_n^{\xi'-1/2}|\log b_n|}{\rho_n}\rightarrow 0 
\end{equation}
as $n\rightarrow\infty$.
\item $(\alpha_n^{l}(t))_{n\in\mathbb{N}}$ is a sequence of (not necessarily adapted) positive-valued stochastic processes. Moreover, there exists a sequence $(R_k^{l})$ of stopping times (with respect to $\mathbf{F}$) such that $R^{l}_k\uparrow\infty$ and both of the sequences  
$\left(\sup_{0\leq t< R_k^{l}}\alpha_n^{l}(t)\right)_{n\in\mathbb{N}}$
and
$\left(\sup_{0\leq t< R_k^{l}}[1/\alpha_n^{l}(t)]\right)_{n\in\mathbb{N}}$
are tight for all $k$. 
\end{enumerate}
\end{enumerate}
Then we obtain the following theorem.

%%%%%%%%%%%%%%%%%%%%%%%%%%%%%%%%%%%%%%%%%%%%%%%%%%%%%%%%%%%%%%%%%%%
%                  thmFAnoise
%%%%%%%%%%%%%%%%%%%%%%%%%%%%%%%%%%%%%%%%%%%%%%%%%%%%%%%%%%%%%%%%%%%

\begin{theorem}\label{thmFAnoise}
Suppose $[\mathrm{F}]$, $[\mathrm{T}]$, $[\mathrm{C}1]$-$[\mathrm{C}2]$ and $[\mathrm{N}^\flat_r]$ hold for some $r\in(2,\infty)$. Then we have
\begin{equation*}
\sup_{0\leq s \leq t}|\widehat{PTHY}(\mathsf{Z}^{1},\mathsf{Z}^{2})^n_s-\widehat{PHY}(\mathsf{X}^{1},\mathsf{X}^{2})^n_s|
=o_p(b_n^{1/4})+O_p\left(\left(b_n^{1/2}\rho_n^{-1}\right)^{\frac{r-2}{2}}\right)
\end{equation*}
as $n\rightarrow\infty$ for any $t>0$.
\end{theorem}

Proof of this theorem is given in Appendix \ref{proofthmFAnoise}. Combining this result with Theorem \ref{consistency} or Theorem \ref{AMN}, we obtain the following results:

\begin{theorem}[Consistency of the PTHY estimator in finite activity case]
Suppose $[\mathrm{F}]$, $[\mathrm{T}]$, $[\mathrm{C}1]$-$[\mathrm{C}2]$ and $[\mathrm{N}^\flat_r]$ hold for some $r\in(2,\infty)$. Then we have
\begin{equation}\label{pthycon}
\widehat{PTHY}(\mathsf{Z}^{1}, \mathsf{Z}^{2})^n\xrightarrow{ucp} [X^{1}, X^{2}]
\end{equation}
as $n\rightarrow\infty$.
\end{theorem}

\begin{theorem}[Asymptotic mixed normality of the PTHY estimator in finite activity case]\mbox{}
\begin{enumerate}[\normalfont (a)]
\item Suppose $[\mathrm{A}1'](\mathrm{i})$-$(\mathrm{iii})$, $[\mathrm{A}2]$-$[\mathrm{A}6]$ and $[\mathrm{F}]$ are satisfied. Suppose also $\underline{X}^1=\underline{X}^2=0$, $[\mathrm{N}_r]$ holds for some $r\in[8,\infty)$ and $[\mathrm{T}]$ holds with $b_n^{-(r-3)/2(r-2)}\rho_n\rightarrow\infty$ as $n\rightarrow\infty$. Then
\begin{equation}\label{pthyAMN}
b_n^{-1/4}\{\widehat{PTHY}(\mathsf{Z}^1,\mathsf{Z}^2)^n-[X^1,X^2]\}\to^{d_s}\int_0^\cdot w_s\mathrm{d}\widetilde{W}_s\qquad\mathrm{in}\ \mathbb{D}(\mathbb{R}_+)
\end{equation}
as $n\to\infty$, where $\widetilde{W}$ is the same one in Theorem \ref{AMN} and $w$ is given by $(\ref{avar})$.

%\begin{enumerate}
%\item[$(\mathrm{b})$]
\item Suppose $[\mathrm{A}1']$, $[\mathrm{A}2]$-$[\mathrm{A}6]$ and $[\mathrm{F}]$ are satisfied. Suppose also $[\mathrm{N}_r]$ holds for some $r\in[8,\infty)$ and $[\mathrm{T}]$ holds with  $b_n^{-(r-3)/2(r-2)}\rho_n\rightarrow\infty$ as $n\rightarrow\infty$. Then $(\ref{pthyAMN})$ holds with that $\widetilde{W}$ is as in the above and $w$ is given by $(\ref{avarend})$.
\end{enumerate}
\end{theorem}

%%%%%%%%%%%%%%%%%%%%%%%%%%%%%%%%%%%%%%%%%%%%%%%%%%%%%%%%
%                 IAjump case
%%%%%%%%%%%%%%%%%%%%%%%%%%%%%%%%%%%%%%%%%%%%%%%%%%%%%%%%

\subsection{Infinite activity jump case}

Next we consider the case that the observed processes are two general semimartingales contaminated by noise. We need the following structural assumption. Let $\beta\in[0,2]$.
\begin{enumerate}
\item[{$[\mathrm{K}_\beta]$}] 
For each $l=1,2$, we have
\begin{align*}
Z^{l}=X^l+\kappa(\delta^{l})\star(\mu^{l}-\nu^{l})+\kappa'(\delta^{l})\star\mu^{l},
\end{align*}
where
\begin{enumerate}[(i)]
\item $X^l$ is a continuous semimartingale given by $(\ref{CSM})$.
\item $\mu^{l}$ is a Poisson random measure on $\mathbb{R}_+\times E^{l}$ with intensity measure $\nu^{l}(\mathrm{d}t,\mathrm{d}x)=\mathrm{d}tF^{l}(\mathrm{d}x)$, where $(E^{l},\mathcal{E}^{l})$ is a Polish space and $F^{l}$ is a $\sigma$-finite measure on $(E^{l},\mathcal{E}^{l})$.
\item $\kappa(x)=x1_{\{|x|\leq 1\}}$ and $\kappa'(x)=x-\kappa(x)$ for each $x\in\mathbb{R}$.
\item $\delta^{l}$ is a predictable map from $\Omega^{(0)}\times\mathbb{R}_+\times E^{l}$ into $\mathbb{R}$. Moreover, there are a sequence $(R_k^{l})$ of stopping times increasing to $\infty$ and a sequence $(\psi_k^{l})$ of non-negative measurable functions on $E^{l}$ such that
\begin{equation*}
\sup_{\omega^{(0)}\in\Omega^{(0)} ,t<R_k^{l}(\omega^{(0)})}|\delta^{l}(\omega, t,x)|\leq\psi_k^{l}(x)~\textrm{and}~\int_{E^{l}} 1\wedge\psi_k^{l}(x)^{\beta}\underline{F}^{l}(\mathrm{d}x)<\infty .
\end{equation*}
\item If $\beta<1$, for the process $f_t=\int_{E^{l}}\kappa(\delta^{l}(t,x))\underline{F}^{l}(\mathrm{d}x)$, there is a sequence $(\sigma_k)$ of $\mathbf{F}^{(0)}$-stopping times such that for every $k$ we have a positive constant $C_{k}$ and $\lambda_k\in(0,3/4)$ satisfying $(\ref{eqA5})$ for every $t>0$ and any bounded $\mathbf{F}^{(0)}$-stopping time $\tau$.
\end{enumerate}
\end{enumerate}

Here and below $\star$ denotes the integral (either stochastic or ordinary) with respect to a some (integer-valued) random measure; see Chapter II of \cite{JS} for details. The above type of assumption appears in a lot of literature, for example \cite{J2008}. [$\mathrm{K}_\beta$] implies that for each $l=1,2$ the generalized Blumenthal-Getoor index of $Z^{l}$ is less than $\beta$.

%%%%%%%%%%%%%%%%%%%%%%%%%%%%%%%%%%%%%%%%%%%%%%%%%%%%%%%%%%%%%%%%%%%
%                  thmIAnoise
%%%%%%%%%%%%%%%%%%%%%%%%%%%%%%%%%%%%%%%%%%%%%%%%%%%%%%%%%%%%%%%%%%%

\begin{theorem}\label{thmIAnoise}
Suppose $[\mathrm{K}_{\beta}]$ and $[\mathrm{N}^\flat_r]$ hold for some $\beta\in[0,2]$ and $r\in(2,\infty)$. Suppose also $[\mathrm{C}1]$-$[\mathrm{C}2]$, $[\mathrm{A}4]$, $[\mathrm{A}6]$ and $[\mathrm{T}]$ are satisfied. Then we have
\begin{equation*}
\sup_{0\leq s \leq t}|\widehat{PTHY}(\mathsf{Z}^{1},\mathsf{Z}^{2})^n_s-\widehat{PHY}(\mathsf{X}^{1},\mathsf{X}^{2})^n_s|
=o_p(b_n^{1/4})+O_p\left(\left(b_n^{1/2}\rho_n^{-1}\right)^{\frac{r-2}{2}}\right)+o_p(\rho_n^{1-\beta/2})
\end{equation*}
as $n\rightarrow\infty$ for any $t>0$.
\end{theorem}

Proof of this theorem is given in Appendix \ref{proofthmIAnoise}. Combining this result with Theorem \ref{consistency} or Theorem \ref{AMN}, we obtain the following results:

\begin{theorem}[Consistency of the PTHY estimator in infinite activity case]\label{conpthy}
Suppose $[\mathrm{K}_2]$ and $[\mathrm{N}^\flat_r]$ hold for some $r\in(2,\infty)$. Suppose also $[\mathrm{C}1]$-$[\mathrm{C}2]$, $[\mathrm{A}4]$, $[\mathrm{A}6]$ and $[\mathrm{T}]$ are satisfied. Then we have $(\ref{pthycon})$ as $n\rightarrow\infty$.
\end{theorem}

\begin{theorem}[Asymptotic mixed normality of the PTHY estimator in infinite activity case]\label{AMNIAnoise}\mbox{}
\begin{enumerate}[\normalfont (a)]
\item Suppose $[\mathrm{A}1'](\mathrm{i})$-$(\mathrm{iii})$ and $[\mathrm{A}2]$-$[\mathrm{A}6]$ are satisfied. Suppose also $[\mathrm{N}_r]$ holds for some $r\in[8,\infty)$, $[\mathrm{K}_\beta]$ holds for some $\beta\in[0,2-\frac{1}{2\xi'-1})$ and $[\mathrm{T}]$ holds with $b_n^{-(r-3)/2(r-2)}\rho_n\rightarrow\infty$ and $\rho_n=O(b_n^{1/2(2-\beta)})$ as $n\rightarrow\infty$. Moreover, suppose $\underline{X}^1=\underline{X}^2=0$. Then $(\ref{pthyAMN})$ holds true as $n\to\infty$ with that $\widetilde{W}$ is the same one in Theorem \ref{AMN} and $w$ is given by $(\ref{avar})$.

%\begin{enumerate}
%\item[$(\mathrm{b})$]
\item Suppose $[\mathrm{A}1']$ and $[\mathrm{A}2]$-$[\mathrm{A}6]$ are satisfied. Suppose also $[\mathrm{N}_r]$ holds for some $r\in[8,\infty)$, $[\mathrm{K}_\beta]$ holds for some $\beta\in[0,2-\frac{1}{2\xi'-1})$ and $[\mathrm{T}]$ holds with $b_n^{-(r-3)/2(r-2)}\rho_n\rightarrow\infty$ and $\rho_n=O(b_n^{1/2(2-\beta)})$ as $n\rightarrow\infty$. Then $(\ref{pthyAMN})$ holds true as $n\to\infty$ with that $\widetilde{W}$ is as in the above and $w$ is given by $(\ref{avarend})$.
\end{enumerate}
\end{theorem}

Note that the assumptions of Theorem \ref{AMNIAnoise} require at least $\beta<1$.

%% file: pthy/pthy_topics.tex
%%%%%%%%%%%%%%%%%%%%%%%%%%%%%%%%%%%%%%%%%%%%%%%%%%%%%%%%%%%%%%%%%%%%%
%                             Topics
%%%%%%%%%%%%%%%%%%%%%%%%%%%%%%%%%%%%%%%%%%%%%%%%%%%%%%%%%%%%%%%%%%%%%

\section{Some related topics for statistical application to finance}\label{topics}

\subsection{Estimation of the quadratic covariation of jump parts}

As stated in the introduction, we are interested in the estimation of the quadratic covariation of jump parts of two semimartingales $Z^1$ and $Z^2$. This is achieved by estimating the quadratic variation $[Z^1,Z^2]$ due to the formula $(\ref{qc})$ because we can estimate the integrated covariance $\langle Z^{1,c},Z^{2,c}\rangle_t$ by the PTHY estimator as investigated in the previous section. In the literature such estimators are usually given by consistent estimators for the integrate covariance in the absence of jumps. See \cite{MG2012} and \cite{PV2009SPA} for example. Following this approach, we consider the pre-averaged HY estimator and we obtain the following result.
\begin{prop}\label{propphy}
Suppose $[\mathrm{C}1]$-$[\mathrm{C}2]$, $[\mathrm{A}2]$, $[\mathrm{A}4]$, $[\mathrm{A}6]$, $[\mathrm{K}_2]$ and $[\mathrm{N}^\flat_2]$ are satisfied. Then
\begin{align*}
\widehat{PHY}(\mathsf{Z}^1,\mathsf{Z}^2)^n_t=[Z^1,Z^2]_t+O_p(b_n^{1/4})
\end{align*}
as $n\to\infty$ for any $t\in\mathbb{R}_+$.
\end{prop}
See Appendix \ref{proofpropphy} for a proof. Consequently, we obtain the following result on the issue of the estimation of the quadratic covariation of jump parts:
\begin{cor}
Suppose $[\mathrm{C}1]$-$[\mathrm{C}2]$, $[\mathrm{A}2]$, $[\mathrm{A}4]$, $[\mathrm{A}6]$, $[\mathrm{K}_2]$ and $[\mathrm{N}^\flat_r]$ for some $r\in(2,\infty)$ are satisfied. Suppose also that $[\mathrm{T}]$ holds. Then
\begin{align*}
\widehat{PHY}(\mathsf{Z}^1,\mathsf{Z}^2)^n_t-\widehat{PTHY}(\mathsf{Z}^1,\mathsf{Z}^2)^n_t\to^p\sum_{0\leq s\leq t}\Delta Z^1_s\Delta Z^2_s
\end{align*}
as $n\to\infty$ for any $t\in\mathbb{R}_+$. Furthermore, if $b_n^{-(r-3)/2(r-2)}\rho_n\rightarrow\infty$ and $\rho_n=O(b_n^{1/2(2-\beta)})$ as $n\rightarrow\infty$, then
\begin{align*}
\widehat{PHY}(\mathsf{Z}^1,\mathsf{Z}^2)^n_t-\widehat{PTHY}(\mathsf{Z}^1,\mathsf{Z}^2)^n_t=\sum_{0\leq s\leq t}\Delta Z^1_s\Delta Z^2_s+O_p(b_n^{1/4})
\end{align*}
as $n\to\infty$ for any $t\in\mathbb{R}_+$.
\end{cor}

\subsection{Autocorrelated noise}

We have so far assumed that the observation noise is not autocorrelated, conditionally on $\mathcal{F}^{(0)}$. In empirical studies of financial high-frequency data, however, there is a lot of evidence that microstructure noise is autocorrelated (see \cite{HL2006} and \cite{UO2009} for instance). In this subsection we briefly discuss the case that the observation noise is autocorrelated conditionally on $\mathcal{F}^{(0)}$ as \cite{CPV2011} did in the continuous case.

We focus on the synchronous case. That is, we assume that $S^i=T^i$ for all $i$. Note that in this case it holds that $\widehat{S}^k=\widehat{T}^k=R^k=S^k$ for all $k$. Let $(\lambda^l_u)_{u\in\mathbb{Z}_+}$ and $(\mu^l_u)_{u\in\mathbb{Z}_+}$ $(l=1,2)$ be four sequences of real numbers such that
\begin{equation}\label{weakdep}
\sum_{u=1}^{\infty}u|\lambda^l_u|<\infty
\quad\mathrm{and}\quad
\sum_{u=1}^{\infty}u|\mu^l_u|<\infty.
\end{equation}
We assume that the observation data $(\mathsf{Z}^1_{S^i})$ and $(\mathsf{Z}^2_{T^j})$ are of the form
\begin{equation}\label{depmodel}
\left.\begin{array}{l}
\displaystyle\mathsf{Z}^1_{S^i}=Z^1_{S^i}+\sum_{u=0}^i\lambda^1_u\zeta^1_{S^{i-u}}+b_n^{-1/2}\sum_{u=0}^i\mu^1_u(\underline{X}^1_{S^{i-u}}-\underline{X}^1_{S^{i-u-1}}),\\
\displaystyle\mathsf{Z}^2_{T^j}=Z^2_{T^j}+\sum_{u=0}^i\lambda^2_u\zeta^2_{T^{j-u}}+b_n^{-1/2}\sum_{u=0}^i\mu^2_u(\underline{X}^2_{T^{j-u}}-\underline{X}^2_{T^{j-u-1}}).
\end{array}\right\}
\end{equation}
In other words, the observation noise follows a kind of linear processes. Under such a situation the consistency of our estimators is still valid: 
\begin{prop}\label{depcon}
Suppose $(\ref{weakdep})$ and $(\ref{depmodel})$ are satisfied. Suppose also $[\mathrm{C}1]$-$[\mathrm{C}2]$, $[\mathrm{A}2]$, $[\mathrm{A}4]$, $[\mathrm{A}6]$, $[\mathrm{K}_2]$ and $[\mathrm{N}^\flat_2]$ are satisfied. Then
$\widehat{PHY}(\mathsf{Z}^1,\mathsf{Z}^2)^n_t\to^p[Z^1,Z^2]_t$
as $n\to\infty$ for any $t\in\mathbb{R}_+$. Furthermore, if $[\mathrm{T}]$ and $[\mathrm{N}^\flat_r]$ for some $r\in(2,\infty)$ hold, then $\widehat{PTHY}(\mathsf{Z}^1,\mathsf{Z}^2)^n_t\to^p[X^1,X^2]_t$
as $n\to\infty$ for any $t\in\mathbb{R}_+$.
\end{prop}
We give a proof of Proposition \ref{depcon} in Appendix \ref{proofdepcon}. The proof is based on a Beveridge-Nelson type  decomposition for the noise. It might be possible to prove the asymptotic mixed normality of the PTHY estimator under the above model by refining on the proof of the above proposition given in the present article. On the other hand, in the nonsynchronous case we will need to model the autocorrelation structure of the noise on the time dependence in calender time (as \cite{UO2009} did) rather than tick time (as in the above). This is because we have two axes of tick time, $(S^i)$ and  $(T^j)$, in the nonsynchronous case and this fact complicates the analysis of our estimator. With an appropriate modeling of the autocorrelation structure of the noise, we can probably obtain a Beveridge-Nelson type  decomposition for the noise even in the nonsynchronous case. As a result, we might be able to prove the consistency and the asymptotic normality of our estimator. However, these topics are beyond the scope of this paper, so that we postpone them to further research.

%% file: pthy/pthy_anova_mod.tex
%%%%%%%%%%%%%%%%%%%%%%%%%%%%%%%%%%%%%%%%%%%%%%%%%%%%%%%%%%%%%%%%%%%
%                Studentization
%%%%%%%%%%%%%%%%%%%%%%%%%%%%%%%%%%%%%%%%%%%%%%%%%%%%%%%%%%%%%%%%%%%

\subsection{Estimation of asymptotic variance}

In this subsection we shall briefly discuss the estimation of the asymptotic variance of the PTHY estimator. This is necessary to construct feasible confidence intervals of this estimator, for example. We focus on the simple case that the endogenous terms of the microstructure noise are absent, i.e. $\underline{X}^1=\underline{X}^2=0$. In this case our aim can be achieved by a \textit{kernel-based approach} as in \cite{HY2011} and \cite{Koike2012}.

More precisely, let $(h_n)$ be a sequence of positive numbers tending to 0 as $n\rightarrow\infty$. For any $s\in\mathbb{R}_+$, put
\begin{align*}
\widehat{[X^{l},X^{l'}]'}_s=h_n^{-1}\left(\widehat{PTHY}(\mathsf{Z}^{l},\mathsf{Z}^{l'})^n_s-\widehat{PTHY}(\mathsf{Z}^{l}, \mathsf{Z}^{l'})^n_{(s-h_n)_+}\right),\qquad \widehat{[X^{l}]'}_s=\widehat{[X^{l},X^{l}]'}_s,\qquad l,l'=1,2
\end{align*}
and
\begin{align*}
&\widehat{\Psi^{11}}_s=-\frac{1}{h_n k_n^2}\sum_{i:s-h_n<S^{i+1}\leq s}(\mathsf{Z}^1_{S^i}-\mathsf{Z}^1_{S^{i-1}})(\mathsf{Z}^1_{S^{i+1}}-\mathsf{Z}^1_{S^{i}}),\\
&\widehat{\Psi^{22}}_s=-\frac{1}{h_n k_n^2}\sum_{j:s-h_n<T^{j+1}\leq s}(\mathsf{Z}^2_{T^j}-\mathsf{Z}^2_{T^{j-1}})(\mathsf{Z}^2_{T^{j+1}}-\mathsf{Z}^2_{T^{j}}),\\
&\widehat{\Psi^{12}\chi}_s=-\frac{1}{2 h_n k_n^2}\sum_{k:s-h_n<R^{k+1}\leq s}\left\{(\mathsf{Z}^1_{\widehat{S}^k}-\mathsf{Z}^1_{\widehat{S}^{k-1}})(\mathsf{Z}^2_{\widehat{T}^{k+1}}-\mathsf{Z}^2_{\widehat{T}^{k}})+(\mathsf{Z}^1_{\widehat{S}^{k+1}}-\mathsf{Z}^1_{\widehat{S}^{k}})(\mathsf{Z}^2_{\widehat{T}^{k}}-\mathsf{Z}^2_{\widehat{T}^{k-1}})\right\}1_{\{\widehat{S}^k=\widehat{T}^k\}}.
\end{align*}
Then we set
\begin{align}
\widehat{w^2}_{R^k}
&=k_n\psi_{HY}^{-4}\left[\kappa\left\{\widehat{[X^1]'}_{R^{k}}\widehat{[X^2]'}_{R^{k}}+\left(\widehat{[X^1,X^2]'}_{R^{k}}\right)^2\right\}
+\widetilde{\kappa}\left\{\widehat{\Psi^{11}}_{R^{k}}\widehat{\Psi^{22}}_{R^{k}}+\left(\widehat{\Psi^{12}\chi}_{R^{k}}\right)^2\right\}\right.\nonumber\\
&\hphantom{=k_n\psi_{HY}^{-4}[}\left.+\overline{\kappa}\left\{\widehat{[X^1]'}_{R^{k}}\widehat{\Psi^{22}}_{R^{k}}+\widehat{[X^2]'}_{R^{k}}\widehat{\Psi^{11}}_{R^{k}}+2\widehat{[X^1,X^2]'}_{R^{k}}\widehat{\Psi^{12}\chi}_{R^{k}}\right\}\right]|\Gamma^{k}||\Gamma^{k+1}|\label{what}
\end{align}
for every $k\in\mathbb{N}$ and $\widehat{\int_0^t w_s^2\mathrm{d}s}=b_n^{-1/2}\sum_{k:R^k\leq t}\widehat{w^2}_{R^k}$ for every $t\in\mathbb{R}_+$. 
\begin{prop}
Suppose $[\mathrm{A}1'](\mathrm{i})$-$(\mathrm{iii})$ and $[\mathrm{A}2]$-$[\mathrm{A}6]$ are satisfied. Suppose also $[\mathrm{N}_r]$ holds for some $r\in[8,\infty)$, $[\mathrm{K}_\beta]$ holds for some $\beta\in[0,2-\frac{1}{2\xi'-1})$ and $[\mathrm{T}]$ holds with $b_n^{-(r-3)/2(r-2)}\rho_n\rightarrow\infty$ and $\rho_n=O(b_n^{1/2(2-\beta)})$ as $n\rightarrow\infty$. Moreover, suppose $\underline{X}^1=\underline{X}^2=0$. Then
\begin{equation*}
\widehat{\int_0^{\cdot} w_s^2\mathrm{d}s}\xrightarrow{ucp}\int_0^{\cdot} w_s^2\mathrm{d}s
\end{equation*}
as $n\rightarrow\infty$, provided that $h_n^{-1}b_n^{1/4}\rightarrow 0$ and $\sup_{0\leq t\leq T}h_n^{-1}b_n(N^n_t-N^n_{(t-h_n)_+})$ is tight as $n\to\infty$ for any $T>0$.
\end{prop}

\begin{proof}
Since $\int_0^\cdot w_s^2\mathrm{d}s$ is a continuous non-decreasing process, it is sufficient to prove the pointwise convergence. By Lemma 10.2(a) of \cite{Koike2012phy}, we have $b_n^{-1/2}\widehat{\int_0^{t} w_s^2\mathrm{d}s}-b_n^{1/2}\sum_{k:R^k\leq t}\widetilde{w^2}_{R^k}$ as $n\to\infty$ for every $t$, where $\widetilde{w^2}_{R^k}$ is defined by $(\ref{what})$ with replacing $|\Gamma^{k+1}|$ by $G^n_{R^k}$. Then, we obtain the desired result by the assumptions and the dominated convergence theorem.
\end{proof}

The above approach has the disadvantage that it depends strongly on the particular form of the asymptotic variance due to the noise. In fact, it is not adapted to the case that the endogenous terms of the noise is present because we have so far known no estimator for the statistic $\int_0^t\left([\underline{X}^1,X^2]'_s F^1_s-[X^1,\underline{X}^2]'_s F^2_s\right)^2G_s^{-1}\mathrm{d}s$ which is the asymptotic variance due to the presence of the endogenous noise. We will also need to modify it if the noise is autocorrelated since we will need to replace the covariance matrix $\Psi$ of the noise in the asymptotic variance with the long-run covariance matrix of the noise (see the proof of Proposition \ref{depcon} in Appendix \ref{proofdepcon}). To avoid this problem, we might rely on the approach used in Section 4 of \citet{CPV2011} or the \textit{subsampling approach} developed by \citet{Kal2011} recently though it remains for further research to verify the theoretical validity of them.

%% file: pthy/pthy_sim_3.tex
%%%%%%%%%%%%%%%%%%%%%%%%%%%%%%%%%%%%%%%%%%%%%%%%%%%%%%%%%%%%%%%%%%%%%
%                 simulation
%%%%%%%%%%%%%%%%%%%%%%%%%%%%%%%%%%%%%%%%%%%%%%%%%%%%%%%%%%%%%%%%%%%%%%

\section{Simulation study}\label{simulation}
In this section, we examine the finite sample performance of our estimators by using Monte Carlo experiments.

\subsection{Choice of the threshold processes}

As is well known, the thresholding method is often sensitive to the selection of thresholds in finite samples; see \cite{S2010} or the Web Appendix of \cite{MG2012} for instance. Therefore, it is important to determine a reasonable rule of selecting thresholds. Here we present an easy but effective way to determine thresholds. Formal study of methods for optimal threshold selection in a given model is an important issue for the future.

We will determine the thresholds for individual processes so that we focus on the univariate case. First we compute an auxiliary estimator $\widehat{\Sigma}^n_t$ for the spot variance process $\Sigma_t=\theta\psi_2[X^1]'_t+\frac{1}{\theta}\psi_1\Psi^{11}_t$ for each sampling time $t$, where $\psi_1=\int_0^1g'(s)^2\mathrm{d}s$ and $\psi_2=\int_0^1g(s)^2\mathrm{d}s$. In this paper we will use a numerical derivative of the pre-averaged bipower variation, i.e.
\begin{align*}
\widehat{\Sigma}_{\widehat{S}^i}=\frac{\mu_1^{-2}}{K-2k_n+1}\sum_{p=i-K}^{i-2 k_n}|\overline{\mathsf{Z}}^{1}(\widehat{\mathcal{I}})^p||\overline{\mathsf{Z}}^{1}(\widehat{\mathcal{I}})^{p+k_n}|,\qquad i=K,K+1,\dots,N
\end{align*}
and $\widehat{\Sigma}_{\widehat{S}^i}=\widehat{\Sigma}_{\widehat{S}^K}$ if $i<K$. Here, $\mu_1$ is the absolute moment of the standard normal distribution, $N$ is the number of the available pre-averaging data $(\overline{\mathsf{Z}}^{1}(\widehat{\mathcal{I}})^i)$ and $K$ is a bandwidth parameter such that $K=O(b_n^{-\alpha})$ as $n\rightarrow\infty$ for some $\alpha\in(0.5,1)$. We will set $K=\lceil N^{3/4}\rceil$ below. Such a kind of spot variance estimator was studied in \citet{BJK2012}. Then we choose
\begin{equation}\label{PLUT}
\varrho^1_n(\widehat{S}^i)=2\log(N)^{1+\varepsilon}\widehat{\Sigma}_{\widehat{S}^i},\qquad i=0,1,\dots,N
\end{equation}
for some $\varepsilon>0$. We will set $\varepsilon=0.2$ below.

The heuristic idea behind the above choice of thresholds is as follows. First we recall the following classic result:
\begin{theorem}[\protect\citet{Pickands1967}, Theorem 3.4]\label{univ}
Let $(X_i)_{i\in\mathbb{N}}$ be a stationary Gaussian process such that $E[X_i]=0$, $E[X_i^2]=1$ and $E[X_iX_{i+k}]=\gamma(k)$. If $\lim_{k\to\infty}\gamma(k)=0$, then $\max_{1\leq i\leq n}X_i/\sqrt{2\log n}\to1,$ almost surely, as $n\to\infty$.
\end{theorem}
The most important point of the above theorem is that the random variables $X_i$, $i=1,2,\dots,$ in the theorem can have a kind of dependence structure. This fact is crucial for the present situation because the pre-averaging data $(\overline{\mathsf{Z}}^{1}(\widehat{\mathcal{I}})^i)$ is $k_n$-dependent. As a result, Theorem \ref{univ} has the following implication: Suppose that the observation data is given by a scaled Brownian motion with i.i.d.~Gaussian noise. That is, suppose that $\mathsf{Z}_{S^i}=\sigma W_{S^i}+u_i$, $i=0,1,\dots,$ where $\sigma>0$, $W_t$ is a standard Wiener process and $(u_i)$ is an i.i.d.~random variables independent of $W$ with $u_i\sim N(0,\omega^2)$. Suppose also that $(S^i)$ is an equidistant sampling scheme. Then the pre-averaging data $(\overline{\mathsf{Z}}^{1}(\widehat{\mathcal{I}})^i)$ is a centered stationary Gaussian process with the autocovariance function vanishing at infinity, so that Theorem \ref{univ} yields
\begin{equation*}
\max_{0\leq i\leq N-1}\overline{\mathsf{Z}}^{1}(\widehat{\mathcal{I}})^i\Big/\sqrt{\Sigma\cdot 2\log N}\to1\qquad\mathrm{a.s.}
\end{equation*}
as $n\to\infty$, where $\Sigma=\theta\psi_2\sigma^2+\frac{1}{\theta}\omega^2$ is the variance of $\overline{\mathsf{Z}}^{1}(\widehat{\mathcal{I}})^i$. This result suggests that we may use $\Sigma\cdot 2\log N$ as thresholds. This idea has already been introduced as the \textit{universal threshold} by \cite{DJ1994} in the context of wavelet shrinkage. \cite{DJ1994} estimated the unknown parameter $\Sigma$ by the square of the median absolute deviation (MAD) of $(\overline{\mathsf{Z}}^{1}(\widehat{\mathcal{I}})^i)$ divided by 0.6745, the 0.75-quantile of the standard normal distribution. In the present situation $(\overline{\mathsf{Z}}^{1}(\widehat{\mathcal{I}})^i)$ is heteroscedastic in general, hence we need to replace $\Sigma$ with the spot variance process $\Sigma_t$ and estimate $\Sigma_t$ by $\widehat{\Sigma}^n_t$. Consequently, we have arrived at the threshold process given by $(\ref{PLUT})$, where we multiply the usual universal threshold by $(\log N)^\varepsilon$ to ensure the condition $(\ref{threshold2})$. 

Since the threshold process $(\ref{PLUT})$ can be regarded as a pre-averaging version of the local universal threshold proposed in \cite{Koike2012}, so we call $(\ref{PLUT})$ the \textit{pre-averaged local universal threshold} (abbreviated PLUT). The author conjectures that the PLUT satisfies the condition [T] under some mild regularity conditions though we do not investigate this topic in this paper and postpone it in the future. Instead, we will numerically show that the PLUT gives a reasonable choice of the threshold processes below.  

\subsection{Simulation design}

We simulate over the interval $t\in[0,1]$. We normalize one second to be $1/23400$, so that the interval $[0,1]$ contains $6.5$ hours. In generating the observation data, we discretize $[0,1]$ into a number $n=23400$ of intervals. 

In order to extract irregular, nonsynchronous observation times from $n$ equi-spaced division points, we generate random observation times $(S^i)$ and $(T^j)$ using two independent Poisson processes with intensity $n/\lambda_1$ and $n/\lambda_2$. Here $\lambda_l$ denotes the average waiting time for new data from process $\mathsf{Z}^l$, so that a typical simulation will have $n/\lambda_l$ observations of $\mathsf{Z}^l$, $l=1,2$. Following \cite{BNHLS2011}, we vary $\lambda:=(\lambda_1,\lambda_2)$ through the following configurations (3,6), (10,20) and (30,60). Note that because we are simulating in discrete time, it is possible to see common points to the observation times $(S^i)$ and $(T^j)$.

We consider three types of bivariate L\'evy processes $J_t$ with no Brownian components to introduce jumps to models. The specifics of the jump processes are as follows:

\begin{description}
\item[NO] $J\equiv0$, i.e. there are no jumps.

\item[SCP1] Let $L$ be a stratified normal inverse Gaussian compound Poisson
process with a single jump per unit time (i.e., the jump time is uniformly
distributed over $[0,1]$ and the jump size follows a normal inverse Gaussian
distribution). The jump size is drawn from $\varepsilon\sqrt{S}$, where $\varepsilon\indepe S$, $\varepsilon\sim N(0,1)$ and $S\sim IG(c,c^2/\gamma)$, so that $\mathrm{Var}[\varepsilon\sqrt{S}]=E[S]=c$ and $\mathrm{Var}[S]=c^3/(c^2/\gamma)=c\gamma$. Then, we set $J^1=J^2=L$.
\if0
and
\begin{equation*}
\Lambda=
\left(\begin{array}{cc}
1&\sqrt{1-(\rho^1)^2}\sqrt{1-(\rho^2)^2}\\
\sqrt{1-(\rho^1)^2}\sqrt{1-(\rho^2)^2}&1
\end{array}\right).
\end{equation*}
\fi

%\item[SCP1(I)] For each $l=1,2$ $J^l$ is a stratified normal inverse Gaussian compound Poisson process with a single jump per unit time and $J^1$ and $J^2$ are mutually independent. For each $l$ the jump size is drawn from $\varepsilon\sqrt{S}$, where $\varepsilon\indepe S$, $\varepsilon\sim N(0,1)$, and $S\sim IG(c,c^2/\gamma)$. In the simulation, we set $\gamma=0.25$.

\item[VG] Let $L^1$ and $L^2$ be mutually independent variance Gamma processes such that $L^l_1\sim\varepsilon\sqrt{S}$, where $\varepsilon\indepe S$, $\varepsilon\sim N(0,1)$, and $S\sim \Gamma(c/\gamma,1/\gamma)$, so that
$\mathrm{Var}[\varepsilon^l\sqrt{S}]=E[S]=c$ for each $l=1,2$ and $\mathrm{Var}[S]=(c/\gamma)/(1/\gamma)^2=c\gamma$. Then, we set $J^1=L^1$ and $J^2=R L^1+\sqrt{1-R^2}L^2$.

\end{description}
In the simulation we set $c=0.1$ and $\gamma=0.25$. The value of $R$ is given for each model below. Note that each component of the above models coincides with the model simulated in \citet{Veraart2010}.

The observation data $(\mathsf{Z}^1_{S^i})$ and $(\mathsf{Z}^2_{T^j})$ are generated from the models below.

{\bf Model 1} (\citet{BNHLS2011}) --- the case of stochastic volatility \& additive noise. The following bivariate factor stochastic volatility model is used to generate the continuous semimartingales $X^1$ and $X^2$:
\begin{align*}
\mathrm{d}X^l_t=\mu^l\mathrm{d}t+\rho^l\sigma^l_t B^k_t+\sqrt{1-(\rho^l)^2}\sigma^l_t W_t,\qquad
\sigma^l_t=\exp(\beta^l_0+\beta^l_1\varrho^l_t),\qquad
\mathrm{d}\varrho^l_t=\alpha^l\varrho^l_t\mathrm{d}t+\mathrm{d}B^l_t,\qquad l=1,2,
\end{align*}
where $(B^1,B^2,W)$ is a 3-dimensional standard Wiener processes. The initial values for the $\varrho^l_t$ processes at each simulation run are drawn randomly from their stationary distribution, which is $\varrho^l_t\sim N(0,(-2\alpha^l)^{-1})$. We carry out our numerical experiments by using the following parametrization, assumed to be identical across the two volatility factors: $(\mu^l,\beta^l_0,\beta^l_1,\alpha^l,\rho^l)=(0.03,-5/16,1/8,-1/40,-0.3)$, so that $\beta_0^l=(\beta^l_1)^2/2\alpha^l$. This
choice of parameters implies that integrated volatility has been normalized, in the sense that $E[\int_0^1(\sigma^l_s)^2\mathrm{d}s]=1$. At each simulation run we add noise $(U^l_{k/n})_{k=0}^n$ simulated as
\begin{align*}
U^l_{k/n}|\{\sigma,X,J\}\overset{i.i.d.}{\sim}N(0,\omega^2),\qquad\omega^2=\eta^2\sqrt{\frac{1}{n}\sum_{i=1}^n\left(\sigma^l_{i/n}\right)^4},\quad\mathrm{and}\quad
\mathrm{Corr}(U^1_t,U^2_s)=
\left\{\begin{array}{cl}
R&\textrm{if }t=s,\\
0&\textrm{if }t\neq s,
\end{array}\right.
\end{align*}
where $R=\sqrt{1-(\rho^1)^2}\sqrt{1-(\rho^2)^2}$ and the noise-to-signal ratio, $\eta^2$ takes the value 0.001. Finally,  $\mathsf{Z}^1_{S^i}$ and $\mathsf{Z}^2_{T^j}$ are given by $\mathsf{Z}^1_{S^i}=Z^1_{S^i}+U^1_{S^i}$ and $\mathsf{Z}^2_{T^j}=Z^2_{T^j}+U^2_{T^j}$, where $Z^l=X^l+J^l$, $l=1,2$.

{\bf Model 2} (\citet{JLMPV2009}) --- the case of constant volatility \& rounding plus error.
\begin{gather*}
Z^l_t=X^l_t+\sigma^l J^l_t,\qquad X^l_t=X^l_0+\sigma^l W^l_t,\\
\mathsf{Z}^l_{k/n}=\log\left(\alpha^l\left\lfloor\frac{\exp(Z^l_{k/n}+u^l_{k/n})}{\alpha^l}\right\rfloor\right),\qquad
u^l_{k/n}=\eta^l_{k/n}\log\frac{\alpha^l\lceil\frac{\exp(Z^l_{k/n})}{\alpha^l}\rceil}{\exp(Z^l_{k/n})},
\end{gather*}
where $W^1$ and $W^2$ are correlated standard Winer processes independent of $J$ with $\mathrm{d}[W^1,W^2]_t=R\mathrm{d}t$ and $(\eta^l_{k/n})$ is a sequence of independent Bernoulli variables (probabilities $p^l_{k/n}$ and $1-p^l_{k/n}$ of taking values 1 and 0), with
$p^l_{k/n}=\log\left(\frac{\exp(Z^l_{k/n})}{\alpha^l\lfloor\frac{\exp(Z^l_{k/n})}{\alpha^l}\rfloor}\right)\left/\log\left(\frac{\alpha^l\lceil\frac{\exp(Z^l_{k/n})}{\alpha^l}\rceil}{\alpha^l\lfloor\frac{\exp(Z^l_{k/n})}{\alpha^l}\rfloor}\right)\right..$
We assume that the sequences $(\eta^1_{k/n})$ and $(\eta^2_{k/n})$ are mutually independent as well as independent of $W$ and $J$. Parameters used: $\sigma^l=0.2/\sqrt{252}$, $X^l_0=\log(8)$, $\alpha^l=0.01$ and $R=0.5$.

{\bf Model 3} --- the case of stochastic volatility \& endogenous noise. The model of the continuous semimartingales $X^1$ and $X^2$ is the same one as in Model 1, but the noise processes $U^1_{S^i}$ and $U^2_{T^j}$ are given by
\begin{align*}
U^1_{S^i}=\delta^1\sqrt{n/\lambda^1}(X^1_{S^i}-X^1_{S^{i-1}}),\qquad
U^2_{T^j}=\delta^2\sqrt{n/\lambda^2}(X^2_{T^j}-X^2_{T^{j-1}}).
\end{align*}
Here we set $\delta^1=\delta^2=-0.01$, so that the microstructure noise is negatively correlated with the returns of the latent continuous semimartingale processes $X^1$ and $X^2$. This choice reflects the empirical findings reported in \citet{HL2006}. Note that the magnitude of the noise processes in this model is smaller than the one in Model 1. Finally, as in Model 1 we set  $\mathsf{Z}^1_{S^i}=Z^1_{S^i}+U^1_{S^i}$ and $\mathsf{Z}^2_{T^j}=Z^2_{T^j}+U^2_{T^j}$, where $Z^l=X^l+J^l$, $l=1,2$.

1000 iterations were run for each model. The simulation of the model paths of $X^1$ and $X^2$ has been made using the Euler-Maruyama scheme with $n$ equi-spaced division points.

The estimators are calculated componentwise. That is, we use all of the observations for the variance estimators and the refresh time sampling of the observations for the covariance estimator. The tuning parameters for pre-averaging are selected as follows. We use $\theta=0.15$ and $g(x)=x\wedge(1-x)$ following \citet{CPV2011}, and set $k_n=\lceil\theta\sqrt{m}\rceil$. Here, $m$ represents the number of the observed returns when calculating the variance estimators while the number of refresh times minus 1 when calculating the covariance estimators. 

\subsection{Simulation results}

In Table \ref{ICmodel1}--\ref{ICmodel3}, we present the bias and the root mean squared error of our PTHY estimator in Model 1--3 respectively. As a comparison, we also computed the subsampled realized bipower (co)variations (BPVs) based on 5-minutes returns. More precisely, we computed a total of 300 realized bipower (co)variations by shifting the time of the first observation in 1-second increments. Then, we took the average of these estimators. Such an operation is commonly used in empirical work in this area; see \cite{BNHLS2009} and \cite{NSS2011} for example. Note that we divide the reported numbers in Table \ref{ICmodel2} by $(0.2/\sqrt{252})^2$ for normalization. We can see that our estimator performs well even in finite samples especially when the frequency of the observations are relatively large. When the frequency of the observations is small, it is downward biased due to the loss of summands induced by pre-averaging and the nonsynchronisty of the observation times. It is not surprising, because reducing the effects induced by microstructure noise forces the estimation to be less efficient. In contrast, the variance estimation by the BPV statistic is significantly upward biased across all scenarios for Model 1--2, while downward biased across almost all scenarios for Model 3. Although for Model 1--2 the covariance estimator of the BPV is surprisingly more precise than the variance estimator, our PTHY estimator is superior to it in most of the scenarios.  

In Table \ref{JVmodel1}--\ref{JVmodel3}, we present the bias and the root mean squared error of estimators for the quadratic (co)variations $JV^{k,l}:=[Z^k,Z^l]_1-[X^k,X^l]_1$ $(k,l=1,2)$ of the jump processes (JV) in Model 1--3 respectively. We also report the results for the estimators based on the differences between the subsampled realized (co)variances (RV) based on 5-minutes returns and the BPVs for a comparison. The reported numbers in Table \ref{ICmodel2} are divided by $(0.2/\sqrt{252})^2$ for normalization as above. As the tables reveal, our estimator is downward biased in the presence of jumps when the frequency of the observations is large and such a downward bias tends to be large in the VG case. It is not surprising because the thresholding technique cannot detect too small jumps whose sizes have the same magnitude as those of Brownian increments (note that in finite samples we cannot identify very small jumps \textit{in principle}; see \citet{Zhang2007} for details). Although the bias is modest compared with that of the BPV based statistic (when the sampling frequency is sufficiently large), we might need to investigate the possibilities of making finite sample adjustments. We also find that both of the estimators are upward biased in the absence of jumps. It is theoretically natural because these estimators should be non-negative asymptotically (note that in the simulation we always have $JV^{1,2}\geq0$). In particular, our estimator is far more precise than the BPV based statistic if the sampling frequency is large. However, our estimator does not perform well when the frequency of the observations is very small. In this case our estimator is inferior to the BPV based statistic in some situations. It is also worth mentioning that in terms of the bias our estimator has slightly worth performance in Model 3 than Model 1, despite the fact that the magnitude of the noise in Model 3 is smaller than that in Model 1. 

Finally, we briefly compare our simulation results with some existing empirical studies in this area. Recently \citet{COP2011} indicated low-frequency based measures of jump variations such as the above BPV based statistic tend to be upward biased. In our simulation, the BPV based estimators for the quadratic variations of the jump processes are upward biased in the absence of jumps. They are also upward biased across all scenarios in Model 2. These findings have the following implication. First, in real markets jumps might not often occur (or they might be too active to be disentangled from diffusive components as indicated in \cite{AJL2011}) and perhaps their magnitude is not so large. Second, the round-off effects could be significantly important for measuring jump variations. This is of course quite natural intuitively. On the other hand, in Table 3 of \cite{COP2011} we can find that the reported values of their low-frequency based bipower variations are smaller than those of their tick frequency based estimators for almost all assets, so that we expect the low-frequency based bipower variations will be downward biased. We can observe such phenomenons (only) across all of the scenarios in Model 3. This is not surprising because the downward bias of the BPV based estimators for integrated variances presumably appears only in the presence of the negative correlation between  the noise and the efficient returns.  

%In Table \ref{RJmodel1} and \ref{RJmodel2}, we present the bias and root mean squared error of our estimators $$1-\frac{\widehat{PTHY}(\mathsf{Z}^k,\mathsf{Z}^l)^n_1}{\widehat{PHY}(\mathsf{Z}^k,\mathsf{Z}^l)^n_1}$$ for the quantities $\mathrm{RJ}^{kl}=1-[X^k,X^l]_1/[Z^k,Z^l]_1$ $(k,l=1,2)$. The reported values are expressed in percent. As a comparison, we also computed the estimators based on the realized (co)variance (RV) and the BPV. RV was computed based on five minute returns as above.  

{\linespread{1.2}
\begin{table}
\caption{Simulation results of the estimation of integrated (co)variances in Model 1}
\begin{center}
\begin{tabular}{lccccccc}\hline
 &  & PTHY &  &&  & BPV &  \\[1pt] \cline{2-4}\cline{6-8}
Target & $[X^1]_1$ & $[X^1,X^2]_1$ & $[X^2]_1$ && $[X^1]_1$ & $[X^1,X^2]_1$ & $[X^2]_1$ \\[1pt] \hline
NO &  &  &  &&  &  &  \\
$\lambda=(3,6)$ & $-$.003 (.136) & $-$.002 (.104) & $-$.004 (.241) && .132 (.291) & $-$.020 (.136) & .131 (.379) \\
$\lambda=(10,20)$ & $-$.005 (.196) & $-$.011 (.135) & $-$.022 (.300) && .131 (.291) & $-$.058 (.165) & .122 (.399) \\
$\lambda=(30,60)$ & $-$.028 (.256) & $-$.037 (.200) & $-$.054 (.425) && .120 (.291) & $-$.142 (.246) & .092 (.412) \\[5pt]
SCP1 &  & &  &&  &  & \\
$\lambda=(3,6)$ & .004 (.137) & .003 (.104) & .002 (.237) && .164 (.318) & .012 (.148) & .164 (.399) \\
$\lambda=(10,20)$ & .005 (.198) & $-$.006 (.137) & $-$.014 (.298) && .163 (.315) & $-$.028 (.166) & .155 (.418) \\
$\lambda=(30,60)$ & $-$.018 (.259) & $-$.033 (.203) & $-$.044 (.421) && .151 (.315) & $-$.116 (.236) & .123 (.428) \\[5pt]
VG &  &  & &&  &  &  \\
$\lambda=(3,6)$ & .005 (.136) & .005 (.103) &  .011 (.239) && .171 (.328) & .017 (.150) & .177 (.412) \\
$\lambda=(10,20)$ & .006 (.200) & $-$.003 (.137) &  $-$.002 (.303) && .170 (.326) & $-$.022 (.165) & .168 (.430) \\
$\lambda=(30,60)$ & $-$.014 (.256) & $-$.030 (.200) &  $-$.033 (.427) && .158 (.322) & $-$.111 (.235) & .136 (.438) \\ \hline
\end{tabular}
\label{ICmodel1}\vspace{2mm}

\parbox{17cm}
{\footnotesize
\textit{Note}. We report the bias and rmse of the estimators for 
the integrated (co)variances included in the simulation study. The number reported in parenthesis is rmse.
}

\end{center}
\end{table}
}

{\linespread{1.2}
\begin{table}
\caption{Simulation results of the estimation of integrated (co)variances in Model 2}
\begin{center}
\begin{tabular}{lccccccc}\hline
 &  & PTHY &  &&  & BPV &  \\[1pt] \cline{2-4}\cline{6-8}
Target & $[X^1]_1$ & $[X^1,X^2]_1$ & $[X^2]_1$ && $[X^1]_1$ & $[X^1,X^2]_1$ & $[X^2]_1$ \\[1pt] \hline
NO &  &  &  &&  &  &  \\
$\lambda=(3,6)$ & $-$.000 (.090) & .003 (.084) & .002 (.010) && .111 (.176) & $-$.021 (.109) & .110 (.177) \\
$\lambda=(10,20)$ & .001 (.118) & $-$.007 (.114) & $-$.011 (.142) && .109 (.115) & $-$.043 (.119) & .106 (.181) \\
$\lambda=(30,60)$ & $-$.021 (.159) & $-$.036 (.161) & $-$.057 (.186) && .101 (.180) & $-$.105 (.156) & .073 (.181) \\[5pt]
SCP1 &  & &  &&  &  & \\
$\lambda=(3,6)$ & .008 (.089) & .011 (.085) & .010 (.103) && .150 (.214) & .013 (.122) & .149 (.213) \\
$\lambda=(10,20)$ & .014 (.118) & .005 (.115) & .001 (.144) && .148 (.215) & $-$.011 (.124) & .145 (.218) \\
$\lambda=(30,60)$ & $-$.006 (.160) & $-$.025 (.161) & $-$.045 (.188) && .140 (.218) & $-$.076 (.145) & .107 (.205) \\[5pt]
VG &  &  & &&  &  &  \\
$\lambda=(3,6)$ & .010 (.090) & .002 (.086) &  .016 (.109) && .154 (.222) & .000 (.118) & .158 (.220) \\
$\lambda=(10,20)$ & .004 (.118) & $-$.003 (.118) &  .008 (.118) && .152 (.224) & $-$.019 (.121) & .153 (.222) \\
$\lambda=(30,60)$ & $-$.010 (.159) & $-$.038 (.159) &  $-$.036 (.192) && .143 (.226) & $-$.086 (.152) & .119 (.212) \\ \hline
\end{tabular}
\label{ICmodel2}\vspace{2mm}

\parbox{17cm}{\footnotesize
\textit{Note}. We report the bias and rmse of the estimators for 
the integrated (co)variances included in the simulation study. The number reported in parenthesis is rmse. All of the reported numbers are divided by $(0.2/\sqrt{252})^2$.
}

\end{center}
\end{table}
}

{\linespread{1.2}
\begin{table}
\caption{Simulation results of the estimation of integrated (co)variances in Model 3}
\begin{center}
\begin{tabular}{lccccccc}\hline
 &  & PTHY &  &&  & BPV &  \\[1pt] \cline{2-4}\cline{6-8}
Target & $[X^1]_1$ & $[X^1,X^2]_1$ & $[X^2]_1$ && $[X^1]_1$ & $[X^1,X^2]_1$ & $[X^2]_1$ \\[1pt] \hline
NO &  &  &  &&  &  &  \\
$\lambda=(3,6)$ & $-$.004 (.128) & $-$.002 (.100) & $-$.007 (.218) && $-$.027 (.229) & $-$.034 (.147) & $-$.045 (.370) \\
$\lambda=(10,20)$ & $-$.008 (.186) & $-$.017 (.138) & $-$.033 (.305) && $-$.042 (.236) & $-$.066 (.171) & $-$.073 (.401) \\
$\lambda=(30,60)$ & $-$.039 (.248) & $-$.055 (.198) & $-$.079 (.433) && $-$.073 (.254) & $-$.155 (.259) & $-$.144 (.460) \\[5pt]
SCP1 &  & &  &&  &  & \\
$\lambda=(3,6)$ & .$-$.000 (.127) & .000 (.101) & $-$.004 (.215) && $-$.010 (.229) & $-$.016 (.149) & $-$.026 (.370) \\
$\lambda=(10,20)$ & $-$.003 (.186) & $-$.014 (.138) & $-$.028 (.300) && $-$.012 (.237) & $-$.037 (.170) & $-$.042 (.399) \\
$\lambda=(30,60)$ & $-$.035 (.246) & $-$.053 (.199) & $-$.074 (.428) && $-$.056 (.251) & $-$.141 (.253) & $-$.127 (.455) \\[5pt]
VG &  &  & &&  &  &  \\
$\lambda=(3,6)$ & .002 (.127) & .002 (.100) &  .003 (.215) && .010 (.243) & .003 (.156) & .000 (.375) \\
$\lambda=(10,20)$ & $-$.001 (.186) & $-$.013 (.139) &  $-$.018 (.303) && $-$.005 (.246) & $-$.031 (.171) & $-$.028 (.399) \\
$\lambda=(30,60)$ & $-$.028 (.246) & $-$.051 (.196) &  $-$.062 (.435) && $-$.037 (.255) & $-$.013 (.247) & $-$.102 (.452) \\ \hline
\end{tabular}
\label{ICmodel3}\vspace{2mm}

\parbox{17cm}{\footnotesize
\textit{Note}. We report the bias and rmse of the estimators for 
the integrated (co)variances included in the simulation study. The number reported in parenthesis is rmse.
}
\end{center}
\end{table}
}

{\linespread{1.2}
\begin{table}
\caption{Simulation results of the estimation of jump (co)variations in Model 1}
\begin{center}
\begin{tabular}{lccccccc}\hline
 &  & PTHY &  &&  & BPV &  \\[1pt] \cline{2-4}\cline{6-8}
Target & $JV^{1,1}$ & $JV^{1,2}$ & $JV^{2,2}$ && $JV^{1,1}$ & $JV^{1,2}$ & $JV^{2,2}$ \\[1pt] \hline
NO &  &  &  &&  &  &  \\
$\lambda=(3,6)$ & .000 (.005) & .001 (.003) & .001 (.005) && .011 (.075) & .016 (.049) & .011 (.093) \\
$\lambda=(10,20)$ & .003 (.018) & .006 (.020) & .009 (.047) && .013 (.074) & .014 (.055) & .016 (.107) \\
$\lambda=(30,60)$ & .018 (.071) & .025 (.066) & .030 (.113) && .020 (.097) & .023 (.063) & .044 (.132) \\[5pt]
SCP1 &  & &  &&  &  & \\
$\lambda=(3,6)$ & $-$.006 (.032) & $-$.003 (.037) & $-$.004 (.041) && $-$.021 (.098) & $-$.022 (.077) & $-$.021 (.111) \\
$\lambda=(10,20)$ & $-$.006 (.050) & .002 (.055) & .002 (.071) && $-$.018 (.098) & $-$.020 (.084) & $-$.016 (.122) \\
$\lambda=(30,60)$ & .008 (.097) & .021 (.089) & .021 (.125) && $-$.012 (.113) & $-$.016 (.096) & .013 (.135) \\[5pt]
VG &  &  & &&  &  &  \\
$\lambda=(3,6)$ & $-$.008 (.039) & $-$.007 (.041) &  $-$.017 (.046) && $-$.029 (.116) & $-$.030 (.097) & $-$.038 (.129) \\
$\lambda=(10,20)$ & $-$.007 (.053) & $-$.003 (.053) &  $-$.013 (.070) && $-$.026 (.112) & $-$.028 (.100) & $-$.034 (.138) \\
$\lambda=(30,60)$ & .003 (.098) & .014 (.092) &  .004 (.125) && $-$.020 (.124) & $-$.026 (.121) & $-$.005 (.140) \\ \hline
\end{tabular}
\label{JVmodel1}\vspace{2mm}

\parbox{17cm}{\footnotesize
\textit{Note}. We report the bias and rmse of the estimators for 
the jump (co)variations included in the simulation study. The number reported in parenthesis is rmse.
}

\end{center}
\end{table}
}

{\linespread{1.2}
\begin{table}
\caption{Simulation results of the estimation of jump (co)variations in Model 2}
\begin{center}
\begin{tabular}{lccccccc}\hline
 &  & PTHY &  &&  & BPV &  \\[1pt] \cline{2-4}\cline{6-8}
Target & $JV^{1,1}$ & $JV^{1,2}$ & $JV^{2,2}$ && $JV^{1,1}$ & $JV^{1,2}$ & $JV^{2,2}$ \\[1pt] \hline
NO &  &  &  &&  &  &  \\
$\lambda=(3,6)$ & .000 (.002) & .001 (.003) & .001 (.005) && .134 (.141) & .010 (.038) & .137 (.145) \\
$\lambda=(10,20)$ & .003 (.011) & .006 (.015) & .009 (.023) && .136 (.144) & .013 (.040) & .141 (.152) \\
$\lambda=(30,60)$ & .017 (.039) & .035 (.059) & .052 (.087) && .143 (.155) & .018 (.049) & .169 (.186) \\[5pt]
SCP1 &  & &  &&  &  & \\
$\lambda=(3,6)$ & $-$.007 (.032) & $-$.009 (.037) & $-$.008 (.040) && .095 (.127) & $-$.025 (.082) & .099 (.136) \\
$\lambda=(10,20)$ & $-$.009 (.047) & $-$.005 (.054) & $-$.001 (.060) && .098 (.130) & $-$.024 (.087) & .103 (.141) \\
$\lambda=(30,60)$ & .004 (.065) & .023 (.086) & .038 (.107) && .104 (.139) & $-$.022 (.104) & .133 (.173) \\[5pt]
VG &  &  & &&  &  &  \\
$\lambda=(3,6)$ & $-$.012 (.036) & $-$.009 (.034) &  $-$.019 (.045) && .082 (.135) & $-$.021 (.071) & .083 (.122) \\
$\lambda=(10,20)$ & $-$.012 (.052) & $-$.004 (.046) &  $-$.015 (.061) && .084 (.137) & $-$.021 (.072) & .085 (.127) \\
$\lambda=(30,60)$ & $-$.002 (.078) & .024 (.085) &  .025 (.108) && .089 (.146) & $-$.013 (.084) & .118 (.163) \\ \hline
\end{tabular}
\label{JVmodel2}\vspace{2mm}

\parbox{17cm}{\footnotesize
\textit{Note}. We report the bias and rmse of the estimators for 
the jump (co)variations included in the simulation study. The number reported in parenthesis is rmse. All of the reported numbers are divided by $(0.2/\sqrt{252})^2$.
}

\end{center}
\end{table}
}

{\linespread{1.2}
\begin{table}
\caption{Simulation results of the estimation of jump (co)variations in Model 3}
\begin{center}
\begin{tabular}{lccccccc}\hline
 &  & PTHY &  &&  & BPV &  \\[1pt] \cline{2-4}\cline{6-8}
Target & $JV^{1,1}$ & $JV^{1,2}$ & $JV^{2,2}$ && $JV^{1,1}$ & $JV^{1,2}$ & $JV^{2,2}$ \\[1pt] \hline
NO &  &  &  &&  &  &  \\
$\lambda=(3,6)$ & .001 (.007) & .002 (.006) & .003 (.016) && .013 (.087) & .009 (.051) & .012 (.107) \\
$\lambda=(10,20)$ & .007 (.032) & .011 (.032) & .017 (.060) && .014 (.091) & .009 (.052) & .017 (.108) \\
$\lambda=(30,60)$ & .030 (.089) & .039 (.081) & .055 (.153) && .022 (.094) & .015 (.054) & .053 (.138) \\[5pt]
SCP1 &  & &  &&  &  & \\
$\lambda=(3,6)$ & $-$.003 (.022) & $-$.001 (.027) & .000 (.035) && $-$.004 (.093) & $-$.010 (.065) & $-$.006 (.112) \\
$\lambda=(10,20)$ & .001 (.052) & .009 (.059) & .013 (.078) && $-$.016 (.110) & $-$.024 (.083) & $-$.014 (.120) \\
$\lambda=(30,60)$ & .025 (.098) & .038 (.093) & .051 (.157) && .006 (.099) & $-$.005 (.069) & .037 (.135) \\[5pt]
VG &  &  & &&  &  &  \\
$\lambda=(3,6)$ & $-$.005 (.038) & $-$.004 (.041) &  $-$.010 (.045) && $-$.026 (.123) & $-$.032 (.098) & $-$.037 (.138) \\
$\lambda=(10,20)$ & $-$.000 (.055) & .005 (.058) &  .001 (.073) && $-$.025 (.124) & $-$.033 (.100) & $-$.031 (.136) \\
$\lambda=(30,60)$ & .018 (.109) & .003 (.103) &  .034 (.151) && $-$.015 (.122) & $-$.033 (.118) & .006 (.137) \\ \hline
\end{tabular}
\label{JVmodel3}\vspace{2mm}

\parbox{17cm}{\footnotesize
\textit{Note}. We report the bias and rmse of the estimators for 
the jump (co)variations included in the simulation study. The number reported in parenthesis is rmse.
}
\end{center}
\end{table}
}

%% file: pthy/pthy_proof_mod.tex
%%%%%%%%%%%%%%%%%%%%%%%%%%%%%%%%%%%%%%%%%%%%%%%%%%%%%%%%%%%%%%%%%%%%%%%%%%%%%
%                         Consistensy of PTHY
%%%%%%%%%%%%%%%%%%%%%%%%%%%%%%%%%%%%%%%%%%%%%%%%%%%%%%%%%%%%%%%%%%%%%%%%%%%%%

\section{Proof of Theorem \ref{thmFAnoise}}\label{proofthmFAnoise}

%%%%%%%%%%%%%%%%%%%%%%%%%%%%%%%%%%%%%%%%%%%%%%%%%%%%%%%%%%%%%%%%%
%                    localzation (FAnoise)
%%%%%%%%%%%%%%%%%%%%%%%%%%%%%%%%%%%%%%%%%%%%%%%%%%%%%%%%%%%%%%%%%

First note that for the proof we can use a localization procedure, and which allows us to systematically replace the conditions [C1], [C2], [N$^\flat_r$] and [F] by the following strengthened version:
\begin{enumerate}
\item[{[SC1]}] There is a positive constant $K$ such that $b_nN^n_t\leq K$ for all $n$ and $t$. 
\item[{[SC2]}] [C2] holds, and $(A^1)'$, $(A^2)'$, $(\underline{A}^1)'$, $(\underline{A}^2)'$ and $[V,W]'$ for each $V,W=X^1,X^2,\underline{X}^1,\underline{X}^2$ are bounded.
\item[{[SN$^\flat_r$]}] $(\int |z|^rQ_t(\mathrm{d}z))_{t\in\mathbb{R}_+}$ is a bounded process.
\item[{[SF]}]  We have [F] and there is a positive constant $B$ such that
\begin{equation}\label{jumpsize}
B^{-1}<\inf_{k\in\mathbb{N}}|\gamma_k^{l}|<\sup_{k\in\mathbb{N}}|\gamma_k^{l}|<B
\end{equation}
for each $l=1,2$.
\end{enumerate} 

Next we introduce the following strengthened version of the condition [T]:
\begin{enumerate}
\item[{[ST]}] For each $l=1,2$ we have $\varrho^{l}_n(t)=\alpha_n^{l}(t)\rho_n$, where $(\rho_n)_{n\in\mathbb{N}}$ is the same one in [T] and $(\alpha_n^{l}(t))_{n\in\mathbb{N}}$ is a sequence of (not necessarily adapted) positive-valued stochastic processes such that there exists a positive constant $K_0$ satisfying
\begin{equation*}
\frac{1}{K_0}<\inf_{t\in\mathbb{R}_+}\alpha_n^{(q)}(t)<\sup_{t\in\mathbb{R}_+}\alpha_n^{(q)}(t)<K_0,\qquad n=1,2,\dots.
\end{equation*}
\end{enumerate}

\begin{lem}\label{lemthreshold}
Let $(c_n)$ be a sequence of positive numbers. If we have
\begin{equation}\label{auxucp}
c_n^{-1}\{ \widehat{PTHY}(\mathsf{Z}^{1}, \mathsf{Z}^{2})^n-\widehat{PHY}(\mathsf{X}^{1}, \mathsf{X}^{2})^n\}\xrightarrow{ucp}0\quad\mathrm{as}\quad n\rightarrow\infty
\end{equation}
under the condition $[\mathrm{ST}]$, then we have also $(\ref{auxucp})$ under the condition $[\mathrm{T}]$. 
\end{lem}

\begin{proof}
Let $t>0$ and $k\in\mathbb{N}$. Suppose that [T] holds. Then, for an arbitrary $\varepsilon>0$, there exists a positive number $K$ such that
\begin{equation*}
\sup_{n\in\mathbb{N}}P\left(\sup_{0\leq s< R_k^{l}}\alpha_n^{l}(s)\geq K\right)<\varepsilon\quad \mathrm{and}\quad
\sup_{n\in\mathbb{N}}P\left(\sup_{0\leq s< R_k^{l}}[1/\alpha_n^{l}(s)]\geq K\right)<\varepsilon,\qquad l=1,2.
\end{equation*}
Hence for any $\eta>0$ we have
\begin{align*}
P\left(\Psi^n(t)>\eta\right)
\leq  P(R_k^{1}\wedge R_k^{2}\leq t)+4\varepsilon
+P\left(\Psi^n(t\wedge R_k^{1}\wedge R_k^{2})>\eta,\max_{l\in\{1,2\}}\left[\sup_{0\leq s<R_k^{l}}\alpha_n^{l}(s)\vee\sup_{0\leq s<R_k^{l}}\frac{1}{\alpha_n^{l}(s)}\right]<K\right),
\end{align*}
where $\Psi^n(t):=\sup_{0\leq s\leq t}c_n^{-1}|\widehat{PTHY}(\mathsf{Z}^{1}, \mathsf{Z}^{2})^n_s-\widehat{PHY}(\mathsf{X}^{1}, \mathsf{X}^{2})^n_s|$. Therefore, by the assumption we obtain
\begin{equation*}
\limsup_{n\rightarrow\infty}P\left(\Psi^n(t)>\eta\right)
\leq P(R_k^{1}\wedge R_k^{2}\leq t)+4\varepsilon.
\end{equation*}
Since $\varepsilon$ is arbitrary, we can replace $\varepsilon$ in the above inequality with 0. Finally, with $k$ tending to 0, we obtain the desired result.
\end{proof}

We need a modification of sampling times as follows. We write $\bar{r}_n=b_n^{\xi'}$. Next, let $\upsilon_n=\inf\{t|r_n(t)>\bar{r}_n\}$, and define a sequence $(\widetilde{S}^i)_{i\in\mathbb{Z}^+}$ sequentially by $\widetilde{S}^i=S^i$ if $S^i<\upsilon_n$, otherwise $\widetilde{S}^i=\widetilde{S}^{i-1}+\bar{r}_n$.
\if0
\begin{equation*}
\widetilde{S}^i=
\left\{\begin{array}{ll}
S^i & \textrm{if $S^i<\upsilon_n$},\\
\widetilde{S}^{i-1}+\bar{r}_n & \textrm{otherwise}.
\end{array}\right.
\end{equation*}
\fi
Then, $(\widetilde{S}^i)$ is obviously a sequence of $\mathbf{F}^{(0)}$-stopping times satisfying $(\ref{increace})$ and $\sup_{i\in\mathbb{N}}(\widetilde{S}^i-\widetilde{S}^{i-1})\leq\bar{r}_n$. Furthermore, for any $t>0$ we have $P(\bigcap_i\{\widetilde{S}^i\wedge t\neq S^i\wedge t\})\leq P(\upsilon_n<t)\to0$ as $n\to\infty$ by $(\ref{A4})$. By replacing $(S^i)$ with $(T^j)$, we can construct a sequence $(\widetilde{T}^j)$ in a similar manner. This argument implies that we may also assume that
\begin{equation}\label{SA4}
\sup_{t\in\mathbb{R}_+}r_n(t)\leq\bar{r}_n
\end{equation}
by an appropriate localization procedure.

Finally, we note that the inequality (3.2) of \cite{Koike2012phy} holds true:
\begin{equation}\label{sumbarK}
\sum_{j=0}^\infty\bar{K}^{kj}\leq 2 k_n+1\quad\mathrm{and}\quad\sum_{i=0}^\infty\bar{K}^{ik}\leq 2 k_n+1,\qquad k=0,1,\dots.
\end{equation}

Now we introduce some notation and prove some lemmas which we will also use later. Set $g^n_p=g(p/k_n)$ and $\Delta(g)^n_p=g^n_{p+1}-g^n_p$ for every $n,p$. For any semimartingale $V$ and any (random) interval $I$, we define the processes $V(I)_t$ and $I_t$ by $V(I)_t=\int_0^t1_I(s-)\mathrm{d}V_s$ and $I_t=1_I(t)$ respectively. Moreover, set
\begin{align*}
\bar{V}(\widehat{\mathcal{I}})^i_t=\sum_{p=1}^{k_n-1}g^n_p V(\widehat{I}^{i+p})_t,\qquad
\bar{V}(\widehat{\mathcal{J}})^j_t=\sum_{q=1}^{k_n-1}g^n_q V(\widehat{J}^{j+q})_t
\end{align*}
and
\begin{align*}
\widetilde{V}(\widehat{\mathcal{I}})^i_t=-b_n^{-1/2}\sum_{p=0}^{k_n-1}\Delta(g)^n_p V(\check{I}^{i+p})_t,\qquad
\widetilde{V}(\widehat{\mathcal{J}})^j_t=-b_n^{-1/2}\sum_{q=0}^{k_n-1}\Delta(g)^n_q V(\check{J}^{j+q})_t
\end{align*}
for each $t\in\mathbb{R}_+$ and $i,j\in\mathbb{Z}_+$. The following lemma is an analog to Lemma 3.1 of \cite{Koike2012} and was used in \cite{Koike2012phy}:

%%%%%%%%%%%%%%%%%%%%%%%%%%%%%%%%%%%%%%%%%%%%%%%%%%%%%%%%%%%%%%%%%%
%                         Lemma modulus
%%%%%%%%%%%%%%%%%%%%%%%%%%%%%%%%%%%%%%%%%%%%%%%%%%%%%%%%%%%%%%%%%%

\begin{lem}\label{modulus}
Suppose $A^l$, $[X^l]$, $\underline{A}^l$ and $[\underline{X}^l]$ are absolutely continuous with locally bounded derivatives for $l=1,2$. Suppose also $(\ref{SA4})$ holds. Then a.s. we have
\begin{align}
\limsup_{n\rightarrow\infty}\sup_{i\in\mathbb{N}}
\frac{|\bar{X}^1(\widehat{\mathcal{I}})^i_t|}{\sqrt{2 k_n\bar{r}_n\log\frac{1}{\bar{r}_n}}}\leq \| g\|_\infty\sup_{0\leq s\leq t}|[X^1]'_s|,\qquad
\limsup_{n\rightarrow\infty}\sup_{i\in\mathbb{N}}
\frac{|\widetilde{\underline{X}^1}(\widehat{\mathcal{I}})^i_t|}{\sqrt{2 k_n\bar{r}_n\log\frac{1}{\bar{r}_n}}}\leq L\sup_{0\leq s\leq t}|[\underline{X}^1]'_s|,\label{modulus1}\\
\limsup_{n\rightarrow\infty}\sup_{i\in\mathbb{N}}
\frac{|\bar{X}^2(\widehat{\mathcal{I}})^i_t|}{\sqrt{2 k_n\bar{r}_n\log\frac{1}{\bar{r}_n}}}\leq \| g\|_\infty\sup_{0\leq s\leq t}|[X^2]'_s|,\qquad
\limsup_{n\rightarrow\infty}\sup_{i\in\mathbb{N}}
\frac{|\widetilde{\underline{X}^2}(\widehat{\mathcal{I}})^i_t|}{\sqrt{2 k_n\bar{r}_n\log\frac{1}{\bar{r}_n}}}\leq L\sup_{0\leq s\leq t}|[\underline{X}^2]'_s|\label{modulus2}
\end{align}
for any $t>0$, where $L$ is a positive constant which only depends on $g$.
\end{lem}

\begin{proof}
Combining a representation of a continuous local martingale with Brownian motion and L\'{e}vy's theorem on the uniform modulus of continuity of Brownian motion, we obtain
\begin{align*}
\limsup_{\delta\rightarrow +0}\sup_{
\begin{subarray}{c}
s,u\in[0,t]\\
|s-u|\leq\delta
\end{subarray}}
\frac{|X^1_s-X^1_u|}{\sqrt{2\delta\log\frac{1}{\delta}}}\leq \sup_{0\leq s\leq t}|[X^1]'_s|,\qquad
\limsup_{\delta\rightarrow +0}\sup_{
\begin{subarray}{c}
s,u\in[0,t]\\
|s-u|\leq\delta
\end{subarray}}
\frac{|\underline{X}^1_s-\underline{X}^1_u|}{\sqrt{2\delta\log\frac{1}{\delta}}}\leq \sup_{0\leq s\leq t}|[\underline{X}^1]'_s|,
\end{align*}
where $\underline{X}^1=-(\sum_{p=1}^\infty\check{I}^p_-)\bullet\underline{X}^1$. Since $\bar{X}^1(\widehat{\mathcal{I}})^i_t=-\sum_{p=0}^{k_n-1}\Delta(g)^n_p(X^1_{\widehat{S}^{i+p}\wedge t}-X^1_{\widehat{S}^{i}\wedge t})$ and $|\Delta(g)^n_p|\leq\frac{1}{k_n}\| g\|_\infty$, we obtain the first inequality in $(\ref{modulus1})$. On the other hand, since Abel's partial summation formula yields $\widetilde{\underline{X}^1}(\widehat{\mathcal{I}})^i_t=b_n^{-1/2}\sum_{p=0}^{k_n-1}\{\Delta(g)^n_{p+1}-\Delta(g)^n_p\}(\underline{X}^1_{\widehat{S}^{i+p}\wedge t}-\underline{X}^1_{\widehat{S}^{i-1}\wedge t})$, and $\Delta(g)^n_{p+1}-\Delta(g)^n_p=-\int_{p/k_n}^{(p+1)/k_n}\{g'(x+1/k_n)-g'(x)\}\mathrm{d}x$, the piecewise Lipschitz continuity of $g'$ and $(\ref{window})$ imply the second inequality in $(\ref{modulus1})$.
By symmetry we also obtain $(\ref{modulus2})$.
\end{proof}

We can strengthen Lemma \ref{modulus} by a localization if we assume that $(\ref{SA4})$ and [SC2] hold, so that in the remainder of this section we always assume that we have a positive constant $K$ and a positive integer $n_0$ such that
\begin{align}
\sup_{i\in\mathbb{N}}
\frac{|\bar{X}^1_g(\widehat{\mathcal{I}})^i_t(\omega)|+|\widetilde{\underline{X}^1}_g(\widehat{\mathcal{I}})^i_t(\omega)|}{\sqrt{2 k_n\bar{r}_n|\log b_n|}}
+\sup_{j\in\mathbb{N}}
\frac{|\bar{X}^2_g(\widehat{\mathcal{J}})^j_t(\omega)|+|\widetilde{\underline{X}^2}_g(\widehat{\mathcal{J}})^j_t(\omega)|}{\sqrt{2 k_n \bar{r}_n|\log b_n|}}
\leq K\label{absmod2}
\end{align}
for all $t>0$ and $\omega\in\Omega$ if $n\geq n_0$. Moreover, we only consider sufficiently large $n$ such that $n\geq n_0$.

%%%%%%%%%%%%%%%%%%%%%%%%%%%%%%%%%%%%%%%%%%%%%%%%%%%%%%%%%%%%%%%%%%
%                        Lemma moment
%%%%%%%%%%%%%%%%%%%%%%%%%%%%%%%%%%%%%%%%%%%%%%%%%%%%%%%%%%%%%%%%%%

Next, set
\begin{align*}
\overline{\zeta}^1(\widehat{\mathcal{I}})^i=\sum_{p=0}^{k_n-1}\Delta(g)^n_p\zeta^1_{\widehat{S}^{i+p}},\qquad
\overline{\zeta}^2(\widehat{\mathcal{J}})^j=\sum_{q=0}^{k_n-1}\Delta(g)^n_q\zeta^2_{\widehat{T}^{j+q}}
\end{align*}
for each $i,j\in\mathbb{Z}_+$. Throughout the discussions, for (random) sequences $(x_n)$ and $(y_n)$, $x_n\lesssim y_n$ means that there exists a (non-random) constant $C\in[0,\infty)$ such that $x_n\leq Cy_n$ for large $n$. We denote by $E_0$ a conditional expectation given $\mathcal{F}^{(0)}$, i.e. $E_0[\cdot]:=E[\cdot|\mathcal{F}^{(0)}]$.
\begin{lem}\label{moment}
Suppose $[\mathrm{SN}^\flat_r]$ hold for some $r\in[2,\infty)$. Then there exists a some positive constant $K_r$ independent of $n$ such that
\begin{equation}\label{eqmoment}
E_0[|\overline{\zeta}^{1}(\widehat{\mathcal{I}})^i|^r]\leq K_r k_n^{-r/2},\qquad
E_0[|\overline{\zeta}^{2}(\widehat{\mathcal{J}})^j|^r]\leq K_r k_n^{-r/2}
\end{equation}
for all $i,j\in\mathbb{N}$.
\end{lem}
\begin{proof}
The Burkholder-Davis-Gundy inequality, Jensen's inequality and the Lipschitz continuity of $g$ yield
\begin{align*}
E_0[|\overline{\zeta}^{1}(\widehat{\mathcal{I}})^i|^r]
\lesssim  E_0\left[\left\{\sum_{p=0}^{k_n-1}|\Delta(g)^n_p\zeta^{1}_{\widehat{S}^{i+p}}|^2\right\}^{r/2}\right]
\leq k_n^{r/2-1}\sum_{p=0}^{k_n-1}E_0[|\Delta(g)^n_p \zeta^{1}_{\widehat{S}^{i+p}}|^r]
\lesssim k_n^{-r/2},
\end{align*}
hence we obtain the first inequality of $(\ref{eqmoment})$. By symmetry we also obtain the second one.
\end{proof}

%%%%%%%%%%%%%%%%%%%%%%%%%%%%%%%%%%%%%%%%%%%%%%%%%%%%%%%%%%%%%%%%%%%%%
%                        Lemma FAnoise
%%%%%%%%%%%%%%%%%%%%%%%%%%%%%%%%%%%%%%%%%%%%%%%%%%%%%%%%%%%%%%%%%%%%%

Set $\bar{I}^i=[\widehat{S}^i,\widehat{S}^{i+k_n})$, $\bar{J}^j=[\widehat{T}^j,\widehat{T}^{j+k_n})$ and $\bar{R}^\vee(i,j)=\widehat{S}^{i+k_n}\vee\widehat{T}^{j+k_n}$ for each $i,j\in\mathbb{Z}_+$.
\begin{lem}\label{lemFAnoise}
Let $c$ be a positive number. Suppose $[\mathrm{SC}2]$ and $[\mathrm{SF}]$ hold. Suppose also $[\mathrm{SN}^\flat_r]$ holds for some $r\in(2,\infty)$. Then for all $t>0$ we have
\begin{equation}\label{eqFAnoise}
\sum_{i=1}^{\infty}P\left(|\overline{\zeta}^{1}(\widehat{\mathcal{I}})^i|\geq c, N^{1}(\bar{I}^i)_t\neq 0,  \widehat{S}^{i+k_n}\leq t\right)
\rightarrow 0,\qquad
\sum_{j=1}^{\infty}P\left(|\overline{\zeta}^{2}(\widehat{\mathcal{J}})^j|\geq c, N^{2}(\bar{J}^j)_t\neq 0,  \widehat{T}^{j+k_n}\leq t\right)
\rightarrow 0
\end{equation}
as $n\rightarrow\infty$. 
\end{lem}

\begin{proof}
Lemma \ref{moment} yields
$E_0[|\overline{\zeta}^{1}(\widehat{\mathcal{I}})^i|^r 1_{\{ N^{1}(\bar{I}^i)_t\neq 0\}}]
\lesssim k_n^{-r/2}1_{\{ N^{1}(\bar{I}^i)_t\neq 0\}}$
uniformly in $i$. Since $N^{1}$ is a point process, we obtain
\begin{align*}
\sum_{i=1}^{\infty}P\left(|\overline{\zeta}^{1}(\widehat{\mathcal{I}})^i|\geq c, N^{1}(\bar{I}^i)_t\neq 0,  \widehat{S}^{i+k_n}\leq t\big|\mathcal{F}^{(0)}\right)
\leq \frac{1}{c^r}\sum_{i=1}^{\infty}E_0[|\overline{\zeta}^{1}(\widehat{\mathcal{I}})^i|^r 1_{\{ N^{1}(\bar{I}^i)_t\neq 0\}}]
\lesssim k_n^{1-r/2}N^{1}_t ,
\end{align*} 
and thus we obtain the first equation of $(\ref{eqFAnoise})$ since $r>2$ and $k_n\rightarrow\infty$ as $n\rightarrow\infty$. By symmetry we also obtain the second equation of $(\ref{eqFAnoise})$, and thus we complete the proof of lemma.
\end{proof}

\if0
\begin{proof}
Take a positive $\gamma$ satisfying $\gamma>1/(\xi'-1/2)$. Lemma \ref{moment} and $(\ref{absmod2})$ yield
\begin{align*}
&P\left(|\overline{\zeta}^{1}(\widehat{\mathcal{I}})^i|\geq c, N^{1}(\bar{I}^i)_t\neq 0,  \widehat{S}^{i+k_n}\leq t\big|\mathcal{F}^{(0)}\right)
\leq \left\{c^{-r}E_0[|\overline{\zeta}^{1}(\widehat{\mathcal{I}})^i|^r+c^{-\gamma}|\widetilde{\underline{X}^1}(\widehat{\mathcal{I}})^i_t| \right\}1_{\{ N^{1}(\bar{I}^i)_t\neq 0\}}\\
\lesssim& \left(c^{-r}k_n^{-r/2}+c^{-\gamma}(2 k_n \bar{r}_n|\log b_n|)^{\gamma/2}\right)1_{\{ N^{1}(\bar{I}^i)_t\neq 0\}}
\end{align*}
uniformly in $i$. Since $N^{1}$ is a point process, we obtain
\begin{align*}
&\sum_{i=1}^{\infty}P^0\left(|\overline{\zeta}^{1}(\widehat{\mathcal{I}})^i|\geq c, N^{1}(\bar{I}^i)_t\neq 0,  \widehat{S}^{i+k_n}\leq t\big|\mathcal{F}^{(0)}\right)\\
\lesssim& \left(c^{-r}k_n^{-r/2}+c^{-\gamma}(2 k_n \bar{r}_n|\log b_n|)^{\gamma/2}\right)\sum_{i=1}^\infty1_{\{ N^{1}(\bar{I}^i)_t\neq 0\}}
\leq \left(c^{-r}k_n^{1-r/2}+c^{-\gamma}(2 k_n \bar{r}_n|\log b_n|)^{\gamma/2}k_n\right)N^{1}_t ,
\end{align*} 
and thus we obtain $(\ref{eqFAnoise1})$ since $1-\frac{r}{2}<0$ and $\frac{\gamma}{2}(\xi'-\frac{1}{2})>\frac{1}{2}$. By symmetry we also obtain $(\ref{eqFAnoise2})$, and thus we complete the proof of lemma.
\end{proof}
\fi

%%%%%%%%%%%%%%%%%%%%%%%%%%%%%%%%%%%%%%%%%%%%%%%%%%%%%%%%%%%%%%%%%%%%%%
%                        proof of thmFAnoise
%%%%%%%%%%%%%%%%%%%%%%%%%%%%%%%%%%%%%%%%%%%%%%%%%%%%%%%%%%%%%%%%%%%%%%%

\begin{proof}[\upshape{\bfseries{Proof of Theorem \ref{thmFAnoise}}}]
By a localization procedure, we may replace the conditions [F], [C1]-[C2] and [N$^\flat_r$] with [SF], [C1]-[C2] and [SN$^\flat_r$] respectively. Moreover, we can also replace the condition [T] with [ST] by Lemma \ref{lemthreshold}, while $(\ref{A4})$ can be replaced with $(\ref{SA4})$ due to the above argument.

We decompose the target quantity as
\begin{align*}
&PTHY(\mathsf{Z}^{1},\mathsf{Z}^{2})^n_t-PHY(\mathsf{X}^{1},\mathsf{X}^{2})^n_t\\
=&\frac{1}{(\psi_{HY}k_n)^2}\Biggl(-\sum_{i,j:\bar{R}^\vee(i,j)\leq t}\overline{\mathsf{X}}^{1}(\widehat{\mathcal{I}})^i\overline{\mathsf{X}}^{2}(\widehat{\mathcal{J}})^j \bar{K}^{i j} 1_{\{|\overline{\mathsf{Z}}^{1}(\widehat{\mathcal{I}})^i|^2>\varrho_n^{1}[i]\}\cup\{|\overline{\mathsf{Z}}^{2}(\widehat{\mathcal{J}})^j|^2>\varrho_n^{2}[j]\}}\\
&+\sum_{i,j:\bar{R}^\vee(i,j)\leq t}\left\{\bar{D}^{1}(\widehat{\mathcal{I}})^i_t\overline{\mathsf{X}}^{2}(\widehat{\mathcal{J}})^j+\overline{\mathsf{X}}^{1}(\widehat{\mathcal{I}})^i\bar{D}^{2}(\widehat{\mathcal{J}})^j_t+\bar{D}^{1}(\widehat{\mathcal{I}})^i_t\bar{D}^{2}(\widehat{\mathcal{J}})^j_t\right\} \bar{K}^{i j} 1_{\{|\overline{\mathsf{Z}}^{1}(\widehat{\mathcal{I}})^i|^2\leq\varrho_n^{1}[i],|\overline{\mathsf{Z}}^{2}(\widehat{\mathcal{J}})^j|^2\leq\varrho_n^{2}[j]\}}\Biggr)\\
=:&\mathbb{I}_t+\mathbb{II}_t+\mathbb{III}_t+\mathbb{IV}_t,
\end{align*}
where $D^{l}_t:=\sum_{k=1}^{N_t^{l}}\gamma_k^{l}$ for each $l=1,2$.

First consider $\mathbb{I}$. By the Schwarz inequality and $(\ref{sumbarK})$, we have
{\small \begin{align*}
\sup_{0\leq s\leq t}|\mathbb{I}_s|
\leq &\frac{1}{(\psi_{HY}k_n)^2}\left\{\sum_{i,j:\bar{R}^\vee(i,j)\leq t}|\overline{\mathsf{X}}^{1}(\widehat{\mathcal{I}})^i|^2\bar{K}^{i j} 1_{\{|\overline{\mathsf{Z}}^{1}(\widehat{\mathcal{I}})^i|^2>\varrho_n^{1}[i]\}}\right\}^{1/2}
\left\{\sum_{i,j:\bar{R}^\vee(i,j)\leq t}|\overline{\mathsf{X}}^{2}(\widehat{\mathcal{J}})^j|^2\bar{K}^{i j} 1_{\{|\overline{\mathsf{Z}}^{2}(\widehat{\mathcal{J}})^j|^2>\varrho_n^{2}[j]\}}\right\}^{1/2}\\
\lesssim &\frac{1}{k_n}\left\{\sum_{i:\widehat{S}^{i+k_n}\leq t}|\overline{\mathsf{X}}^{1}(\widehat{\mathcal{I}})^i|^2 1_{\{|\overline{\mathsf{Z}}^{1}(\widehat{\mathcal{I}})^i|^2>\varrho_n^{1}[i]\}}\right\}^{1/2}\left\{\sum_{j:\widehat{T}^{j+k_n}\leq t}|\overline{\mathsf{X}}^{2}(\widehat{\mathcal{J}})^j|^2\ 1_{\{|\overline{\mathsf{Z}}^{2}(\widehat{\mathcal{J}})^j|^2>\varrho_n^{2}[j]\}}\right\}^{1/2}.
\end{align*}}
Consider $\sum_{i:\widehat{S}^{i+k_n}\leq t}|\overline{\mathsf{X}}^{1}(\widehat{\mathcal{I}})^i|^2 1_{\{|\overline{\mathsf{Z}}^{1}(\widehat{\mathcal{I}})^i|^2>\varrho_n^{1}[i]\}}$. We decompose it as
\begin{align*}
\sum_{i:\widehat{S}^{i+k_n}\leq t}|\overline{\mathsf{X}}^{1}(\widehat{\mathcal{I}})^i|^2 1_{\{|\overline{\mathsf{Z}}^{1}(\widehat{\mathcal{I}})^i|^2>\varrho_n^{1}[i]\}}
&= \sum_{i:\widehat{S}^{i+k_n}\leq t}|\overline{\mathsf{X}}^{1}(\widehat{\mathcal{I}})^i|^2\left( 1_{\{|\overline{\mathsf{X}}^{1}(\widehat{\mathcal{I}})^i|^2>\varrho_n^{1}[i], N^{1}(\bar{I}^i)_t=0\}}
+1_{\{|\overline{\mathsf{Z}}^{1}(\widehat{\mathcal{I}})^i|^2>\varrho_n^{1}[i], N^{1}(\bar{I}^i)_t\neq 0\}}\right)\\
&=:A_{1,t}+A_{2,t}.
\end{align*}
On $\{|\overline{\mathsf{X}}^{1}(\widehat{\mathcal{I}})^i|^2>\varrho_n^{1}[i], \widehat{S}^{i+k_n}\leq t\}$ we have
\begin{align*}
|\overline{\zeta}^{1}(\widehat{\mathcal{I}})^i|
\geq |\overline{\mathsf{X}}^{1}(\widehat{\mathcal{I}})^i|-|\bar{X}^{1}(\widehat{\mathcal{I}})^i_t|-|\widetilde{\underline{X}^1}(\widehat{\mathcal{I}})^i_t|
>\sqrt{\rho_n}\left(\frac{1}{\sqrt{K_0}}-2K\sqrt{\frac{2 k_n\bar{r}_n|\log b_n|}{\rho_n}}\right)
\end{align*}
by [$\mathrm{ST}$] and $(\ref{absmod2})$. Hence by $(\ref{threshold2})$ we have
{\small \begin{align*}
A_{1,t}\leq
\sum_{i:\widehat{S}^{i+k_n}\leq t} |\overline{\mathsf{X}}^{1}(\widehat{\mathcal{I}})^i|^21_{\{\overline{\zeta}^{1}(\widehat{\mathcal{I}})^i|^2>\rho_n/4K_0\}}
\leq 2\left(\frac{4 K_0}{\rho_n}\right)^{\frac{r}{2}}\sum_{i:\widehat{S}^{i+k_n}\leq t} |\bar{X}^{1}(\widehat{\mathcal{I}})^i_t|^2|\overline{\zeta}^{1}(\widehat{\mathcal{I}})^i|^r
+2\left(\frac{4 K_0}{\rho_n}\right)^{\frac{r-2}{2}}\sum_{i:\widehat{S}^{i+k_n}\leq t} |\overline{\zeta}^{1}(\widehat{\mathcal{I}})^i|^r,
\end{align*}}
and thus $(\ref{threshold2})$ and $(\ref{absmod2})$ imply that
$A_{1,t}
\lesssim(\rho_n)^{-\frac{r-2}{2}}\sum_{i:\widehat{S}^{i+k_n}\leq t} |\overline{\zeta}^{1}(\widehat{\mathcal{I}})^i|^r.$
Hence Lemma \ref{moment} and [SC1] yield
\begin{equation}\label{nFAA1}
E[A_{1,t}]\lesssim (b_n k_n)^{-1}(k_n\rho_n)^{-\frac{r-2}{2}}.
\end{equation}
On the other hand, since $|\bar{I}^i(t)|\leq k_n\bar{r}_n\rightarrow 0$ and $N^{1}$ is a point process, pathwise for sufficiently large $n$ there exists a some index $k(i)\in\mathbb{N}$ for each $i$ such that $\bar{D}^{1}(\widehat{\mathcal{I}})^i_t=\gamma_k(i)N^{1}(\bar{I}^i)_t$. Hence by $(\ref{jumpsize})$ we have $|\bar{D}^{1}(\widehat{\mathcal{I}})^i_t|\geq B^{-1}$ on $\{|\overline{\mathsf{Z}}^{1}(\widehat{\mathcal{I}})^i|^2\leq\varrho_n^{1}[i], N^{1}(\bar{I}^i)_t\neq 0, \widehat{S}^{i+k_n}\leq t\}$ for each $i$ pathwise for sufficiently large $n$.
Moreover, on $\{|\overline{\mathsf{Z}}^{1}(\widehat{\mathcal{I}})^i|^2\leq\varrho_n^{1}[i],|\bar{D}^{1}(\widehat{\mathcal{I}})^i_t|\geq B^{-1},\widehat{S}^{i+k_n}\leq t\}$ we have
\begin{align*}
|\overline{\zeta}^{1}(\widehat{\mathcal{I}})^i|
\geq|\bar{D}^{1}(\widehat{\mathcal{I}})^i_t|-|\mathsf{Z}^{1}(\widehat{\mathcal{I}})^i|-|\bar{X}^{1}(\widehat{\mathcal{I}})^i_t|-|\widetilde{\underline{X}^1}(\widehat{\mathcal{I}})^i_t|
\geq B^{-1} -\sqrt{\varrho_n^{1}[i]}-|\bar{X}^{1}(\widehat{\mathcal{I}})^i_t|-|\widetilde{\underline{X}^1}(\widehat{\mathcal{I}})^i_t|, 
\end{align*}
hence by $(\ref{absmod2})$ a.s. for sufficiently large $n$ we have
$A_{2,t}
\leq\sum_{i:\widehat{S}^{i+k_n}\leq t}|\overline{\mathsf{X}}^{1}(\widehat{\mathcal{I}})^i|^2 1_{\{|\overline{\zeta}^{1}(\widehat{\mathcal{I}})^i|>1/2 B, N^{1}(\bar{I}^i)_t\neq 0\}}.$
Therefore Lemma \ref{lemFAnoise} yields
\begin{equation}\label{nFAA2}
A_{2,t}= o_p\left((b_n k_n)^{-1}(k_n\rho_n)^{-\frac{r-2}{2}}\right).
\end{equation}
By $(\ref{nFAA1})$ and $(\ref{nFAA2})$ we obtain
$\sum_{i:\widehat{S}^{i+k_n}\leq t}|\overline{\mathsf{X}}^{1}(\widehat{\mathcal{I}})^i|^2 1_{\{|\overline{\mathsf{Z}}^{1}(\widehat{\mathcal{I}})^i|^2>\varrho_n^{1}[i]\}}=O_p\left((b_n k_n)^{-1}(k_n\rho_n)^{-\frac{r-2}{2}}\right),$
and by symmetry we also obtain
$\sum_{i:\widehat{T}^{j+k_n}\leq t}|\overline{\mathsf{X}}^{2}(\widehat{\mathcal{J}})^j|^2 1_{\{|\overline{\mathsf{Z}}^{2}(\widehat{\mathcal{J}})^j|^2>\varrho_n^{2}[j]\}}=O_p\left((b_n k_n)^{-1}(k_n\rho_n)^{-\frac{r-2}{2}}\right).$
Consequently, by $(\ref{window})$ we have
$\sup_{0\leq s\leq t}|\mathbb{I}_s|=
O_p\left(\left(b_n^{1/2}\rho_n^{-1}\right)^{\frac{r-2}{2}}\right).$

Next consider $\mathbb{II}$. Since $\bar{D}^{1}(\widehat{\mathcal{I}})^i_t=0$ on $\{N^{1}(\bar{I}^i)_t=0\}$, we have
\begin{equation*}
\mathbb{II}_t
=\frac{1}{(\psi_{HY}k_n)^2}\sum_{i,j\in\mathbb{Z}_+, \bar{R}^\vee(i,j)\leq t}\bar{D}^{1}(\widehat{\mathcal{I}})^i_t\overline{\mathsf{Z}}^{2}(\widehat{\mathcal{J}})^j \bar{K}^{i j} 1_{\{|\overline{\mathsf{Z}}^{1}(\widehat{\mathcal{I}})^i|^2\leq\varrho_n^{1}[i],|\overline{\mathsf{Z}}^{2}(\widehat{\mathcal{J}})^j|^2\leq\varrho_n^{2}[j], N^{1}(\bar{I}^i)_t\neq 0\}},
\end{equation*}
and thus an argument similar to the proof of $(\ref{nFAA2})$ yield $\sup_{0\leq s\leq t}|\mathbb{II}_s|=o_p(b_n^{1/4})$. Similarly we can show $\sup_{0\leq s\leq t}|\mathbb{III}_s|=o_p(b_n^{1/4})$ and $\sup_{0\leq s\leq t}|\mathbb{IV}_s|=o_p(b_n^{1/4})$.
Consequently, we complete the proof of Theorem \ref{thmFAnoise}.
\end{proof}

%%%%%%%%%%%%%%%%%%%%%%%%%%%%%%%%%%%%%%%%%%%%%%%%%%%%%%%%%%%%%%%%%%%%%%%%%%%%%%
%                         Proof of PTHY in IA case
%%%%%%%%%%%%%%%%%%%%%%%%%%%%%%%%%%%%%%%%%%%%%%%%%%%%%%%%%%%%%%%%%%%%%%%%%%%%%%

\section{Proof of Theorem \ref{thmIAnoise}}\label{proofthmIAnoise}

%%%%%%%%%%%%%%%%%%%%%%%%%%%%%%%%%%%%%%%%%%%%%%%%%%%%%%%%%%%%%%%%%
%                    localzation (IAnoise)
%%%%%%%%%%%%%%%%%%%%%%%%%%%%%%%%%%%%%%%%%%%%%%%%%%%%%%%%%%%%%%%%%

Exactly as in the previous section, we can use a localization procedure for the proof, and which allows us to systematically replace the conditions [A4], [A6] and [K$_\beta$] by the following strengthened versions: 

\begin{enumerate}
\item[{[SA4]}] $\xi\vee\frac{9}{10}<\xi'$ and $(\ref{SA4})$ holds.
\item[{[SA6]}] There exists a positive constant $C$ such that $b_n^{-1}H_n(t)\leq C$ for every $t$.
\item[{ [$\mathrm{SK}_\beta$]}] We have [$\mathrm{K}_\beta$] with $E^{1}=E^{2}=:E$ and $(A^{l})'$, $[M^{l}]'$, $(\underline{A}^l)'$ and $[\underline{M}^{l}]'$ ($l=1,2$) are bounded. Moreover, there is a non-negative bounded measurable function $\psi$ on $E$ such that
\begin{equation*}
\sup_{\omega\in\Omega ,t\in\mathbb{R}_+}|\delta^{l}(\omega, t,x)|\leq\psi(x)~\textrm{and}~\int_E \psi(x)^{\beta}F^{l}(\mathrm{d}x)<\infty ,~l=1,2.
\end{equation*}   
\end{enumerate}

Next, an argument similar to the one in the first part of Section 12 of \cite{HY2011} allows us to assume that $\frac{9}{10}<\xi<\xi'<1$ under [A2]. Furthermore, in the following we only consider sufficiently large $n$ such that
\begin{equation}\label{spp}
k_n\bar{r}_n<b_n^{\xi-1/2}.
\end{equation}
Note that we can use Lemma 11.2 of \cite{Koike2012phy} under [A2] and [SA4] in this situation. 

Now we prove some auxiliary results. Let
\begin{align*}
N^{l}:=1_{\{|\delta^{l}|>1\}}\star\mu^{l},\qquad
L^{l}:=\kappa(\delta^{l})\star(\mu^{l}-\nu^{l}).
\end{align*}
First we need the pre-averaged versions of some lemmas in Section 6 of \cite{Koike2012}. For processes $V$ and $W$, $V\bullet W$ denotes the integral (either stochastic or ordinary) of $V$ with respect to $W$.

%%%%%%%%%%%%%%%%%%%%%%%%%%%%%%%%%%%%%%%%%%%%%%%%%%%%%%%%%%%%%%%%%%%%%%%%%%%%%
%                         basic lemma
%%%%%%%%%%%%%%%%%%%%%%%%%%%%%%%%%%%%%%%%%%%%%%%%%%%%%%%%%%%%%%%%%%%%%%%%%%%%%
 
\begin{lem}\label{lembasic}
Suppose $[\mathrm{ST}]$, $[\mathrm{A}2]$, $[\mathrm{SA}4]$ and $[\mathrm{SK}_2]$ hold. Then for any $t>0$ we have
\begin{equation}\label{eqbasic}
\sum_{i=1}^{\infty}P\left( N^{1}(\bar{I}^i)_t\neq 0, |\bar{L}^{1}(\widehat{\mathcal{I}})^i_t|^2>4\varrho_n^{1}[i]\right)\rightarrow 0,\qquad
\sum_{j=1}^{\infty}P\left( N^{2}(\bar{J}^j)_t\neq 0, |\bar{L}^{2}(\widehat{\mathcal{J}})^j_t|^2>4\varrho_n^{2}[j]\right)\rightarrow 0
\end{equation}
as $n\rightarrow\infty$.
\end{lem}

\begin{proof}
Since $N^{1}$ and $L^{1}$ have no common jump, It\^{o}'s formula yields
\begin{align*}
N^{1}(\bar{I}^i)|\bar{L}^{1}(\widehat{\mathcal{I}})^i|^2
=|\bar{L}^{1}(\widehat{\mathcal{I}})^i|^2_{-}\bar{I}^i_-\bullet N^{1}+2 N^{1}(\bar{I}^i)_-\bar{L}^{1}(\widehat{\mathcal{I}})^i_-\bullet \bar{L}^{1}(\widehat{\mathcal{I}})^i+N^{1}(\bar{I}^i)_-\bullet[\bar{L}^{1}(\widehat{\mathcal{I}})^i],
\end{align*}
and thus we obtain
$E[N^{1}(\bar{I}^i)_t|\bar{L}^{1}(\widehat{\mathcal{I}})^i_t|^2]
=E[|\bar{L}^{1}(\widehat{\mathcal{I}})^i|^2_{-}\bar{I}^i_-\bullet \Lambda^{1}_t]+E[N^{1}(\bar{I}^i)_-\bullet\langle\bar{L}^{1}(\widehat{\mathcal{I}})^i\rangle_t]$
by the optional sampling theorem, where $\Lambda^1$ is the compensator of $N^1$. Since $\Lambda^{1}=1_{\{|\delta^{1}|>1\}}\star\nu^{1}$ and $\langle \bar{L}^{1}(\widehat{\mathcal{I}})^i \rangle=\sum_{p=1}^{k_n-1}g^n_p (\widehat{I}^{i+p}_-\bullet\langle L^{1}\rangle)=\sum_{p=1}^{k_n-1}g^n_p[\widehat{I}^{i+p}_-\kappa(\delta^{1})]\star\nu^{1}$, by [$\mathrm{SK}_2$], [A2] and the optional sampling theorem we have
\begin{align*}
E[N^{1}(\bar{I}^i)_t |\bar{L}^{1}(\widehat{\mathcal{I}})^i_t|^2]
\lesssim &\int_0^t E[|\bar{L}^{1}(\widehat{\mathcal{I}})^i_{s}|^2 \bar{I}^i_{s}]\mathrm{d}s+\sum_{p=1}^{k_n-1}\int_0^t E[N^{1}(\bar{I}^i)_{s}\widehat{I}^{i+p}_{s}]\mathrm{d}s\\
=&\int_0^t E[\langle \bar{L}^{1}(\widehat{\mathcal{I}})^i\rangle_{s} \bar{I}^i_{s}]\mathrm{d}s+\sum_{p=1}^{k_n-1}\int_0^t E[\Lambda^{1}(\bar{I}^i)_{s}\widehat{I}^{i+p}_{s}]\mathrm{d}s,
\end{align*}
and thus again [$\mathrm{SK}_2$] and the representations of $\Lambda^{1}$ and $\langle \bar{L}^{1}(\widehat{\mathcal{I}})^i\rangle$ yield
{\small \begin{align*}
E[N^{1}(\bar{I}^i)_t |\bar{L}^{1}(\widehat{\mathcal{I}})^i_t|^2]
\lesssim\sum_{p=1}^{k_n-1}\int_0^t E[|g^n_p[\widehat{I}^{i+p}_-\kappa(\delta^{1})]|\star\nu^{1}_{s} \widehat{I}^{i+p}_{s}]\mathrm{d}s+\sum_{p=1}^{k_n-1}\int_0^t E[\Lambda^{1}(\widehat{I}^{i+p})_{s}\widehat{I}^{i+p}_{s}]\mathrm{d}s
\lesssim\sum_{p=1}^{k_n-1}E\left[|\widehat{I}^{i+p}(t)|^2\right].
\end{align*}}
Since [SA4] implies $\sum_{i=1}^{\infty}|\widehat{I}^i(t)|^2\leq\overline{r}_n t$, we have
\begin{align*}
\sum_{i=1}^{\infty}P\left( N^{1}(\bar{I}^i)_t\neq 0, |\bar{L}^{1}(\widehat{\mathcal{I}})^i_t|^2>4\varrho_n^{1}[i]\right)
\leq \frac{K_0}{4\rho_n}\sum_{i=1}^{\infty}E[N^{1}(\bar{I}^i)_t |\bar{L}^{1}(\widehat{\mathcal{I}})^i_t|^2]
\lesssim \frac{k_n\overline{r}_n}{\rho_n} ,
\end{align*}
and thus $(\ref{threshold2})$ yields the first equation of $(\ref{eqbasic})$. Similarly we can prove the second equation of $(\ref{eqbasic})$.
\end{proof}

%%%%%%%%%%%%%%%%%%%%%%%%%%%%%%%%%%%%%%%%%%%%%%%%%%%%%%%%%%%%%%%%%%%%%%%%%%%%%%%%
%                     Lemma epsspecify
%%%%%%%%%%%%%%%%%%%%%%%%%%%%%%%%%%%%%%%%%%%%%%%%%%%%%%%%%%%%%%%%%%%%%%%%%%%%%%%%

Let $\varphi_p(\varepsilon)=\sum_{l=1}^2\int_{\{\psi\leq\varepsilon\}}\psi(x)^p F^{l}(\mathrm{d}x)$ for each $p\in[\beta,\infty)$. The following lemma is the same one as Lemma 6.7 of \cite{Koike2012}, and will be useful to prove the lemmas below.

\begin{lem}\label{epsspecify}
Suppose $[\mathrm{SK}_{\beta}]$ for some $\beta\in[0,2]$. Let $p$ be a positive number and $(\rho_n)$ be a sequence of positive numbers which tends to 0. Then there exists a sequence of numbers $\varepsilon_n\in(0,1]$ such that
\begin{equation}\label{eps2}
\limsup_{n\rightarrow\infty}(\rho_n^{-1}\varepsilon_n^2)^p\varphi_{\beta}(\varepsilon_n)\leq 1
\end{equation}
and
\begin{equation}\label{eps3}
\varphi_\beta(\varepsilon_n)\rightarrow 0,\qquad\sqrt{\rho_n}/\varepsilon_n\rightarrow 0
\end{equation} 
as $n\rightarrow\infty$.
\end{lem}

%%%%%%%%%%%%%%%%%%%%%%%%%%%%%%%%%%%%%%%%%%%%%%%%%%%%%%%%%%%%%%%%%%%%%%%%%%%%%%%%%
%                         proof of lemma epsspecify
%%%%%%%%%%%%%%%%%%%%%%%%%%%%%%%%%%%%%%%%%%%%%%%%%%%%%%%%%%%%%%%%%%%%%%%%%%%%%%%%%

\begin{proof}
The strategy of the proof is the same as the one in the proof of Lemma 7.4 of \cite{BNS2006}. Let
\begin{align*}
a_n':=\sup\{y\in(0,\infty)|y^p\varphi_{\beta}(y\sqrt{\rho_n})\leq 1\},\qquad a_n:=1\vee(a_n'-n^{-1}).
\end{align*}
Since $\varphi_{\beta}(\varepsilon)\rightarrow 0$ as $\varepsilon\rightarrow 0$, for any $C>0$ there exists a positive number $\varepsilon_0$ such that $\varepsilon\leq\varepsilon_0$ implies $\varphi_{\beta}(\varepsilon)\leq C^{-p}$. Moreover, since $\rho_n\rightarrow 0$ as $n\rightarrow\infty$, there exists a positive integer $n_0$ such that $n\geq n_0$ implies $\rho_n<\varepsilon_0^2/C^2$, hence $C^p\varphi_{\beta}(C\sqrt{\rho_n})\leq 1$. Therefore we have $a_n'\rightarrow\infty$, hence $a_n\rightarrow\infty$. Furthermore, for sufficiently large $n$ $a_n<a_n'$, hence $a_n^p\varphi_{\beta}(a_n\sqrt{\rho_n})\leq 1$. Therefore, if we put $\varepsilon_n:=a_n\sqrt{\rho_n}\wedge 1$, we obtain $(\ref{eps2})$ and $\sqrt{\rho_n}/\varepsilon_n\rightarrow 0$. Moreover, since $\varphi_{\beta}(\varepsilon_n)\leq\varphi_\beta(a_n\sqrt{\rho_n})\leq a_n^{-p}\rightarrow 0$, we complete the proof. 
\end{proof}

%%%%%%%%%%%%%%%%%%%%%%%%%%%%%%%%%%%%%%%%%%%%%%%%%%%%%%%%%%
%                     Auxiliary notation
%%%%%%%%%%%%%%%%%%%%%%%%%%%%%%%%%%%%%%%%%%%%%%%%%%%%%%%%%%

We introduce some auxiliary notation. We introduce aa auxiliary sequence $(\varepsilon_n)$ of numbers in $(0,1]$ such that
\begin{equation}\label{eps}
\limsup_{n\rightarrow 0}\frac{\sqrt{\rho_n}}{\varepsilon_n}<\infty,
\end{equation}
and we set $E_n:=\left\{x\in E\big|\psi(x)>\varepsilon_n\right\}$. We will more specify the sequence $(\varepsilon_n)$ later. Furthermore, we put
\begin{gather*}
D^{l}:=\kappa'(\delta^{l})\star\mu^{l},\qquad
X^{l}:=A^{l}+M^{l},\qquad
Y^{l}:=X^{l}+D^{l},\qquad
\tilde{N}^{l}:=1_{E_n}\star\mu^{l},\\
\mathfrak{L}^{l}:=\kappa(\delta^{l})1_{E_n^c}\star(\mu^{l}-\nu^{l})-\kappa(\delta^{l})1_{E_n}\star\nu^{l},\qquad
\Xi^{l}:=\kappa(\delta^{l})1_{E_n^c}\star(\mu^{l}-\nu^{l})\qquad
\Theta^{l}:=\kappa(\delta^{l})1_{E_n}\star\nu^{l}
\end{gather*}
for each $l=1,2$ and
\begin{align*}
\mathsf{X}^{1}_{S^i}=X^{1}_{S^i}+U^{1}_{S^i},\qquad
\mathsf{X}^{2}_{T^j}=X^{2}_{T^j}+U^{2}_{T^j},\qquad
\mathsf{Y}^{1}_{S^i}=Y^{1}_{S^i}+U^{1}_{S^i},\qquad
\mathsf{Y}^{2}_{T^j}=Y^{2}_{T^j}+U^{2}_{T^j}
\end{align*}
for each $i,j\in\mathbb{Z}_+$.

%%%%%%%%%%%%%%%%%%%%%%%%%%%%%%%%%%%%%%%%%%%%%%%%%%%%%%%%%%%%%%%%%%%
%                          Lemma barhatN*
%%%%%%%%%%%%%%%%%%%%%%%%%%%%%%%%%%%%%%%%%%%%%%%%%%%%%%%%%%%%%%%%%%%%

\begin{lem}\label{barhatN*}
Suppose $[\mathrm{SK}_\beta]$ holds for some $\beta\in[0,2]$. Then for any $t>0$ we have
\begin{equation}\label{eqbarhatN*}
\sum_{i=1}^{\infty}\tilde{N}^{1}(\bar{I}^i)_t=O_p(k_n\varepsilon_n^{-\beta}),\qquad
\sum_{j=1}^{\infty}\tilde{N}^{2}(\bar{J}^j)_t=O_p(k_n\varepsilon_n^{-\beta}).
\end{equation}
\end{lem}

\begin{proof}
Since
$E\left[\sum_{i=1}^{\infty}\tilde{N}^{1}(\bar{I}^i)_t\right]
=\sum_{i=1}^{\infty}E[\hat{\Lambda}^{1}(\bar{I}^i)_t]
\lesssim \varepsilon_n^{-\beta}\sum_{i=1}^{\infty}E\left[|\bar{I}^i(t)|\right]\leq\varepsilon_n^{-\beta}k_n t,$
we obtain the first equation of $(\ref{barhatN*})$. By symmetry we also obtain the second one.
\end{proof}

%%%%%%%%%%%%%%%%%%%%%%%%%%%%%%%%%%%%%%%%%%%%%%%%%%%%%%%%%%%%%%%%
%                           Lemma barsmalljumpsum
%%%%%%%%%%%%%%%%%%%%%%%%%%%%%%%%%%%%%%%%%%%%%%%%%%%%%%%%%%%%%%%%%

\begin{lem}\label{barsmalljumpsum}
Suppose $[\mathrm{SK}_\beta]$ holds for some $\beta\in[0,2]$. Then for any $t>0$ we have
\begin{equation}\label{barsjs}
E\left[\sum_{i=1}^{\infty}|\bar{\mathfrak{L}}^{1}(\widehat{\mathcal{I}})^i_t|^2\right]\lesssim\varepsilon_n^{2-\beta}\varphi_{\beta}(\varepsilon_n)k_n,\qquad
E\left[\sum_{j=1}^{\infty}|\bar{\mathfrak{L}}^{2}(\widehat{\mathcal{J}})^j_t|^2\right]\lesssim\varepsilon_n^{2-\beta}\varphi_{\beta}(\varepsilon_n)k_n.
\end{equation}
\end{lem}

%%%%%%%%%%%%%%%%%%%%%%%%%%%%%%%%%%%%%%%%%%%%%%%%%%%%%%%%%%%%%%%%%%%%%
%                           proof of lemma smalljumpsum
%%%%%%%%%%%%%%%%%%%%%%%%%%%%%%%%%%%%%%%%%%%%%%%%%%%%%%%%%%%%%%%%%%%%%

\begin{proof}
Since $E[|\bar{\Xi}^{(1)}(\widehat{\mathcal{I}})^i_t|^2]=E[\langle\bar{\Xi}^{(1)}(\widehat{\mathcal{I}})^i\rangle_t]\lesssim\varphi_2(\varepsilon_n)E\left[\sum_{p=1}^{k_n-1}|\widehat{I}^{i+p}(t)|\right]$ and $E[|\bar{\Theta}(\widehat{\mathcal{I}})^i_t|^2]\lesssim\varepsilon_n^{-2(\beta-1)_+}E\left[|\bar{I}^i(t)|^2\right],$ we have
\begin{equation*}
E\left[\sum_{i=1}^{\infty}|\bar{\mathfrak{L}}^{1}(\widehat{\mathcal{I}})^i_t|^2\right]\lesssim \varphi_2(\varepsilon_n)k_n t+\varepsilon_n^{-2(\beta-1)_+}k_n\bar{r}_n\cdot k_n t .
\end{equation*}
Since $\varphi_2(\varepsilon_n)\leq(\varepsilon_n)^{2-\beta}\varphi_{\beta}(\varepsilon_n)$, $(\beta-1)_+\leq\beta/2$ and $k_n\bar{r}_n=o((\varepsilon_n)^2)$ by $(\ref{threshold2})$ and $(\ref{eps})$, we obtain the first equation of $(\ref{barsjs})$. By symmetry we obtain the second equation of $(\ref{barsjs})$.
\end{proof}

%%%%%%%%%%%%%%%%%%%%%%%%%%%%%%%%%%%%%%%%%%%%%%%%%%%%%%%%%%%%%%
%                     lemma barlarge
%%%%%%%%%%%%%%%%%%%%%%%%%%%%%%%%%%%%%%%%%%%%%%%%%%%%%%%%%%%%%%

\begin{lem}\label{barlarge}
Suppose $[\mathrm{ST}]$ and $[\mathrm{SK}_\beta]$ hold for some $\beta\in[0,2]$. Then for any $t>0$ we have
\begin{equation}\label{eqbarlarge}
\sum_{i=1}^{\infty}1_{\{|\bar{L}^{1}(\widehat{\mathcal{I}})^i_t|^2>c\varrho_n^{1}[i]\}}=o_p(k_n\rho_n^{-\beta/2}),\qquad
\sum_{j=1}^{\infty}1_{\{|\bar{L}^{2}(\widehat{\mathcal{J}})^j_t|^2>c\varrho_n^{(2)}[j]\}}=o_p(k_n\rho_n^{-\beta/2}).
\end{equation}
\end{lem}

\begin{proof}
Combining Lemma \ref{barhatN*} with Lemma \ref{barsmalljumpsum}, we obtain
\begin{align*}
\sum_{i=1}^{\infty}1_{\{|\bar{L}^1(\widehat{\mathcal{I}})^i_t|^2>c\varrho_n^{1}[i]\}}
\leq\frac{K_0}{\rho_n}\sum_{i=1}^{\infty}|\bar{\mathfrak{L}}^1(\widehat{\mathcal{I}})^i_t|^2+\sum_{i=1}^{\infty}\tilde{N}^{1}(\bar{I}^i)_t
\lesssim\frac{\varepsilon_n^2}{\rho_n}\varphi_{\beta}(\varepsilon_n)\varepsilon_n^{-\beta}k_n+\varepsilon_n^{-\beta}k_n,
\end{align*}
hence Lemma \ref{epsspecify} with $p=1$ yields the first equation of $(\ref{eqbarlarge})$. By symmetry we also obtain the second one. 
\end{proof}

%%%%%%%%%%%%%%%%%%%%%%%%%%%%%%%%%%%%%%%%%%%%%%%%%%%%%%%%%%%%%%%%%%
%                      lemma jumpsum
%%%%%%%%%%%%%%%%%%%%%%%%%%%%%%%%%%%%%%%%%%%%%%%%%%%%%%%%%%%%%%%%%%

\begin{lem}\label{matanian}
Suppose $[\mathrm{ST}]$ and $[\mathrm{SK}_\beta]$ hold for some $\beta\in[0,2]$. Then for any $t>0$ we have
\begin{equation}\label{eqmatanian}
\sum_{i=1}^{\infty}|\bar{L}^{1}(\widehat{\mathcal{I}})^i_t|^21_{\{|\bar{L}^{1}(\widehat{\mathcal{I}})^i_t|^2\leq c\varrho_n^{1}[i]\}}=o_p\left(k_n\rho_n^{1-\beta/2}\right),\qquad
\sum_{j=1}^{\infty}|\bar{L}^{2}(\widehat{\mathcal{J}})^j_t|^21_{\{|\bar{L}^{2}(\widehat{\mathcal{J}})^j_t|^2\leq c\varrho_n^{2}[j]\}}=o_p\left(k_n\rho_n^{1-\beta/2}\right).
\end{equation}
\end{lem}

\begin{proof}
Since
\begin{align*}
\sum_{i=1}^{\infty}|\bar{L}^{1}(\widehat{\mathcal{I}})^i_t|^21_{\{|\bar{L}^{1}(\widehat{\mathcal{I}})^i_t|^2\leq c\varrho_n^{1}[i]\}}
\leq &\sum_{i=1}^{\infty}|\bar{L}^{1}(\widehat{\mathcal{I}})^i_t|^21_{\{ N^{1}(\bar{I}^i)_t= 0\}}
+\sum_{i=1}^{\infty}|\bar{L}^{1}(\widehat{\mathcal{I}})^i_t|^21_{\{|\bar{L}^{1}(\widehat{\mathcal{I}})^i_t|^2\leq c\varrho_n^{1}[i], N^{1}(\bar{I}^i)_t\neq 0\}}\\
\lesssim &\sum_{i=1}^{\infty}|\bar{\mathfrak{L}}^{1}(\widehat{\mathcal{I}})^i_t|^2
+\rho_n\sum_{i=1}^{\infty}1_{\{ N^{1}(\bar{I}^i)_t\neq 0\}},
\end{align*}
Lemma \ref{barhatN*} and Lemma \ref{barsmalljumpsum} yield
$\sum_{i=1}^{\infty}|\bar{L}^{1}(\widehat{\mathcal{I}})^i_t|^21_{\{|\bar{L}^{1}(\widehat{\mathcal{I}})^i_t|^2\leq c\varrho_n^{1}[i]\}}=O_p\left(\{(\varepsilon_n)^2\varphi_{\beta}(\varepsilon_n)+\rho_n\}k_n\varepsilon_n^{-\beta}\right),$
hence by Lemma \ref{epsspecify} with $p=2$ we obtain the first equation of $(\ref{eqmatanian})$. By symmetry we also obtain the second equation of $(\ref{eqmatanian})$.
\end{proof}

%%%%%%%%%%%%%%%%%%%%%%%%%%%%%%%%%%%%%%%%%%%%%%%%%%%%%%%%%%%%%%%%%%%%%%%%%
%                      IA jump and noise Lemma
%%%%%%%%%%%%%%%%%%%%%%%%%%%%%%%%%%%%%%%%%%%%%%%%%%%%%%%%%%%%%%%%%%%%%%%%%%
%%%%%%%%%%%%%%%%%%%%%%%%%%%%%%%%%%%%%%%%%%%%%%%%%%%%%%%%%%%%%%%%%%%%%%%%%% 

Next we prove some lemmas which deal with the events that the noise part corrects the effect of small jumps. Let $\eta_n:=(k_n\rho_n)^{-1}$. Set
\begin{align*}
\overline{\mathsf{Y}}^1(\widehat{\mathcal{I}})^i=\sum_{p=1}^{k_n-1}g^n_p\left(\mathsf{Y}^1_{\widehat{S}^{i+p}}-\mathsf{Y}^1_{\widehat{S}^{i+p-1}}\right),\qquad
\overline{\mathsf{Y}}^2(\widehat{\mathcal{J}})^j=\sum_{q=1}^{k_n-1}g^n_q\left(\mathsf{Y}^2_{\widehat{T}^{j+q}}-\mathsf{Y}^2_{\widehat{T}^{j+q-1}}\right)
\end{align*}
for each $i,j\in\mathbb{Z}_+$.

%%%%%%%%%%%%%%%%%%%%%%%%%%%%%%%%%%%%%%%%%%%%%%%%%%%%%%%%%%%%%%%%%%%%%%%%%%%%%
%                        lemma LX
%%%%%%%%%%%%%%%%%%%%%%%%%%%%%%%%%%%%%%%%%%%%%%%%%%%%%%%%%%%%%%%%%%%%%%%%%%%%%%

\begin{lem}\label{lemLX}
Let $c_1$ and $c_2$ be two positive numbers. Suppose $[\mathrm{ST}]$, $[\mathrm{SA}4]$ and $[\mathrm{SK}_\beta]$ hold for some $\beta\in[0,2]$. Suppose also that $[\mathrm{SN}^\flat_r]$ holds for some $r\in(2,\infty)$. Then for any $t>0$ we have
\begin{align}
\sum_{i:\widehat{S}^{i+k_n}\leq t} |\overline{\mathsf{Y}}^{1}(\widehat{\mathcal{I}})^i|^21_{\{|\overline{\mathsf{Y}}^{1}(\widehat{\mathcal{I}})^i|^2>c_1\varrho_n^{1}[i], |\bar{L}^{1}(\widehat{\mathcal{I}})^i_t|^2>c_2\varrho_n^{1}[i] \}}
=o_p\left(\eta_n^{\frac{r-2}{2}}\rho_n^{-\beta/2}\right),\label{eqLX1}\\
\sum_{j:\widehat{T}^{j+k_n}\leq t} |\overline{\mathsf{Y}}^{2}(\widehat{\mathcal{J}})^j|^21_{\{|\overline{\mathsf{Y}}^{2}(\widehat{\mathcal{J}})^j|^2>c_1\varrho_n^{2}[j], |\bar{L}^{2}(\widehat{\mathcal{J}})^j_t|^2>c_2\varrho_n^{2}[j] \}}
=o_p\left(\eta_n^{\frac{r-2}{2}}\rho_n^{-\beta/2}\right).\label{eqLX2}
\end{align}

\end{lem}

\begin{proof}
Consider $(\ref{eqLX1})$. We decompose the target quantity as
\begin{align*}
\sum_{i:\widehat{S}^{i+k_n}\leq t} |\overline{\mathsf{Y}}^{1}(\widehat{\mathcal{I}})^i|^21_{\{|\overline{\mathsf{Y}}^{1}(\widehat{\mathcal{I}})^i|^2>c_1\varrho_n^{1}[i], |\bar{L}^{1}(\widehat{\mathcal{I}})^i_t|^2>c_2\varrho_n^{1}[i] \}}
&=\sum_{i:\widehat{S}^{i+k_n}\leq t} |\overline{\mathsf{Y}}^{1}(\widehat{\mathcal{I}})^i|^21_{\{|\overline{\mathsf{Y}}^{1}(\widehat{\mathcal{I}})^i|^2>c_1\varrho_n^{1}[i], |\bar{L}^{1}(\widehat{\mathcal{I}})^i_t|^2>c_2\varrho_n^{1}[i], N^{1}(\bar{I}^i)_t\neq 0 \}}\\
&\hphantom{=}+\sum_{i:\widehat{S}^{i+k_n}\leq t} |\overline{\mathsf{X}}^{1}(\widehat{\mathcal{I}})^i|^21_{\{|\overline{\mathsf{X}}^{1}(\widehat{\mathcal{I}})^i|^2>c_1\varrho_n^{1}[i], |\bar{L}^{1}(\widehat{\mathcal{I}})^i_t|^2>c_2\varrho_n^{1}[i], N^{1}(\bar{I}^i)_t=0 \}}\\
&=:A_{1,t}+A_{2,t}.
\end{align*}
By Lemma \ref{lembasic} we obtain
$A_{1,t}=o_p\left(\eta_n^{\frac{r-2}{2}}\rho_n^{-\beta/2}\right).$
On the other hand, on $\{|\overline{\mathsf{X}}^{1}(\widehat{\mathcal{I}})^i|^2>c_1\varrho_n^{1}[i]\}$ we have
\begin{align*}
|\overline{\zeta}^{1}(\widehat{\mathcal{I}})^i|
\geq |\overline{\mathsf{X}}^{1}(\widehat{\mathcal{I}})^i|-|\bar{X}^{1}(\widehat{\mathcal{I}})^i_t|-|\widetilde{\underline{X}^1}(\widehat{\mathcal{I}})^i_t|
>\sqrt{\rho_n}\left(\sqrt{\frac{c_1}{K_0}}-2K\sqrt{\frac{2 k_n\bar{r}_n|\log b_n|}{\rho_n}}\right)
\end{align*}
by [SA4], [$\mathrm{ST}$] and $(\ref{absmod2})$. Hence by $(\ref{threshold2})$ we have
\begin{align*}
A_{2,t}
\leq &\sum_{i:\widehat{S}^{i+k_n}\leq t} |\overline{\mathsf{X}}^{1}(\widehat{\mathcal{I}})^i|^21_{\{|\bar{L}^{1}(\widehat{\mathcal{I}})^i_t|^2>\varrho_n^{1}[i], |\overline{\zeta}^{1}(\widehat{\mathcal{I}})^i|^2>c_1\rho_n/4 K_0\}}\\
\leq &2\left(\frac{2 K_0}{c_1\rho_n}\right)^{\frac{r}{2}}\sum_{i:\widehat{S}^{i+k_n}\leq t} |\bar{X}^{1}(\widehat{\mathcal{I}})^i_t|^2|\overline{\zeta}^{1}(\widehat{\mathcal{I}})^i|^r 1_{\{|\bar{L}^{1}(\widehat{\mathcal{I}})^i_t|^2>c_2\varrho_n^{1}[i]\}}
+2\left(\frac{2 K_0}{c_1\rho_n}\right)^{\frac{r-2}{2}}\sum_{i:\widehat{S}^{i+k_n}\leq t} |\overline{\zeta}^{1}(\widehat{\mathcal{I}})^i|^r 1_{\{|\bar{L}^{1}(\widehat{\mathcal{I}})^i_t|^2>c_2\varrho_n^{1}[i]\}},
\end{align*}
and thus $(\ref{threshold2})$ and $(\ref{absmod2})$ imply that
$A_{2,t}
\lesssim(\rho_n)^{-\frac{r-2}{2}}\sum_{i:\widehat{S}^{i+k_n}\leq t} |\overline{\zeta}^{1}(\widehat{\mathcal{I}})^i|^r 1_{\{|\bar{L}^{1}(\widehat{\mathcal{I}})^i_t|^2>c_2\varrho_n^{1}[i]\}}.$
By Lemma \ref{moment} we obtain
$E_{0}\left[A_{2,t}\right]\lesssim
\eta_n^{\frac{r-2}{2}}k_n^{-1}\sum_{i=1}^{\infty} 1_{\{\bar{L}^{1}(\widehat{\mathcal{I}})^i_t|^2>\varrho_n^{1}[i]\}},$
hence Lemma \ref{barlarge} yield
$A_{2,t}
=o_p\left(\eta_n^{\frac{r-2}{2}}\rho_n^{-\beta/2}\right).$
Consequently, we obtain $(\ref{eqLX1})$. By symmetry we also obtain $(\ref{eqLX2})$, and thus we complete the proof.
\end{proof}

%%%%%%%%%%%%%%%%%%%%%%%%%%%%%%%%%%%%%%%%%%%%%%%%%%%%%%%%%%%%%%%%%%%%%%%%%%
%                         lemma LZ
%%%%%%%%%%%%%%%%%%%%%%%%%%%%%%%%%%%%%%%%%%%%%%%%%%%%%%%%%%%%%%%%%%%%%%%%%%

\begin{lem}\label{lemLZ}
Suppose $[\mathrm{ST}]$, $[\mathrm{SA}4]$ and $[\mathrm{SK}_\beta]$ hold for some $\beta\in[0,2]$. Suppose also that $[\mathrm{SN}^\flat_r]$ holds for some $r\in(2,\infty)$. Then for any $t>0$ we have
\begin{align}
\sum_{i:\widehat{S}^{i+k_n}\leq t} |\bar{L}^{1}(\widehat{\mathcal{I}})^i_t|1_{\{|\overline{\mathsf{Z}}^{1}(\widehat{\mathcal{I}})^i|^2\leq\varrho_n^{1}[i], |\bar{L}^{1}(\widehat{\mathcal{I}})^i_t|^2>4\varrho_n^{1}[i] \}}
=o_p\left( k_n\eta_n^{r/4}\rho_n^{-\beta/4}\right),\label{eqLZ1}\\
\sum_{j:\widehat{T}^{j+k_n}\leq t} |\bar{L}^{2}(\widehat{\mathcal{J}})^j_t|1_{\{|\overline{\mathsf{Z}}^{2}(\widehat{\mathcal{J}})^j|^2\leq\varrho_n^{2}[j], |\bar{L}^{2}(\widehat{\mathcal{J}})^j_t|^2>4\varrho_n^{2}[j] \}}
=o_p\left( k_n\eta_n^{r/4}\rho_n^{-\beta/4}\right).\label{eqLZ2}
\end{align}
\end{lem}

\begin{proof}
Consider $(\ref{eqLZ1})$. We decompose the target quantity as
\begin{align*}
&\hphantom{=}\sum_{i:\widehat{S}^{i+k_n}\leq t} |\bar{L}^{1}(\widehat{\mathcal{I}})^i_t|1_{\{|\overline{\mathsf{Z}}^{1}(\widehat{\mathcal{I}})^i|^2\leq\varrho_n^{1}[i], |\bar{L}^{1}(\widehat{\mathcal{I}})^i_t|^2>4\varrho_n^{1}[i] \}}\\
&=\sum_{i:\widehat{S}^{i+k_n}\leq t} |\bar{L}^{1}(\widehat{\mathcal{I}})^i_t|\left(1_{\{|\overline{\zeta}^{1}(\widehat{\mathcal{I}})^i|^2>\varrho_n^{1}[i]/4, |\overline{\mathsf{Z}}^{1}(\widehat{\mathcal{I}})^i|^2\leq\varrho_n^{1}[i], |\bar{L}^{1}(\widehat{\mathcal{I}})^i_t|^2>4\varrho_n^{1}[i] \}}+1_{\{|\overline{\zeta}^{1}(\widehat{\mathcal{I}})^i|^2\leq\varrho_n^{1}[i]/4, |\overline{\mathsf{Z}}^{1}(\widehat{\mathcal{I}})^i|^2\leq\varrho_n^{1}[i], |\bar{L}^{1}(\widehat{\mathcal{I}})^i_t|^2>4\varrho_n^{1}[i] \}}\right)\\
&=:\mathbb{B}_{1,t}+\mathbb{B}_{2,t}.
\end{align*}
By the Schwarz inequality we have
{\small \begin{align*}
&\mathbb{B}_{1,t}\leq
\sum_{i:\widehat{S}^{i+k_n}\leq t} |\bar{L}^{1}(\widehat{\mathcal{I}})^i_t|1_{\{|\overline{\zeta}^1(\widehat{\mathcal{I}})^i|^2>\varrho_n^{1}[i]/4, |\bar{L}^{1}(\widehat{\mathcal{I}})^i_t|^2>4\varrho_n^{1}[i] \}}\\
\leq &\left\{\sum_{i:\widehat{S}^{i+k_n}\leq t}|\bar{L}^{1}(\widehat{\mathcal{I}})^i_t|^2 1_{\{|\overline{\zeta}^{1}(\widehat{\mathcal{I}})^i|^2>\varrho_n^{1}[i]/4\}}\right\}^{1/2}\left\{\sum_{i:\widehat{S}^{i+k_n}\leq t}1_{\{|\bar{L}^{1}(\widehat{\mathcal{I}})^i_t|^2>4\varrho_n^{1}[i] \}}\right\}^{1/2}.
\end{align*}}
[ST] and Lemma \ref{moment} yield
\begin{align*}
E\left[\sum_{i:\widehat{S}^{i+k_n}\leq t}|\bar{L}^{1}(\widehat{\mathcal{I}})^i_t|^2 1_{\{|\overline{\zeta}^{1}(\widehat{\mathcal{I}})^i|^2>\varrho_n^{1}[i]/4\}}\right]
\lesssim \rho_n^{-r/2}\sum_{i=1}^{\infty}E\left[|\bar{L}^{1}(\widehat{\mathcal{I}})^i_t|^2 E_0[|\overline{\zeta}^{1}(\widehat{\mathcal{I}})^i|^r]\right]
\lesssim \eta_n^{r/2}\sum_{i=1}^{\infty}E[|\bar{I}^i(t)|]
\leq \eta_n^{r/2}k_n.
\end{align*}
Combining this with Lemma \ref{barlarge}, we obtain
$\mathbb{B}_{1,t}
=o_p\left( k_n\eta_n^{r/4}\rho_n^{-\beta/4}\right).$
On the other hand, on $\{|\overline{\mathsf{Z}}^{1}(\widehat{\mathcal{I}})^i|^2\leq\varrho_n^{1}[i], |\bar{L}^{1}(\widehat{\mathcal{I}})^i_t|^2>4\varrho_n^{1}[i] \}$ we have $|\overline{\mathsf{Y}}^{1}(\widehat{\mathcal{I}})^i|\geq|\bar{L}^{1}(\widehat{\mathcal{I}})^i_t|-|\overline{\mathsf{Z}}^{1}(\widehat{\mathcal{I}})^i|>\sqrt{\varrho_n^{1}[i]}$. Moreover, by $(\ref{absmod2})$ we have
\begin{align*} 
|\bar{D}^{1}(\widehat{\mathcal{I}})^i_t|
\geq|\overline{\mathsf{Y}}^{1}(\widehat{\mathcal{I}})^i|-|\bar{X}^{1}(\widehat{\mathcal{I}})^i_t|
-|\overline{\zeta}^{1}(\widehat{\mathcal{I}})^i|-|\widetilde{\underline{X}^1}(\widehat{\mathcal{I}})^i_t|>\sqrt{\rho_n}\left(\frac{1}{2\sqrt{K_0}}-2\sqrt{\frac{2 k_n\bar{r}_n|\log b_n|}{\rho_n}}\right),
\end{align*}
on $\{|\overline{\zeta}^{1}(\widehat{\mathcal{I}})^i|^2\leq\varrho_n^{1}[i]/4, |\overline{\mathsf{Y}}^{1}(\widehat{\mathcal{I}})^i|^2>\varrho_n^{1}[i],\widehat{S}^{i+k_n}\leq t\}$, and thus $(\ref{threshold2})$ yields $|\bar{D}^{1}(\widehat{\mathcal{I}})^i_t|>0$, hence $N^{1}(\bar{I}^i)\neq 0$. Therefore, we obtain
\begin{equation}\label{estLZ}
\mathbb{B}_{2,t}\leq\sum_{i:\widehat{S}^{i+k_n}\leq t} |\bar{L}^{1}(\widehat{\mathcal{I}})^i_t|1_{\{|\bar{L}^{1}(\widehat{\mathcal{I}})^i_t|^2>4\varrho_n^{1}[i], N^{1}(\bar{I}^i)_t\neq 0 \}},
\end{equation}
and thus Lemma \ref{lembasic} yields
$\mathbb{B}_{2,t}
=o_p\left( k_n\eta_n^{r/4}\rho_n^{-\beta/4}\right).$
Consequently, we obtain $(\ref{eqLZ1})$. By symmetry we also obtain $(\ref{eqLZ2})$, and thus we complete the proof.
\end{proof}

%%%%%%%%%%%%%%%%%%%%%%%%%%%%%%%%%%%%%%%%%%%%%%%%%%%%%%%%%%%%%%%%%%%%%%%%%%
%                         lemma L2Z
%%%%%%%%%%%%%%%%%%%%%%%%%%%%%%%%%%%%%%%%%%%%%%%%%%%%%%%%%%%%%%%%%%%%%%%%%%

\begin{lem}\label{lemL2Z}
Suppose $[\mathrm{ST}]$, $[\mathrm{SA}4]$ and $[\mathrm{SK}_2]$ hold. Suppose also that $[\mathrm{SN}^\flat_r]$ for some $r\in(2,\infty)$. Then for any $t>0$ we have
\begin{align}
\sum_{i:\widehat{S}^{i+k_n}\leq t} |\bar{L}^{1}(\widehat{\mathcal{I}})^i_t|^21_{\{|\overline{\mathsf{Z}}^{1}(\widehat{\mathcal{I}})^i|^2\leq\varrho_n^{1}[i], |\bar{L}^{1}(\widehat{\mathcal{I}})^i_t|^2>4\varrho_n^{1}[i] \}}
=O_p\left( k_n\eta_n^{r/2}\right),\label{eqL2Z1}\\
\sum_{j:\widehat{T}^{j+k_n}\leq t} |\bar{L}^{2}(\widehat{\mathcal{J}})^j_t|^21_{\{|\overline{\mathsf{Z}}^{2}(\widehat{\mathcal{J}})^j|^2\leq\varrho_n^{2}[j], |\bar{L}^{2}(\widehat{\mathcal{J}})^j_t|^2>4\varrho_n^{2}[j] \}}
=O_p\left( k_n\eta_n^{r/2}\right).\label{eqL2Z2}
\end{align}
\end{lem}

\begin{proof}
Consider $(\ref{eqL2Z1})$. We decompose the target quantity as
\begin{align*}
&\hphantom{=}\sum_{i:\widehat{S}^{i+k_n}\leq t} |\bar{L}^{1}(\widehat{\mathcal{I}})^i_t|^21_{\{|\overline{\mathsf{Z}}^{1}(\widehat{\mathcal{I}})^i|^2\leq\varrho_n^{1}[i], |\bar{L}^{1}(\widehat{\mathcal{I}})^i_t|^2>4\varrho_n^{1}[i] \}}\\
&=\sum_{i:\widehat{S}^{i+k_n}\leq t} |\bar{L}^{1}(\widehat{\mathcal{I}})^i_t|^2\left(1_{\{|\overline{\zeta}^{1}(\widehat{\mathcal{I}})^i|^2>\varrho_n^{1}[i]/4, |\overline{\mathsf{Z}}^{1}(\widehat{\mathcal{I}})^i|^2\leq\varrho_n^{1}[i], |\bar{L}^{1}(\widehat{\mathcal{I}})^i_t|^2>4\varrho_n^{1}[i] \}}+1_{\{|\overline{\zeta}^{1}(\widehat{\mathcal{I}})^i|^2\leq\varrho_n^{1}[i]/4, |\overline{\mathsf{Z}}^{1}(\widehat{\mathcal{I}})^i|^2\leq\varrho_n^{1}[i], |\bar{L}^{1}(\widehat{\mathcal{I}})^i_t|^2>4\varrho_n^{1}[i] \}}\right)\\
&=:\Gamma_{1,t}+\Gamma_{2,t}.
\end{align*}
An argument similar to that in the proof of $(\ref{estLZ})$ yields
$\Gamma_{2,t}=o_p\left( k_n\eta_n^{r/2}\right).$
On the other hand, Lemma \ref{moment} yields
\begin{align*}
E[\Gamma_{1,t}]\lesssim
\rho_n^{-r/2}\sum_{i=1}^{\infty} E\left[|\bar{L}^{1}(\widehat{\mathcal{I}})^i_t|^2 E_0[|\overline{\zeta}^{1}(\widehat{\mathcal{I}})^i|^r]\right]
\lesssim\eta_n^{r/2}\sum_{i=1}^{\infty}E\left[|\bar{I}^{i}(t)|\right]
\leq\eta_n^{r/2}k_n t,
\end{align*}
hence we obtain $(\ref{eqL2Z1})$.
By symmetry we also obtain $(\ref{eqL2Z2})$, and thus we complete the proof.
\end{proof}

%%%%%%%%%%%%%%%%%%%%%%%%%%%%%%%%%%%%%%%%%%%%%%%%%%%%%%%%%%%%%%%%%%%%%%%%%%%
%                      lemma barK
%%%%%%%%%%%%%%%%%%%%%%%%%%%%%%%%%%%%%%%%%%%%%%%%%%%%%%%%%%%%%%%%%%%%%%%%%%%

Let $\bar{K}^{ij}_t=1_{\{\bar{I}^i(t)\cap\bar{J}^j(t)\neq\emptyset\}}$ for each $i,j\in\mathbb{Z}_+$ and $t\in\mathbb{R}_+$.
\begin{lem}\label{barKssp}
Suppose that $[\mathrm{A}2]$ and $[\mathrm{SA}4]$ are satisfied. Let $\tau$ be a $\mathbf{G}^{(n)}$-stopping time.  Then $H\bar{K}^{i j}_\tau$ is $\mathcal{F}_{\widehat{S}^i\wedge\widehat{T}^j}$-measurable for any $i,j\in\mathbb{Z}_+$, provided that $H$ is a $\mathcal{G}^{(n)}_{\bar{R}^\vee(i,j)}$-measurable random variable.
\end{lem}

\begin{proof}
Let $B=\left\{\bar{I}^i(\tau)\cap\bar{J}^j(\tau)\neq\emptyset\right\}$. It is sufficient to show that $A\cap B\cap C\in\mathcal{F}_u$ for any $u\in\mathbb{R}_+$, where $A\in\mathcal{G}^{(n)}_{\bar{R}^\vee(i,j)}$ and  $C=\{\widehat{S}^i\wedge\widehat{T}^j\leq u\}$. On $B$ we have $\bar{R}^\vee(i,j)-\widehat{S}^i\wedge\widehat{T}^j\leq|\bar{I}^i|\vee|\bar{J}^j|\vee(\hat{S}^{i+k_n}-\hat{T}^j)\vee(\hat{T}^{j+k_n}-\hat{S}^i)\leq k_n\bar{r}_n,$ hence
$\bar{R}^\vee(i,j)
=\{\bar{R}^\vee(i,j)-\widehat{S}^i\wedge\widehat{T}^j\}
+\widehat{S}^i\wedge\widehat{T}^j
\leq\widehat{S}^i\wedge\widehat{T}^j+k_n\bar{r}_n,$
and thus we have
$B\cap C=B\cap C\cap\{ \bar{R}^\vee(i,j)\leq u+k_n\bar{r}_n\}.$
Since $A,B\in\mathcal{G}^{(n)}_{\bar{R}^\vee(i,j)}$, we have $A\cap B\cap \{ \bar{R}^\vee(i,j)\leq u+k_n\bar{r}_n\}\in\mathcal{G}_{u+k_n\bar{r}_n}$,
however, $\mathcal{G}_{u+k_n\bar{r}_n}=\mathcal{F}_{(u+k_n\bar{r}_n-b_n^{\xi-\frac{1}{2}})_+}\subset\mathcal{F}_u$ by $(\ref{spp})$. This together with the fact that $C\in\mathcal{F}_u$ implies $A\cap B\cap C\in\mathcal{F}_u$.
\end{proof}

%%%%%%%%%%%%%%%%%%%%%%%%%%%%%%%%%%%%%%%%%%%%%%%%%%%%%%%%%%%%%%%%%%%%%%%%%%%%%%
%                      lemma XiM
%%%%%%%%%%%%%%%%%%%%%%%%%%%%%%%%%%%%%%%%%%%%%%%%%%%%%%%%%%%%%%%%%%%%%%%%%%%%%%

\begin{lem}\label{XiM}
Suppose $[\mathrm{SC}1]$-$[\mathrm{SC}2]$, $[\mathrm{A}2]$, $[\mathrm{SA}4]$, $[\mathrm{SA}6]$, $[\mathrm{SK}_2]$ and $[\mathrm{SN}^\flat_2]$ hold. Then for any $t>0$, there exists a positive constant $K$ independent of both $n$ and $(\varepsilon_n)$ such that
\begin{align}
E\left[\sup_{0\leq s\leq t}\left|\sum_{i,j:\bar{R}^\vee(i,j)\leq s}\bar{\Xi}^{1}(\widehat{\mathcal{I}})^i_s\bar{M}^{2}(\widehat{\mathcal{J}})^j_s\bar{K}^{i j}\right|\right]\leq K k_n^2\sqrt{\varphi_2(\varepsilon_n)} b_n^{1/4},\label{eqXiM1}\\
E\left[\sup_{0\leq s\leq t}\left|\sum_{i,j:\bar{R}^\vee(i,j)\leq s}\bar{\Xi}^{1}(\widehat{\mathcal{I}})^i_s\widetilde{\underline{M}^2}(\widehat{\mathcal{J}})^j_s\bar{K}^{i j}\right|\right]\leq K k_n^2\sqrt{\varphi_2(\varepsilon_n)} b_n^{1/4},\label{eqXiM2}\\
E\left[\sup_{0\leq s\leq t}\left|\sum_{i,j:\bar{R}^\vee(i,j)\leq s}\bar{\Xi}^{1}(\widehat{\mathcal{I}})^i_s\overline{\zeta}^{2}(\widehat{\mathcal{J}})^j\bar{K}^{i j}\right|\right]\leq K k_n^2\sqrt{\varphi_2(\varepsilon_n)} b_n^{1/4}.\label{eqXiM3}
\end{align}
\end{lem}

\begin{proof}
First consider $(\ref{eqXiM1})$. since integration by parts and Lemma 4.3 of \cite{Koike2012phy} yield
\begin{equation}\label{IBP} 
\bar{\Xi}^{1}(\widehat{\mathcal{I}})^i_t\bar{M}^{2}(\widehat{\mathcal{J}})^j_t\bar{K}^{ij}_s=\left\{\bar{K}^{ij}_-\bar{\Xi}^{1}(\widehat{\mathcal{I}})^i_-\right\}\bullet\bar{M}^{2}(\widehat{\mathcal{J}})^j_s+\left\{\bar{K}^{ij}_-\bar{M}^{2}(\widehat{\mathcal{J}})^j_-\right\}\bullet\bar{\Xi}^{1}(\widehat{\mathcal{I}})^i_s,
\end{equation}
we can decompose the target quantity as
\begin{align*}
&\hphantom{=}\sum_{i,j:\bar{R}^\vee(i,j)\leq s}\bar{\Xi}^{1}(\widehat{\mathcal{I}})^i_s\bar{M}^{2}(\widehat{\mathcal{J}})^j_s\bar{K}^{i j}\\
&=\sum_{i,j=0}^\infty\left[\left\{\bar{K}^{ij}_-\bar{\Xi}^{1}(\widehat{\mathcal{I}})^i_-\right\}\bullet\bar{M}^{2}(\widehat{\mathcal{J}})^j_s+\left\{\bar{K}^{ij}_-\bar{M}^{2}(\widehat{\mathcal{J}})^j_-\right\}\bullet\bar{\Xi}^{1}(\widehat{\mathcal{I}})^i_s\right]
+\sum_{i,j:\bar{R}^\vee(i,j)>s}\bar{\Xi}^{1}(\widehat{\mathcal{I}})^i_s\bar{M}^{2}(\widehat{\mathcal{J}})^j_s\bar{K}^{i j}_s\\
&=:\Delta_{1,s}+\Delta_{2,s}+\Delta_{3,s}.
\end{align*}
First we estimate $\Delta_{1,s}$. Since $\bar{\Xi}^{1}(\widehat{\mathcal{I}})^i_s=\sum_{p=1}^{k_n-1}g^n_p \Xi^{1}(\widehat{I}^{i+p})_s$ and $\bar{M}^{2}(\widehat{\mathcal{J}})^j_s=\sum_{q=1}^{k_n-1}g^n_q M^{2}(\widehat{J}^{j+q})_s$, we have
\begin{align*}
\Delta_{1,s}
=&\sum_{i,j=0}^{\infty}\sum_{p,q=1}^{k_n-1}g_p^n g_q^n\bar{K}^{ij}_-\Xi^{1}(\widehat{I}^{i+p})_-\widehat{J}^{j+q}_-\bullet M^{2}_s\\
=&\sum_{i,j=0}^{\infty}\sum_{p=i+1}^{i+k_n-1}\sum_{q=j+1}^{j+k_n-1}g_{p-i}^n g_{q-j}^n\bar{K}^{ij}_-\Xi^{1}(\widehat{I}^{i+p})_-\widehat{J}^{j+q}_-\bullet M^{2}_s
=\sum_{p,q=1}^{\infty}\upsilon(p,q)_-\Xi^{1}(\widehat{I}^{p})_-\widehat{J}^{q}_-\bullet M^{2}_s,
\end{align*}
where
$\upsilon(p,q)_s=\sum_{i=(p-k_n+1)_+}^{p-1}\sum_{j=(q-k_n+1)_+}^{q-1}g_{p-i}^n g_{q-j}^n\bar{K}^{i j}_s.$
Hence we have
\begin{align*}
\langle\Delta_{1,\cdot}\rangle_s
=\sum_{q=1}^{\infty}\left[\sum_{p=1}^{\infty}\upsilon(p,q)_-\Xi^{1}(\widehat{I}^{p})_-\right]^2\widehat{J}^{q}_-\bullet \langle M^{2}\rangle_s,
\end{align*}
and thus we obtain
\begin{align*}
E[\langle\Delta_{1,\cdot}\rangle_t ]
\lesssim &\sum_{q=1}^{\infty}\int_0^t E\left[\left\{\sum_{p=1}^{\infty}\upsilon(p,q)_s\Xi^{1}(\widehat{I}^{p})_s\right\}^2\widehat{J}^{q}_s\right]\mathrm{d}s
=\sum_{q=1}^{\infty}\int_0^t \sum_{p=1}^{\infty}E\left[\upsilon(p,q)_s^2\widehat{J}^{q}_s E\left[|\Xi^{1}(\widehat{I}^{p})_s|^2|\mathcal{F}^{(0)}_{\widehat{S}^{p-1}}\right]\right]\mathrm{d}s\\
\lesssim &k_n^{2}\varphi_2(\varepsilon_n)\sum_{p,q=1}^{\infty}\sum_{i=(p-k_n+1)_+}^{p-1}\sum_{j=(q-k_n+1)_+}^{q-1}E\left[\bar{K}^{i j}_t |\widehat{I}^{p}(t)||\widehat{J}^{q}(t)|\right]
\lesssim k_n^4\varphi_2(\varepsilon_n) k_n b_n\lesssim k_n^4\varphi_2(\varepsilon_n) b_n^{1/2}
\end{align*}
by the representation of $\langle M^{2}\rangle$ and $\langle \Xi^{1}\rangle$, [A2], $(\ref{sumbarK})$, [SA6] and $(\ref{window})$. Combining this with the Schwarz and Doob inequalities, we conclude that 
\begin{equation}\label{estDelta}
E\left[\sup_{0\leq s\leq t}|\Delta_{1,s}|\right]\lesssim k_n^2\sqrt{\varphi_2(\varepsilon_n)} b_n^{1/4}.
\end{equation}
Similarly we can also show that $E\left[\sup_{0\leq s\leq t}|\Delta_{2,s}|\right]$ $\lesssim k_n^2\sqrt{\varphi_2(\varepsilon_n)} b_n^{1/4}$. Now we estimate $\Delta_{3,s}$. $(\ref{IBP})$, the Doob inequality, [A2] and the optional sampling theorem, $(\ref{sumbarK})$ and [SA6] imply that
\begin{equation}\label{estXiM}
E\left[\sup_{0\leq s\leq t}\sum_{i,j=0}^\infty|\bar{\Xi}^{1}(\widehat{\mathcal{I}})^i_s\bar{M}^{2}(\widehat{\mathcal{J}})^j_s\bar{K}^{i j}_s|^2\right]
\lesssim\varphi_2(\varepsilon_n)E\left[\sum_{i,j=1}^{\infty}\bar{K}^{i j}|\bar{I}^i(t)||\bar{J}^j(t)|\right]
\lesssim\varphi_2(\varepsilon_n)k_n.
\end{equation}
On the other hand, $(\ref{sumbarK})$ also implies
\begin{equation}\label{estshift}
\sum_{i,j:\bar{R}^\vee(i,j)>s}\bar{K}^{ij}_s\leq 
(2k_n+1)\left[\sum_{i:\widehat{S}^{i+k_n}>s}1+\sum_{j:\widehat{T}^{j+k_n}>s}1\right]
\leq (2k_n+1)k_n.
\end{equation}
Therefore, the Schwarz inequality and $\ref{window})$ yield $E\left[\sup_{0\leq s\leq t}|\Delta_{3,s}|\right]\lesssim\sqrt{\varphi_2(\varepsilon_n)k_n}\cdot k_n\lesssim k_n^2\sqrt{\varphi_2(\varepsilon_n)} b_n^{1/4}.$ Consequently, we conclude that $(\ref{eqXiM1})$ holds. $(\ref{eqXiM2})$ can be shown in a similar manner.

Finally consider $(\ref{eqXiM3})$. Define the process $\mathfrak{Z}^2_t$ by 
$\mathfrak{Z}^2_t=\sqrt{b_n}\sum_{j=1}^\infty\zeta^2_{\widehat{T}^j}1_{\{\widehat{T}^j\leq t\}}.$
Then obviously $\mathfrak{Z}^2_t$ is a purely discontinuous locally square-integrable martingale $\mathcal{B}^{(0)}$ and $\overline{\zeta}^{2}(\widehat{\mathcal{J}})^j=\widetilde{\mathfrak{Z}}^2(\widehat{\mathcal{J}})^j$ on $\{\widehat{T}^{j+k_n}\leq t\}$. On the other hand, since $\Xi^1$ is quasi-left continuous by Theorem I-4.2 of \cite{JS} and for every $j$ $\widehat{T}^j$ is $\mathbf{F}^{(0)}$-predictable time by [A2], we have $\Delta\Xi^1_{\widehat{T}^j}=0$ for every $j$. Therefore, we have $[\Xi^1,\mathfrak{Z}^2]=0$, and thus we can decompose the target quantity as
\begin{align*}
&\hphantom{=}\sum_{i,j:\bar{R}^\vee(i,j)\leq s}\bar{\Xi}^{1}(\widehat{\mathcal{I}})^i_s\overline{\zeta}^{2}(\widehat{\mathcal{J}})^j_s\bar{K}^{i j}\\
&=\sum_{i,j=0}^\infty\left[\left\{\bar{K}^{ij}_-\bar{\Xi}^{1}(\widehat{\mathcal{I}})^i_-\right\}\bullet\widehat{\mathfrak{Z}}^{2}(\widehat{\mathcal{J}})^j_s+\left\{\bar{K}^{ij}_-\widehat{\mathfrak{Z}}^{2}(\widehat{\mathcal{J}})^j_-\right\}\bullet\bar{\Xi}^{1}(\widehat{\mathcal{I}})^i_s\right]
+\sum_{i,j:\bar{R}^\vee(i,j)>s}\bar{\Xi}^{1}(\widehat{\mathcal{I}})^i_s\widetilde{\mathfrak{Z}}^{2}(\widehat{\mathcal{J}})^j_s\bar{K}^{i j}_s\\
&=:\Upsilon_{1,s}+\Upsilon_{2,s}+\Upsilon_{3,s}
\end{align*}
due to integration by parts and Lemma 4.3 of \cite{Koike2012phy}. First we estimate $\Upsilon_{1,s}$. Since 
\begin{align*}
\Upsilon_{1,s}
=&b_n^{-1/2}\sum_{i,j=0}^{\infty}\sum_{p=1}^{k_n-1}\sum_{q=0}^{k_n-1}g_p^n \Delta(g)_q^n\bar{K}^{ij}_-\Xi^{1}(\widehat{I}^{i+p})_-\widehat{J}^{j+q}_-\bullet \mathfrak{Z}^{2}_s\\
=&b_n^{-1/2}\sum_{i,j=0}^{\infty}\sum_{p=i+1}^{i+k_n-1}\sum_{q=j}^{j+k_n-1}g_{p-i}^n \Delta(g)_{q-j}^n\bar{K}^{ij}_-\Xi^{1}(\widehat{I}^{i+p})_-\widehat{J}^{j+q}_-\bullet \mathfrak{Z}^{2}_s
=b_n^{-1/2}\sum_{p=1}^{\infty}\sum_{q=0}^\infty\upsilon'(p,q)_-\Xi^{1}(\widehat{I}^{p})_-\widehat{J}^{q}_-\bullet \mathfrak{Z}^{2}_s,
\end{align*}
where
$\upsilon'(p,q)_s=\sum_{i=(p-k_n+1)_+}^{p-1}\sum_{j=(q-k_n+1)_+}^{q}g_{p-i}^n \Delta(g)_{q-j}^n\bar{K}^{i j}_s.$
We have
\begin{align*}
[\Upsilon_{1,\cdot}]_s
=\sum_{q=0}^{\infty}\left[\sum_{p=1}^{\infty}\upsilon'(p,q)_{\widehat{T}^q}\Xi^{1}(\widehat{I}^{p})_{\widehat{T}^q}\right]^2\left(\zeta^2_{\widehat{T}^q}\right)^21_{\{\widehat{T}^q\leq s\}},
\end{align*}
hence we obtain
\begin{align*}
E\left[\left[\Upsilon_{1,\cdot}\right]_t \right]
\lesssim &\sum_{q=0}^{\infty}E\left[\left\{\sum_{p=1}^{\infty}\upsilon'(p,q)_{\widehat{T}^q}\Xi^{1}(\widehat{I}^{p})_{\widehat{T}^q}\right\}^21_{\{\widehat{T}^q\leq t\}}\right]
=\sum_{q=0}^{\infty}\sum_{p=1}^\infty E\left[\upsilon'(p,q)_{\widehat{T}^q\wedge t}^2 E\left[|\Xi^{1}(\widehat{I}^{p})_{\widehat{T}^q\wedge t}|^2|\mathcal{F}^{(0)}_{\widehat{S}^{p-1}\wedge{\widehat{T}^q\wedge t}}\right]\right]\\
\lesssim &\varphi_2(\varepsilon_n)\sum_{q=0}^{\infty} \sum_{p=1}^{\infty}\sum_{i=(p-k_n+1)_+}^{p-1}\sum_{j=(q-k_n+1)_+}^{q}E\left[\bar{K}^{i j}_t |\widehat{I}^{p}(t)|\right]
\lesssim\varphi_2(\varepsilon_n) k_n^3\lesssim k_n^4\varphi_2(\varepsilon_n) b_n^{1/2}
\end{align*}
by the optional sampling theorem, the representation of $\langle \Xi^{1}\rangle$, Lemma \ref{barKssp}, the Lipschitz continuity of $g$,  $(\ref{sumbarK})$ and $(\ref{window})$. Since $\langle\Upsilon_{1,\cdot}\rangle$ is the predictable compensator of $\left[\Upsilon_{1,\cdot}\right]$, the above result and the Schwarz and Doob inequalities yield $E\left[\sup_{0\leq s\leq t}|\Upsilon_{1,s}|\right]\lesssim k_n^2\sqrt{\varphi_2(\varepsilon_n)} b_n^{1/4}.$ On the other hand, we can show that $E\left[\sup_{0\leq s\leq t}|\Upsilon_{2,s}|\right]\lesssim k_n^2\sqrt{\varphi_2(\varepsilon_n)} b_n^{1/4}$ in a similar manner to the proof of $(\ref{estDelta})$. Finally, since 
\begin{equation*}
E\left[\sup_{0\leq s\leq t}\sum_{i,j=0}^\infty|\bar{\Xi}^{1}(\widehat{\mathcal{I}})^i_s\overline{\zeta}^{2}(\widehat{\mathcal{J}})^j_s\bar{K}^{i j}_s|^2\right]
\lesssim \sum_{i=0}^\infty E\left[\sup_{0\leq s\leq t}|\bar{\Xi}^{1}(\widehat{\mathcal{I}})^i_s|^2\right]
\lesssim\varphi_2(\varepsilon_n)E\left[\sum_{i,j=0}^{\infty}|\bar{I}^i(t)|\right]
\lesssim\varphi_2(\varepsilon_n)k_n
\end{equation*}
by $(\ref{sumbarK})$ and the Doob inequality, the Schwarz inequality and $(\ref{estshift})$ yield $E\left[\sup_{0\leq s\leq t}|\Upsilon_{3,s}|\right]\lesssim k_n^2\sqrt{\varphi_2(\varepsilon_n)} b_n^{1/4}$. Consequently, we obtain $(\ref{eqXiM3})$.
\end{proof}

%%%%%%%%%%%%%%%%%%%%%%%%%%%%%%%%%%%%%%%%%%%%%%%%%%%%%%%%%%%%%%%%%%%%%%%%%%%%%%
%                        ThetaM
%%%%%%%%%%%%%%%%%%%%%%%%%%%%%%%%%%%%%%%%%%%%%%%%%%%%%%%%%%%%%%%%%%%%%%%%%%%%%%

\begin{lem}\label{ThetaM}
Suppose $[\mathrm{ST}]$, $[\mathrm{A}2]$, $[\mathrm{SA}4]$ and $[\mathrm{SK}_\beta]$ hold for some $\beta\in[0,2]$. Suppose also that $[\mathrm{SN}^\flat_r]$ holds for some $r\in(2,\infty)$. Then for any $t>0$ we have
\begin{align}
k_n^{-2}\sup_{0\leq s\leq t}\left|\sum_{i,j:\bar{R}^\vee(i,j)\leq s}\bar{\Theta}^{1}(\widehat{\mathcal{I}})^i_s\bar{M}^{2}(\widehat{\mathcal{J}})^j_s\bar{K}^{i j}\right|=o_p\left(b_n^{1/4}\right)+o_p\left(\rho_n^{1-\beta/2}\right),\label{eqThetaM1}\\
k_n^{-2}\sup_{0\leq s\leq t}\left|\sum_{i,j:\bar{R}^\vee(i,j)\leq s}\bar{\Theta}^{1}(\widehat{\mathcal{I}})^i_s\widetilde{\underline{M}^2}(\widehat{\mathcal{J}})^j_s\bar{K}^{i j}\right|=o_p\left(b_n^{1/4}\right)+o_p\left(\rho_n^{1-\beta/2}\right),\label{eqThetaM2}\\
k_n^{-2}\sup_{0\leq s\leq t}\left|\sum_{i,j:\bar{R}^\vee(i,j)\leq s}\bar{\Theta}^{1}(\widehat{\mathcal{I}})^i_s\overline{\zeta}^{2}(\widehat{\mathcal{J}})^j\bar{K}^{i j}\right|=o_p\left(b_n^{1/4}\right)+o_p\left(\rho_n^{1-\beta/2}\right).\label{eqThetaM3}
\end{align}
\end{lem}

\begin{proof}
Define the process $\Upsilon_s$ by
$\Upsilon_s=\sum_{i,j:\bar{R}^\vee(i,j)\leq s}\bar{\Theta}^{1}(\widehat{\mathcal{I}})^i_s\bar{M}^{2}(\widehat{\mathcal{J}})^j_s\bar{K}^{i j}.$
If $\beta\geq 1$, the Schwarz inequality, [S$\mathrm{K}_\beta$] and [SA4] yield
\begin{align*}
&E\left[\sup_{0\leq s\leq t}\left|\Upsilon_s\right|\right]
\leq E\left[\sum_{i,j=1}^{\infty}\bar{K}^{i j}|\bar{\Theta}^{1}(\widehat{\mathcal{I}})^i_t\bar{M}^{2}(\widehat{\mathcal{J}})^j_t|\right]\\
\leq &\left\{E\left[\sum_{i,j=1}^{\infty}\bar{K}^{i j}|\bar{\Theta}^{1}(\widehat{\mathcal{I}})^i_t|^2\right]\right\}^{1/2}
\left\{E\left[\sum_{i,j=1}^{\infty}\bar{K}^{i j}|\bar{M}^{2}(\widehat{\mathcal{J}})^j_t|^2\right]\right\}^{1/2}
=o\left((k_n\bar{r}_n)^{1/2}\rho_n^{-(\beta-1) /2}\right),
\end{align*}
hence we obtain $\sup_{0\leq s\leq t}|\Upsilon_s|=o_p\left(\rho_n^{1-\beta/2}\right)$. If $\beta<1$, we decompose the target quantity as
\begin{align*}
\Upsilon_s=\sum_{i,j:\bar{R}^\vee(i,j)\leq s}\left\{\bar{\hat{\Theta}}^{1}(\widehat{\mathcal{I}})^i_s-\bar{\check{\Theta}}^{1}(\widehat{\mathcal{I}})^i_s\right\}\bar{M}^{2}(\widehat{\mathcal{J}})^j\bar{K}^{i j}
=\Upsilon_{1,s}+\Upsilon_{2,s},
\end{align*}
where $\hat{\Theta}^{1}=\kappa(\delta^{1})\star\nu^{1}$ and $\check{\Theta}^{1}=\kappa(\delta^{1})1_{(E_n^{1})^c}\star\nu^{1}$. Then, by [SA6] and an argument similar to the above, we can show $\sup_{0\leq s\leq t}|\Upsilon_{2,s}|=O_p\left(k_n^2\rho_n^{(1-\beta)/2}k_nb_n^{1/2}\right)=o_p(k_n^2b_n^{1/4})$. On the other hand, note that $|\bar{\Theta}^{(1)}(\widehat{\mathcal{I}})^i_s|\lesssim\varepsilon_n^{-(\beta-1)_+}k_n\bar{r}_n\lesssim\sqrt{k_n\bar{r}_n|\log b_n|}$, an argument similar to the proof of Lemma 6.2 of \cite{Koike2012phy} yields
\begin{align*}
\sup_{0\leq s\leq t}\left|\Upsilon_{1,s}-\sum_{i,j=1}^{\infty}\bar{\hat{\Theta}}^{1}(\widehat{\mathcal{I}})^i_s\bar{M}^{2}(\widehat{\mathcal{J}})^j_s\bar{K}^{i j}_s\right|=o_p\left(k_n^2 b_n^{1/4}\right),
\end{align*}
hence Lemma 12.1 of \cite{Koike2012phy} yields $\sup_{0\leq s\leq t}|\Upsilon_{1,s}|=o_p\left(k_n^2b_n^{1/4}\right)$ since we have the conditions [A2], [K$_\beta$](v) and [SA6]. Consequently, we conclude that $(\ref{eqThetaM1})$. Similarly we can also show $(\ref{eqThetaM2})$ and $(\ref{eqThetaM3})$.
\end{proof}

%%%%%%%%%%%%%%%%%%%%%%%%%%%%%%%%%%%%%%%%%%%%%%%%%%%%%%%%%%%%%%%%%%%%%%%%%%%%%%
%                         Proof of  main theorem (IAnoise)
%%%%%%%%%%%%%%%%%%%%%%%%%%%%%%%%%%%%%%%%%%%%%%%%%%%%%%%%%%%%%%%%%%%%%%%%%%%%%%%

\begin{proof}[\upshape{\bfseries{Proof of Theorem \ref{thmFAnoise}}}]
By a localization procedure, we may replace the conditions [K$_\beta$], [C1]-[C2] and [N$^\flat_r$] with [SF], [C1]-[C2] and [SN$^\flat_r$] respectively. Moreover, we can also replace the condition [T] with [ST] by Lemma \ref{lemthreshold}, while $(\ref{A4})$ can be replaced with $(\ref{SA4})$ due to the above argument.

We decompose the target quantity as
\begin{align}
&PTHY(\mathsf{Z}^{1}, \mathsf{Z}^{2})^n_t-PHY(\mathsf{X}^{1}, \mathsf{X}^{2})^n_t\nonumber\\
=&\frac{1}{(\psi_{HY}k_n)^2}\Biggl[\sum_{i,j:\bar{R}^\vee(i,j)\leq t}\left\{ \overline{\mathsf{Y}}^{1}(\widehat{\mathcal{I}}^i)\overline{\mathsf{Y}}^{2}(\widehat{\mathcal{J}}^j) 1_{\{|\overline{\mathsf{Y}}^{1}(\widehat{\mathcal{I}})^i|^2\leq 4\varrho_n[i] ,|\overline{\mathsf{Y}}^{2}(\widehat{\mathcal{J}})^j|^2\leq 4\varrho_n[j]\}}-\overline{\mathsf{X}}^{1}(\widehat{\mathcal{I}})^i \overline{\mathsf{X}}^{2}(\widehat{\mathcal{J}})^j\right\}\bar{K}^{i j}\nonumber\\
&\quad +\sum_{i,j:\bar{R}^\vee(i,j)\leq t} \overline{\mathsf{Y}}^{1}(\widehat{\mathcal{I}})^i\overline{\mathsf{Y}}^{2}(\widehat{\mathcal{J}})^j\bar{K}^{i j}\left(1_{\{|\overline{\mathsf{Z}}^{1}(\widehat{\mathcal{I}})^i|^2\leq\varrho_n^{1}[i] ,|\overline{\mathsf{Z}}^{2}(\widehat{\mathcal{J}})^j|^2\leq\varrho_n^{2}[j]\}}-1_{\{|\overline{\mathsf{Y}}^{1}(\widehat{\mathcal{I}})^i|^2\leq 4\varrho_n^{1}[i] ,|\overline{\mathsf{Y}}^{2}(\widehat{\mathcal{J}})^j|^2\leq 4\varrho_n^{2}[j]\}}\right)\nonumber\\
&\quad +\sum_{i,j:\bar{R}^\vee(i,j)\leq t}\left\{\bar{L}^{1}(\widehat{\mathcal{I}})^i_t\overline{\mathsf{Y}}^{2}(\widehat{\mathcal{J}})^j+\overline{\mathsf{Y}}^{1}(\widehat{\mathcal{I}})^i \bar{L}^{2}(\widehat{\mathcal{J}})^j_t+\bar{L}^{1}(\widehat{\mathcal{I}})^i_t \bar{L}^{2}(\widehat{\mathcal{J}})^j_t\right\}\bar{K}^{i j} 1_{\{|\overline{\mathsf{Z}}^{1}(\widehat{\mathcal{I}})^i|^2\leq\varrho_n^{1}[i] ,|\overline{\mathsf{Z}}^{2}(\widehat{\mathcal{J}})^j|^2\leq\varrho_n^{2}[j]\}}\Biggr]\nonumber\\
=:&\mathbb{I}_t+\mathbb{I}\mathbb{I}_t+\mathbb{I}\mathbb{I}\mathbb{I}_t+\mathbb{I}\mathbb{V}_t+\mathbb{V}_t.\label{tphy0}
\end{align}

%%%%%%%%%%%%%%%%%%%%%%%%%%%%%%%%%%%%%%%%%%%%%%%%%%%%%%%%%%%%%%%%%%%%%%%%%%%%%%
%                         tphyI
%%%%%%%%%%%%%%%%%%%%%%%%%%%%%%%%%%%%%%%%%%%%%%%%%%%%%%%%%%%%%%%%%%%%%%%%%%%%%%

(a) By Theorem \ref{thmFAnoise}, we have
\begin{equation}\label{tphyI}
\sup_{0\leq s\leq t}|\mathbb{I}_s|=o_p(b_n^{1/4})+O_p\left(\eta_n^{\frac{r-2}{2}}\right).
\end{equation}

%%%%%%%%%%%%%%%%%%%%%%%%%%%%%%%%%%%%%%%%%%%%%%%%%%%%%%%%%%%%%%%%%%%%%%%%%%%%%%
%                         tphyII
%%%%%%%%%%%%%%%%%%%%%%%%%%%%%%%%%%%%%%%%%%%%%%%%%%%%%%%%%%%%%%%%%%%%%%%%%%%%%%%

(b) Next consider $\mathbb{II}$. We decompose it as
\begin{align*}
\mathbb{I}\mathbb{I}_t
&=\frac{1}{(\psi_{HY}k_n)^2}\sum_{i,j:\bar{R}^\vee(i,j)\leq t} \overline{\mathsf{Y}}^{1}(\widehat{\mathcal{I}})^i\overline{\mathsf{Y}}^{2}(\widehat{\mathcal{J}})^j\bar{K}^{i j}
\left( 1_{\{|\overline{\mathsf{Z}}^{1}(\widehat{\mathcal{I}})^i|^2\leq\varrho_n^{1}[i] ,|\overline{\mathsf{Z}}^{2}(\widehat{\mathcal{J}})^j|^2\leq\varrho_n^{2}[j]\}}1_{\{ |\overline{\mathsf{Y}}^{1}(\widehat{\mathcal{I}})^i|^2>4\varrho_n^{1}[i]\}\cup\{ |\overline{\mathsf{Y}}^{2}(\widehat{\mathcal{J}})^j|^2>4\varrho_n^{2}[j]\}}\right.\\
&\hphantom{=\frac{1}{(\psi_{HY}k_n)^2}\sum_{i,j:\bar{R}^\vee(i,j)\leq t}}\left.- 1_{\{ |\overline{\mathsf{Z}}^{1}(\widehat{\mathcal{I}})^i|^2>\varrho_n^{1}[i]\}\cup\{|\overline{\mathsf{Z}}^{2}(\widehat{\mathcal{J}})^j|^2>\varrho_n^{2}[j]\}}1_{\{ \overline{\mathsf{Y}}^{1}(\widehat{\mathcal{I}})^i|^2\leq 4\varrho_n^{1}[i] ,\overline{\mathsf{Y}}^{2}(\widehat{\mathcal{J}})^j|^2\leq 4\varrho_n^{2}[j] \}}\right)\\
&=:\mathbb{I}\mathbb{I}_{1,t}+\mathbb{I}\mathbb{I}_{2,t}.
\end{align*}

%%%%%%%%%%%%%%%%%%%%%%%%%%%%%%%%%%%%%%%%%%%%%%%%%%%%%%%%%%%%%%%%%%%%%%%%%%%%%%%%%%%
%                                  tphyII1
%%%%%%%%%%%%%%%%%%%%%%%%%%%%%%%%%%%%%%%%%%%%%%%%%%%%%%%%%%%%%%%%%%%%%%%%%%%%%%%%%%%%

First estimate $\mathbb{II}_{1,t}$. We decompose it as
\begin{align*}
&\mathbb{I}\mathbb{I}_{1,t}\nonumber\\
=&\frac{1}{(\psi_{HY}k_n)^2}\sum_{i,j:\bar{R}^\vee(i,j)\leq t} \overline{\mathsf{Y}}^{(1)}(\widehat{\mathcal{I}})^i\overline{\mathsf{Y}}^{(2)}(\widehat{\mathcal{J}})^j\bar{K}^{i j}1_{\{|\overline{\mathsf{Z}}^{(1)}(\widehat{\mathcal{I}})^i|^2\leq\varrho_n^{(1)}[i] ,|\overline{\mathsf{Z}}^{(2)}(\widehat{\mathcal{J}})^j|^2\leq\varrho_n^{(2)}[j]\}}\\
&\hphantom{\sum_{i,j:\bar{R}^\vee(i,j)\leq t}}\times \left(1_{\{ |\overline{\mathsf{Y}}^{(1)}(\widehat{\mathcal{I}})^i|^2>4\varrho_n^{(1)}[i], |\overline{\mathsf{Y}}^{(2)}(\widehat{\mathcal{J}})^j|^2\leq 4\varrho_n^{(2)}[j]\}}+1_{\{ |\overline{\mathsf{Y}}^{(1)}(\widehat{\mathcal{I}})^i|^2\leq 4\varrho_n^{(1)}[i], |\overline{\mathsf{Y}}^{(2)}(\widehat{\mathcal{J}})^j|^2>4\varrho_n^{(2)}[j]\}}\right.\\
&\hphantom{\sum_{i,j:\bar{R}^\vee(i,j)\leq t}\times (}\left.+1_{\{ |\overline{\mathsf{Y}}^{(1)}(\widehat{\mathcal{I}})^i|^2> 4\varrho_n^{(1)}[i], |\overline{\mathsf{Y}}^{(2)}(\widehat{\mathcal{J}})^j|^2> 4\varrho_n^{(2)}[j]\}}\right)\\
=:&\mathbb{II}_{1,t}^{(1)}+\mathbb{II}_{1,t}^{(2)}+\mathbb{II}_{1,t}^{(3)}.
\end{align*}
Consider $\mathbb{II}_{1,t}^{(1)}$. The Schwarz inequality, $(\ref{sumbarK})$, [SC1] and $(\ref{window})$ yield
{\small \begin{align*}
&\sup_{0\leq s\leq t}|\mathbb{II}_{1,s}^{(1)}|\\
\leq &\frac{1}{\psi_{HY}^2 k_n}\left\{\sum_{i:\widehat{S}^{i+k_n}\leq t} |\overline{\mathsf{Y}}^{1}(\widehat{\mathcal{I}})^i|^21_{\{|\overline{\mathsf{Z}}^{1}(\widehat{\mathcal{I}})^i|^2\leq\varrho_n^{1}[i], |\overline{\mathsf{Y}}^{1}(\widehat{\mathcal{I}})^i|^2>4\varrho_n^{1}[i]\}}\right\}^{1/2}
\left\{\sum_{j:\widehat{T}^{j+k_n}\leq t} |\overline{\mathsf{Y}}^{2}(\widehat{\mathcal{J}})^j|^21_{\{ |\overline{\mathsf{Y}}^{2}(\widehat{\mathcal{J}})^j|^2\leq 4\varrho_n^{2}[j], |\overline{\mathsf{Z}}^{2}(\widehat{\mathcal{J}})^j|^2\leq\varrho_n^{2}[j]\}}\right\}^{1/2}\nonumber\\
\lesssim &\frac{(b_n^{-1}\rho_n)^{1/2}}{k_n}\left\{\sum_{i:\widehat{S}^{i+k_n}\leq t} |\overline{\mathsf{Y}}^{1}(\widehat{\mathcal{I}})^i|^21_{\{|\overline{\mathsf{Z}}^{1}(\widehat{\mathcal{I}})^i|^2\leq\varrho_n^{1}[i], |\overline{\mathsf{Y}}^{1}(\widehat{\mathcal{I}})^i|^2>4\varrho_n^{1}[i]\}}\right\}^{1/2}\nonumber\\
\lesssim &\rho_n^{1/2}\left\{\sum_{i:\widehat{S}^{i+k_n}\leq t} |\overline{\mathsf{Y}}^{1}(\widehat{\mathcal{I}})^i|^21_{\{|\overline{\mathsf{Z}}^{1}(\widehat{\mathcal{I}})^i|^2\leq\varrho_n^{1}[i], |\overline{\mathsf{Y}}^{1}(\widehat{\mathcal{I}})^i|^2>4\varrho_n^{1}[i]\}}\right\}^{1/2}.
\end{align*}}
On $\{|\overline{\mathsf{Z}}^{1}(\widehat{\mathcal{I}})^i|^2\leq\varrho_n^{1}[i], |\overline{\mathsf{Y}}^{1}(\widehat{\mathcal{I}})^i|^2>4\varrho_n^{1}[i]\}$ we have $|\bar{L}^{1}(\widehat{\mathcal{I}})^i_t|\geq|\overline{\mathsf{Y}}^{1}(\widehat{\mathcal{I}})^i|-|\overline{\mathsf{Z}}^{1}(\widehat{\mathcal{I}})^i|>\sqrt{\varrho_n^{1}[i]}$, hence we obtain
\begin{align*}
\sum_{i:\widehat{S}^{i+k_n}\leq t} |\overline{\mathsf{Y}}^{1}(\widehat{\mathcal{I}})^i|^21_{\{|\overline{\mathsf{Z}}^{1}(\widehat{\mathcal{I}})^i|^2\leq\varrho_n^{1}[i], |\overline{\mathsf{Y}}^{1}(\widehat{\mathcal{I}})^i|^2>4\varrho_n^{1}[i]\}}
\leq \sum_{i:\widehat{S}^{i+k_n}\leq t} |\overline{\mathsf{Y}}^{1}(\widehat{\mathcal{I}})^i|^21_{\{|\bar{L}^{1}(\widehat{\mathcal{I}})^i_t|^2>\varrho_n^{1}[i], |\overline{\mathsf{Y}}^{1}(\widehat{\mathcal{I}})^i|^2>4\varrho_n^{1}[i]\}},
\end{align*}
and thus Lemma \ref{lemLX} yield
\begin{equation}\label{tphyII1(1)}
\sup_{0\leq s\leq t}|\mathbb{II}_{1,s}^{(1)}|=o_p\left(\eta_n^{\frac{r-2}{4}}\rho_n^{1/2-\beta/4}\right).
\end{equation}
By symmetry we obtain
\begin{equation}\label{tphyII1(2)}
\sup_{0\leq s\leq t}|\mathbb{II}_{1,s}^{(2)}|=o_p\left(\eta_n^{\frac{r-2}{4}}\rho_n^{1/2-\beta/4}\right).
\end{equation}  
On the other hand, the Schwarz inequality and $(\ref{sumbarK})$ yield
{\small \begin{align*}
\sup_{0\leq s\leq t}|\mathbb{II}_{1,s}^{(3)}|
\leq &\frac{1}{\psi_{HY}^2 k_n}\left\{\sum_{i:\widehat{S}^{i+k_n}\leq t} |\overline{\mathsf{Y}}^{1}(\widehat{\mathcal{I}})^i|^21_{\{|\overline{\mathsf{Z}}^{1}(\widehat{\mathcal{I}})^i|^2\leq\varrho_n^{1}[i], |\overline{\mathsf{Y}}^{1}(\widehat{\mathcal{I}})^i|^2>4\varrho_n^{1}[i]\}}\right\}^{1/2}\\
&\hphantom{\frac{1}{\psi_{HY}^2 k_n}}\times\left\{\sum_{j:\widehat{T}^{j+k_n}\leq t} |\overline{\mathsf{Y}}^{2}(\widehat{\mathcal{J}})^j|^21_{\{ |\overline{\mathsf{Y}}^{2}(\widehat{\mathcal{J}})^j|^2> 4\varrho_n^{2}[j], |\overline{\mathsf{Z}}^{2}(\widehat{\mathcal{J}})^j|^2\leq\varrho_n^{2}[j]\}}\right\}^{1/2},
\end{align*}}
hence by the same arguments as the above we obtain
$\sup_{0\leq s\leq t}|\mathbb{II}_{1,s}^{(3)}|
=o_p\left(k_n^{-1}\eta_n^{\frac{r-2}{2}}\rho_n^{-\beta/2}\right),$
and thus $\beta\leq 2$, $(\ref{window})$ and $(\ref{threshold2})$ yield
\begin{equation}\label{tphyII1(3)}
\sup_{0\leq s\leq t}|\mathbb{II}_{1,s}^{(3)}|
=o_p\left(\eta_n^{\frac{r-2}{2}}\right).
\end{equation}
$(\ref{tphyII1(1)})$, $(\ref{tphyII1(2)})$ and $(\ref{tphyII1(3)})$ yield
\begin{equation}\label{tphyII1}
\sup_{0\leq s\leq t}|\mathbb{II}_{1,s}|=o_p\left(\eta_n^{\frac{r-2}{2}}\right)+o_p\left(\rho_n^{1-\beta/2}\right).
\end{equation}

%%%%%%%%%%%%%%%%%%%%%%%%%%%%%%%%%%%%%%%%%%%%%%%%%%%%%%%%%%%%%%%%%%%%%%%%%%%%%%%%%%%
%                                  tphyII2
%%%%%%%%%%%%%%%%%%%%%%%%%%%%%%%%%%%%%%%%%%%%%%%%%%%%%%%%%%%%%%%%%%%%%%%%%%%%%%%%%%%%

Next estimate $\mathbb{II}_{2,t}$. Note that $(\ref{sumbarK})$, we obtain
{\small \begin{align*}
&\sup_{0\leq s\leq t}|\mathbb{II}_{2,t}|\nonumber\\
\lesssim &\frac{\rho_n}{k_n}\left\{\sum_{i:\widehat{S}^{i+k_n}\leq t}1_{\{|\overline{\mathsf{Z}}^{1}(\widehat{\mathcal{I}})^i|^2>\varrho_n^{1}[i], |\overline{\mathsf{Y}}^{1}(\widehat{\mathcal{I}})^i|^2\leq 4\varrho_n^{1}[i]\}}
+\sum_{j:\widehat{T}^{j+k_n}\leq t}1_{\{|\overline{\mathsf{Z}}^{2}(\widehat{\mathcal{J}})^j|^2>\varrho_n^{2}[j], |\overline{\mathsf{Y}}^{2}(\widehat{\mathcal{J}})^j|^2\leq 4\varrho_n^{2}[j]\}}\right\}\nonumber\\
\leq &\frac{\rho_n}{k_n}
\Biggl\{\sum_{i:\widehat{S}^{i+k_n}\leq t}1_{\{|\overline{\zeta}^{1}(\widehat{\mathcal{I}})^i|^2>\varrho^{1}_n[i]/4\}}
+\sum_{j:\widehat{T}^{j+k_n}\leq t}1_{\{|\overline{\zeta}^{2}(\widehat{\mathcal{J}})^j|^2>\varrho^{2}_n[j]/4\}}\\
&+\sum_{i:\widehat{S}^{i+k_n}\leq t}1_{\{|\overline{\mathsf{Z}}^{1}(\widehat{\mathcal{I}})^i|^2>\varrho_n^{1}[i], |\overline{\mathsf{Y}}^{1}(\widehat{\mathcal{I}})^i|^2\leq 4\varrho_n^{1}[i],|\overline{\zeta}^{1}(\widehat{\mathcal{I}})^i|^2\leq\varrho^{1}_n[i]/4\}}
+\sum_{j:\widehat{T}^{j+k_n}\leq t}1_{\{|\overline{\mathsf{Z}}^{2}(\widehat{\mathcal{J}})^j|^2>\varrho_n^{2}[j], |\overline{\mathsf{Y}}^{2}(\widehat{\mathcal{J}})^j|^2\leq 4\varrho_n^{2}[j],|\overline{\zeta}^{2}(\widehat{\mathcal{J}})^j|^2\leq\varrho^{2}_n[j]/4\}}\Biggr\}.
\end{align*}}
[SC1] and Lemma \ref{moment} yield
\begin{equation}\label{sumU}
\sum_{i:\widehat{S}^{i+k_n}\leq t}1_{\{|\overline{\zeta}^{1}(\widehat{\mathcal{I}})^i|^2>\varrho^{1}_n[i]/4\}}\lesssim b_n^{-1}\eta_n^{r/2},\qquad
\sum_{j:\widehat{T}^{j+k_n}\leq t}1_{\{|\overline{\zeta}^{2}(\widehat{\mathcal{J}})^j|^2>\varrho^{2}_n[j]/4\}}\lesssim b_n^{-1}\eta_n^{r/2}.
\end{equation}
On the other hand, on $\{|\overline{\mathsf{Y}}^{1}(\widehat{\mathcal{I}})^i|^2\leq 4\varrho_n^{(1)}[i],|\overline{\zeta}^{1}(\widehat{\mathcal{I}})^i|^2\leq\varrho^{1}_n[i]/4,\widehat{S}^{i+k_n}\leq t\}$ we have
\begin{align*}
|\bar{D}^{1}(\widehat{\mathcal{I}})^i_t|
\leq|\overline{\mathsf{Y}}^{1}(\widehat{\mathcal{I}})^i|+|\bar{X}^{1}(\widehat{\mathcal{I}})^i_t|
+|\overline{\zeta}^{1}(\widehat{\mathcal{I}})^i|+|\widetilde{\underline{X}^1}(\widehat{\mathcal{I}})^i_t|
\leq 5\sqrt{\varrho_n^{1}[i]}/2+2K\sqrt{2 k_n\bar{r}_n|\log b_n|}
\rightarrow 0,
\end{align*}
hence a.s. for sufficiently large $n$ we have $\bar{D}^{1}(\widehat{\mathcal{I}})^i_t=0$. Moreover, on $\{|\overline{\mathsf{Z}}^{1}(\widehat{\mathcal{I}})^i|^2>\varrho_n^{1}[i],|\overline{\zeta}^{1}(\widehat{\mathcal{I}})^i|^2\leq\varrho^{1}_n[i]/4,\bar{D}^{1}(\widehat{\mathcal{I}})^i_t=0,\widehat{S}^{i+k_n}\leq t\}$, by $(\ref{absmod2})$ we have
\begin{align*}
|\bar{L}^{1}(\widehat{\mathcal{I}})^i_t|
\geq |\overline{\mathsf{Z}}^{1}(\widehat{\mathcal{I}})^i|-|\bar{X}^{1}(\widehat{\mathcal{I}})^i_t|-|\overline{\zeta}^{1}(\widehat{\mathcal{I}})^i|-|\widetilde{\underline{X}^1}(\widehat{\mathcal{I}})^i_t|
>\sqrt{\rho_n}\left(\frac{1}{2 \sqrt{K_0}}-2K\sqrt{\frac{2 k_n\bar{r}_n|\log b_n|}{\rho_n}}\right),
\end{align*}
hence $(\ref{threshold2})$ yields $|\bar{L}^{1}(\widehat{\mathcal{I}})^i_t|>\sqrt{\rho_n/9 K_0}$. Therefore we obtain
\begin{align*}
\sum_{i:\widehat{S}^{i+k_n}\leq t}1_{\{|\overline{\mathsf{Z}}^{1}(\widehat{\mathcal{I}})^i|^2>\varrho_n^{1}[i], |\overline{\mathsf{Y}}^{1}(\widehat{\mathcal{I}})^i|^2\leq 4\varrho_n^{1}[i],|\overline{\zeta}^{1}(\widehat{\mathcal{I}})^i|^2\leq\varrho^{1}_n[i]/4\}}
\leq\sum_{i=1}^{\infty}1_{\{|\bar{L}^{1}(\widehat{\mathcal{I}})^i_t|^2>\rho_n/9 K_0\}}.
\end{align*} 
By symmetry we also obtain
\begin{align*}
\sum_{j:\widehat{T}^{j+k_n}\leq t}1_{\{|\overline{\mathsf{Z}}^{2}(\widehat{\mathcal{J}})^j|^2>\varrho_n^{2}[j], |\overline{\mathsf{Y}}^{2}(\widehat{\mathcal{J}})^j|^2\leq 4\varrho_n^{2}[j],|\overline{\zeta}^{2}(\widehat{\mathcal{J}})^j|^2\leq\varrho^{2}_n[j]/4\}}
\leq\sum_{j=1}^{\infty}1_{\{|\bar{L}^{2}(\widehat{\mathcal{J}})^j_t|^2>\rho_n/9 K_0\}}.
\end{align*}
Combining these results with Lemma \ref{barlarge} and $(\ref{window})$, we obtain
\begin{equation}\label{tphyII2}
\sup_{0\leq s\leq t}|\mathbb{II}_{2,t}|
=O_p\left(\eta_n^{\frac{r-2}{2}}\right) + o_p\left(\rho_n^{1-\beta/2}\right).
\end{equation}

By $(\ref{tphyII1})$ and $(\ref{tphyII2})$, we conclude
\begin{equation}\label{tphyII}
\sup_{0\leq s\leq t}|\mathbb{II}_{t}|
=O_p\left(\eta_n^{\frac{r-2}{2}}\right) + o_p\left(\rho_n^{1-\beta/2}\right).
\end{equation}

%%%%%%%%%%%%%%%%%%%%%%%%%%%%%%%%%%%%%%%%%%%%%%%%%%%%%%%%%%%%%%%%%%%%%%%%%%%%%%%%%%%%
%                                  tphyIII
%%%%%%%%%%%%%%%%%%%%%%%%%%%%%%%%%%%%%%%%%%%%%%%%%%%%%%%%%%%%%%%%%%%%%%%%%%%%%%%%%%%%

(c) Next consider $\mathbb{III}$. We decompose it as
\begin{align*}
\mathbb{I}\mathbb{I}\mathbb{I}_t
&=\frac{1}{(\psi_{HY}k_n)^2}\sum_{i,j:\bar{R}^\vee(i,j)\leq t}\bar{L}^{1}(\widehat{\mathcal{I}})^i_t\overline{\mathsf{Y}}^{2}(\widehat{\mathcal{J}})^j\bar{K}^{i j}1_{\{|\overline{\mathsf{Z}}^{1}(\widehat{\mathcal{I}})^i|^2\leq\varrho_n^{1}[i] ,|\overline{\mathsf{Z}}^{2}(\widehat{\mathcal{J}})^j|^2\leq\varrho_n^{2}[j]\}}\\
&\hphantom{=\sum_{i,j:S^i\vee T^j\leq t}}\times \left(1_{\{|\bar{L}^{1}(\widehat{\mathcal{I}})^i_t|^2> 4\varrho_n^{1}[i],|\bar{L}^{2}(\widehat{\mathcal{J}})^j_t|^2> 4\varrho_n^{2}[j]\}}+1_{\{ |\bar{L}^{1}(\widehat{\mathcal{I}})^i_t|^2> 4\varrho_n^{1}[i], |\bar{L}^{2}(\widehat{\mathcal{J}})^j_t|^2\leq 4\varrho_n^{2}[j]\}}\right.\\
&\hphantom{=\sum_{i,j:S^i\vee T^j\leq t}\times(}\left.+1_{\{|\bar{L}^{1}(\widehat{\mathcal{I}})^i_t|^2\leq 4\varrho_n^{1}[i], |\bar{L}^{2}(\widehat{\mathcal{J}})^j_t|^2> 4\varrho_n^{2}[j]\}}+1_{\{ |\bar{L}^{1}(\widehat{\mathcal{I}})^i_t|^2\leq 4\varrho_n^{1}[i] ,|\bar{L}^{2}(\widehat{\mathcal{J}})^j_t|^2\leq 4\varrho_n^{2}[j]\}}\right)\\
&=:\mathbb{III}_{1,t}+\mathbb{III}_{2,t}+\mathbb{III}_{3,t}+\mathbb{III}_{4,t}.
\end{align*}

%%%%%%%%%%%%%%%%%%%%%%%%%%%%%%%%%%%%%%%%%%%%%%%%%%%%%%%%%%%%%%%%%%%%%%%%%%%%%
%                           tphyIII1
%%%%%%%%%%%%%%%%%%%%%%%%%%%%%%%%%%%%%%%%%%%%%%%%%%%%%%%%%%%%%%%%%%%%%%%%%%%%%

First we estimate $\mathbb{III}_{1,t}$. The Schwarz inequality and $(\ref{sumbarK})$ yield
{\footnotesize \begin{align*}
&\sup_{0\leq s\leq t}|\mathbb{III}_{1,s}|\\
\leq &\frac{1}{\psi_{HY}^2 k_n}\left\{\sum_{i:\widehat{S}^{i+k_n}\leq t}|\bar{L}^{1}(\widehat{\mathcal{I}})^i_t|^21_{\{|\overline{\mathsf{Z}}^{1}(\widehat{\mathcal{I}})^i|^2\leq\varrho_n^{1}[i] ,|\bar{L}^{1}(\widehat{\mathcal{I}})^i_t|^2> 4\varrho_n^{1}[i]\}}\right\}^{1/2}
\left\{\sum_{j:\widehat{T}^{j+k_n}\leq t}|\overline{\mathsf{Y}}^{2}(\widehat{\mathcal{J}})^j|^21_{\{|\overline{\mathsf{Z}}^{2}(\widehat{\mathcal{J}})^j|^2\leq\varrho_n^{2}[j] ,|\bar{L}^{2}(\widehat{\mathcal{J}})^j_t|^2> 4\varrho_n^{2}[j]\}}\right\}^{1/2}.
\end{align*}}
On $\{|\overline{\mathsf{Z}}^{2}(\widehat{\mathcal{J}})^j|^2\leq\varrho_n^{2}[j] ,|\bar{L}^{2}(\widehat{\mathcal{J}})^j_t|^2> 4\varrho_n^{2}[j]\}$ we have $|\overline{\mathsf{Y}}^{2}(\widehat{\mathcal{J}})^j|\geq|\bar{L}^{2}(\widehat{\mathcal{J}})^j_t|-|\overline{\mathsf{Z}}^{2}(\widehat{\mathcal{J}})^j|>\sqrt{\varrho_n^{2}[j]}$, hence Lemma \ref{lemLX} and Lemma \ref{lemL2Z} imply that
\begin{equation}\label{tphyIII1}
\sup_{0\leq s\leq t}|\mathbb{III}_{1,s}|=o_p\left( k_n^{-1/2}\eta_n^{\frac{r-1}{2}}\rho_n^{-\beta/4}\right)
=o_p\left(\eta_n^{r/2}\rho_n^{1/2-\beta/4}\right).
\end{equation}

%%%%%%%%%%%%%%%%%%%%%%%%%%%%%%%%%%%%%%%%%%%%%%%%%%%%%%%%%%%%%%%%%%%%%%%%%%%%%
%                           tphyIII2
%%%%%%%%%%%%%%%%%%%%%%%%%%%%%%%%%%%%%%%%%%%%%%%%%%%%%%%%%%%%%%%%%%%%%%%%%%%%%

Next we estimate $\mathbb{III}_{2,t}$. On $\{|\overline{\mathsf{Z}}^{2}(\widehat{\mathcal{J}})^j|^2\leq\varrho_n^{2}[j], |\bar{L}^{2}(\widehat{\mathcal{J}})^j_t|^2\leq 4\varrho_n^{2}[j]\}$ we have $|\overline{\mathsf{Y}}^{2}(\widehat{\mathcal{J}})^j|\leq|\overline{\mathsf{Z}}^{2}(\widehat{\mathcal{J}})^j|+|\bar{L}^{2}(\widehat{\mathcal{J}})^j_t|\leq 3\sqrt{\varrho_n^{2}[j]}$, hence by $(\ref{sumbarK})$ we obtain
\begin{equation*}
\sup_{0\leq s\leq t}|\mathbb{III}_{2,s}|
\lesssim\frac{\sqrt{\rho_n}}{k_n}\sum_{i:\widehat{S}^{i+k_n}\leq t}|\bar{L}^{1}(\widehat{\mathcal{I}})^i|1_{\{|\overline{\mathsf{Z}}^{1}(\widehat{\mathcal{I}})^i|^2\leq\varrho_n^{1}[i], |\bar{L}^{1}(\widehat{\mathcal{I}})^i_t|^2> 4\varrho_n^{1}[i]\}},
\end{equation*}
and thus Lemma \ref{lemLZ} yields
\begin{equation}\label{tphyIII2}
\sup_{0\leq s\leq t}|\mathbb{III}_{2,s}|=o_p\left(\eta_n^{r/4}\rho_n^{1/2-\beta/4}\right).
\end{equation}

%%%%%%%%%%%%%%%%%%%%%%%%%%%%%%%%%%%%%%%%%%%%%%%%%%%%%%%%%%%%%%%%%%%%%%%%%%%%%
%                          tphyIII3
%%%%%%%%%%%%%%%%%%%%%%%%%%%%%%%%%%%%%%%%%%%%%%%%%%%%%%%%%%%%%%%%%%%%%%%%%%%%%%

Next we estimate $\mathbb{III}_{3,t}$. By the Schwarz inequality and $(\ref{sumbarK})$ we have
{\small \begin{align*}
\sup_{0\leq s\leq t}|\mathbb{III}_{3,s}|
\lesssim \frac{1}{k_n}\left\{\sum_{i:\widehat{S}^{i+k_n}\leq t}|\bar{L}^{1}(\widehat{\mathcal{I}})^i_t|^21_{\{|\bar{L}^{1}(\widehat{\mathcal{I}})^i_t|^2\leq 4\varrho_n^{1}[i]\}}\right\}^{1/2}
\left\{\sum_{j:\widehat{T}^{j+k_n}\leq t}|\overline{\mathsf{Y}}^{2}(\widehat{\mathcal{J}})^j|^21_{\{|\overline{\mathsf{Z}}^{2}(\widehat{\mathcal{J}})^j|^2\leq \varrho_n^{2}[j], |\bar{L}^{2}(\widehat{\mathcal{I}})^i_t|^2> 4\varrho_n^{2}[j]\}}\right\}^{1/2}.
\end{align*}}
Note that on $\{|\overline{\mathsf{Z}}^{2}(\widehat{\mathcal{J}})^j|^2\leq\varrho_n^{2}[j], |\bar{L}^{2}(\widehat{\mathcal{J}})^j_t|^2> 4\varrho_n^{2}[j]\}$ we have $|\overline{\mathsf{Y}}^{2}(\widehat{\mathcal{J}})^j|>|\bar{L}^{2}(\widehat{\mathcal{J}})^j_t|-|\overline{\mathsf{Z}}^{2}(\widehat{\mathcal{J}})^j|>\sqrt{\varrho_n^{2}[j]}$, Lemma \ref{barsmalljumpsum}, Lemma \ref{lemLX} and $k_n^{-1}=o(\rho_n)$ yield
\begin{equation}\label{tphyIII3}
\sup_{0\leq s\leq t}|\mathbb{III}_{3,s}|
=o_p\left( k_n^{-1/2}\rho_n^{1/2-\beta/2}\eta_n^{\frac{r-2}{4}}\right)
=o_p\left( \rho_n^{1-\beta/2}\eta_n^{\frac{r-2}{4}}\right).
\end{equation}

%%%%%%%%%%%%%%%%%%%%%%%%%%%%%%%%%%%%%%%%%%%%%%%%%%%%%%%%%%%%%%%%%%%%%%%%%%%%%
%                           tphyIII4
%%%%%%%%%%%%%%%%%%%%%%%%%%%%%%%%%%%%%%%%%%%%%%%%%%%%%%%%%%%%%%%%%%%%%%%%%%%%%%

Finally we estimate $\mathbb{I}\mathbb{I}\mathbb{I}_{4,t}$. First we specify $(\varepsilon_n)$. By Lemma \ref{epsspecify} we can choose the sequence $(\varepsilon_n)$ satisfying $(\ref{eps2})$ for $p=2$ and $(\ref{eps3})$. Next we decompose the target quantity as
\begin{align*}
\mathbb{III}_{4,t}
=&\frac{1}{(\psi_{HY}k_n)^2}\sum_{i,j:\bar{R}^\vee(i,j)\leq t}\bar{L}^{1}(\widehat{\mathcal{I}})^i_t\overline{\mathsf{Y}}^{2}(\widehat{\mathcal{J}})^j\bar{K}^{i j}1_{\{|\overline{\mathsf{Z}}^{1}(\widehat{\mathcal{I}})^i|^2\leq\varrho_n^{1}[i] ,|\overline{\mathsf{Z}}^{2}(\widehat{\mathcal{J}})^j|^2\leq\varrho_n^{2}[j], |\bar{L}^{1}(\widehat{\mathcal{I}})^i_t|^2\leq 4\varrho_n^{1}[i] ,|\bar{L}^{2}(\widehat{\mathcal{J}})^j_t|^2\leq 4\varrho_n^{2}[j]\}}\\
&\hphantom{\frac{1}{(\psi_{HY}k_n)^2}}\times \left(1_{\{ N^{2}(\bar{J}^j)_t\neq 0\}}+1_{\{ N^{2}(\bar{J}^j)_t=0, \tilde{N}^{1}(\bar{I}^i)_t\neq 0\}}+1_{\{ N^{2}(\bar{J}^j)_t=\tilde{N}^{1}(\bar{I}^i)_t=0\}}\right)\\
%&\hphantom{\frac{1}{(\psi_{HY}k_n)^2}}\times 1_{\{|\overline{\mathsf{Z}}^{1}(\widehat{\mathcal{I}})^i|^2\leq\varrho_n^{1}[i] ,|\overline{\mathsf{Z}}^{2}(\widehat{\mathcal{J}})^j|^2\leq\varrho_n^{2}[j], |\bar{L}^{1}(\widehat{\mathcal{I}})^i|^2\leq 4\varrho_n^{1}[i] ,|\bar{L}^{2}(\widehat{\mathcal{J}})^j|^2\leq 4\varrho_n^{2}[j]\}}\\
=:&\mathbb{III}_{4,t}^{(1)}+\mathbb{III}_{4,t}^{(2)}+\mathbb{III}_{4,t}^{(3)}.
\end{align*}
By Lemma \ref{lembasic}, we obtain $\sup_{0\leq s\leq t}|\mathbb{III}_{4,s}^{(1)}|=o_p(b_n^{1/4})$. Moreover, on $\{|\overline{\mathsf{Z}}^{2}(\widehat{\mathcal{J}})^j|^2\leq \varrho_n^{2}[j], |\bar{L}^{2}(\widehat{\mathcal{J}})^j|^2\leq 4\varrho_n^{2}[j]\}$ we have $|\overline{\mathsf{Y}}^{2}(\widehat{\mathcal{J}})^j|\leq|\bar{L}^{2}(\widehat{\mathcal{J}})^j|+|\overline{\mathsf{Z}}^{2}(\widehat{\mathcal{J}})^j|\leq 3\sqrt{\varrho_n^{2}[j]}$, and thus by $(\ref{sumbarK})$, Lemma \ref{barhatN*} and $(\ref{eps3})$ we have
\begin{equation*}
\sup_{0\leq s\leq t}|\mathbb{III}_{4,s}^{(2)}|\lesssim k_n^{-1}\rho_n\sum_{i=1}^{\infty}1_{\{\tilde{N}^{1}(\bar{I}^i)\neq 0\}}
=o_p\left(\rho_n^{1-\beta/2}\right).
\end{equation*}

On the other hand, since
\begin{align*}
\mathbb{III}_{4,t}^{(3)}
=&\frac{1}{(\psi_{HY}k_n)^2}\sum_{i,j:\bar{R}^\vee(i,j)\leq t}\bar{L}^{1}(\widehat{\mathcal{I}})^i_t\overline{\mathsf{X}}^{2}(\widehat{\mathcal{J}})^j\bar{K}^{i j}1_{\{ N^{2}(\bar{J}^j)_t=\tilde{N}^{1}(\bar{I}^i)_t=0\}}\\
&\hphantom{\frac{1}{(\psi_{HY}k_n)^2}}\times 1_{\{|\overline{\mathsf{Z}}^{1}(\widehat{\mathcal{I}})^i|^2\leq\varrho_n^{1}[i] ,|\overline{\mathsf{Z}}^{2}(\widehat{\mathcal{J}})^j|^2\leq\varrho_n^{2}[j], |\bar{L}^{1}(\widehat{\mathcal{I}})^i_t|^2\leq 4\varrho_n^{1}[i] ,|\bar{L}^{2}(\widehat{\mathcal{J}})^j_t|^2\leq 4\varrho_n^{2}[j]\}},
\end{align*}
we can decompose it as
\begin{align*}
\mathbb{III}_{4,t}^{(1)}
&=\frac{1}{(\psi_{HY}k_n)^2}\sum_{i,j:\bar{R}^\vee(i,j)\leq t}\bar{L}^{1}(\widehat{\mathcal{I}})^i_t\left[\left\{\bar{A}^{2}(\widehat{\mathcal{J}})^j_t+\widetilde{\underline{A}^2}(\widehat{\mathcal{J}})^j_t\right\}+\bar{M}^{2}(\widehat{\mathcal{J}})^j_t
+\widetilde{\underline{M}^2}(\widehat{\mathcal{J}})^j_t+\overline{\zeta}^{2}(\widehat{\mathcal{J}})^j\right]\bar{K}^{i j}\\
&\hphantom{=\frac{1}{(\psi_{HY}k_n)^2}}\times 1_{\{N^{2}(\bar{J}^j)_t=\tilde{N}^{1}(\bar{I}^i)_t=0,|\overline{\mathsf{Z}}^{1}(\widehat{\mathcal{I}})^i|^2\leq\varrho_n^{1}[i] ,|\overline{\mathsf{Z}}^{2}(\widehat{\mathcal{J}})^j|^2\leq\varrho_n^{2}[j], |\bar{L}^{1}(\widehat{\mathcal{I}})^i_t|^2\leq 4\varrho_n^{1}[i] ,|\bar{L}^{2}(\widehat{\mathcal{J}})^j_t|^2\leq 4\varrho_n^{2}[j]\}}\\
&=:\Gamma_{1,t}+\Gamma_{2,t}+\Gamma_{3,t}+\Gamma_{4,t}.
\end{align*}

%%%%%%%%%%%%%%%%%%%%%%%%%%%%%%%%%%%%%%%%%%%%%%%%%%%%%%%%%%
%                  gamma1
%%%%%%%%%%%%%%%%%%%%%%%%%%%%%%%%%%%%%%%%%%%%%%%%%%%%%%%%%%

First consider $\Gamma_{1,t}$. By the Schwarz inequality and $(\ref{sumbarK})$, we have
\begin{align*}
\sup_{0\leq s\leq t}|\Gamma_{1,s}|
\lesssim&\frac{1}{k_n}\left\{\sum_{i:\widehat{S}^{i+k_n}\leq t}|\bar{L}^{1}(\widehat{\mathcal{I}})^i|^21_{\{|\bar{L}^{1}(\widehat{\mathcal{I}})^i|^2\leq 4\varrho_n^{1}[i]\}}\right\}^{1/2}
\left\{\sum_{j:\widehat{T}^{j+k_n}\leq t}|\bar{A}^{2}(\widehat{\mathcal{J}})^j+\widetilde{\underline{A}^2}(\widehat{\mathcal{J}})^j|^2\right\}^{1/2},
\end{align*}
hence Lemma \ref{matanian}, the boundedness of $(A^{2})'$ and $(\underline{A}^{2})'$, the Lipschitz continuity of $g$ and [SA6] yield
\begin{equation}\label{tphygamma1}
\sup_{0\leq s\leq t}|\Gamma_{1,s}|=o_p\left( k_n^{-1/2}\rho_n^{1/2-\beta/4}\right)= o_p\left( b_n^{1/4}\right).
\end{equation}

%%%%%%%%%%%%%%%%%%%%%%%%%%%%%%%%%%%%%%%%%%%%%%%%%%%%%%%%%%%
%                    gamma2
%%%%%%%%%%%%%%%%%%%%%%%%%%%%%%%%%%%%%%%%%%%%%%%%%%%%%%%%%%%

Next consider $\Gamma_{2,t}$. Since $\bar{L}^{1}(\widehat{\mathcal{I}})^i_t=\bar{\mathfrak{L}}^{1}(\widehat{\mathcal{I}})^i_t=\bar{\Xi}^{1}(\widehat{\mathcal{I}})^i_t+\bar{\Theta}^{1}(\widehat{\mathcal{I}})^i_t$ on $\{\tilde{N}^{2}(\bar{J}^j)_t=0\}$, we can decompose the target quantity as
\begin{align*}
\Gamma_{2,t}
&=\frac{1}{(\psi_{HY}k_n)^2}\sum_{i,j:\bar{R}^\vee(i,j)\leq t}\left\{\bar{\Xi}^{1}(\widehat{\mathcal{I}})^i_t\bar{M}^{2}(\widehat{\mathcal{J}})^j_t+\bar{\Theta}^{1}(\widehat{\mathcal{I}})^i_t\bar{M}^{2}(\widehat{\mathcal{J}})^j_t\right\}\bar{K}^{i j}\\
&\hphantom{=}\times 1_{\{ N^{2}(\bar{J}^j)_t=\tilde{N}^{2}(\bar{J}^j)_t=0, |\overline{\mathsf{Z}}^{1}(\widehat{\mathcal{I}})^i|^2\leq\varrho_n^{1}[i] ,|\overline{\mathsf{Z}}^{2}(\widehat{\mathcal{J}})^j|^2\leq\varrho_n^{2}[j], |\bar{L}^{1}(\widehat{\mathcal{I}})^i_t|^2\leq 4\varrho_n^{1}[i] ,|\bar{L}^{2}(\widehat{\mathcal{J}})^j_t|^2\leq 4\varrho_n^{2}[j]\}}\\
&=:\Gamma_{2,t}^{(1)}+\Gamma_{2,t}^{(2)}.
\end{align*}
First estimate $\Gamma_{2,t}^{(1)}$. We decompose it further as
\begin{align*}
\Gamma_{2,t}^{(1)}
&=\frac{1}{(\psi_{HY}k_n)^2}\sum_{i,j:\bar{R}^\vee(i,j)\leq t}\bar{\Xi}^{1}(\widehat{\mathcal{I}})^i_t\bar{M}^{2}(\widehat{\mathcal{J}})^j_t\bar{K}^{i j}\\
&\hphantom{=}\times \left(1-1_{\{ N^{2}(\bar{J}^j)_t\neq 0\}\cup\{\tilde{N}^{2}(\bar{J}^j)_t\neq 0\}\cup\{|\overline{\mathsf{Z}}^{1}(\widehat{\mathcal{I}})^i|^2>\varrho_n^{1}[i]\}\cup\{|\bar{L}^{1}(\widehat{\mathcal{I}})^i_t|^2> 4\varrho_n^{1}[i]\}\cup\{|\overline{\mathsf{Z}}^{2}(\widehat{\mathcal{J}})^j|^2>\varrho_n^{2}[j]\}\cup\{|\bar{L}^{2}(\widehat{\mathcal{J}})^j_t|^2> 4\varrho_n^{2}[j]\}\}}\right)\\
&=:\Gamma_{2,t}^{(1)\prime}+\Gamma_{2,t}^{(1)\prime\prime}.
\end{align*}
By Lemma \ref{XiM} we have $\sup_{0\leq s\leq t}|\Gamma_{2,s}^{(1)\prime}|=o_p(b_n^{1/4})$. On the other hand, $(\ref{estXiM})$, the Schwarz inequality, $(\ref{sumbarK})$, Lemma \ref{barhatN*}, Lemma \ref{barlarge}, $(\ref{sumU})$ and $(\ref{eps3})$ yield
\begin{align*}
\sup_{0\leq s\leq t}|\Gamma_{2,s}^{(1)\prime\prime}|
=&o_p\left( k_n^{-1}\varphi_2(\varepsilon_n)^{1/2}k_n^{1/2}\rho_n^{-\beta/4}\right) 
+ O_p\left( k_n^{-1}\varphi_2(\varepsilon_n)^{1/2}b_n^{-1/2}\eta_n^{r/4}\right).
\end{align*}
Since we have 
\begin{equation}\label{eps4}
\varphi_2(\varepsilon_n)\leq\varepsilon_n^{-\beta}\varepsilon_n^{2}\varphi_{\beta}(\varepsilon_n)\leq\rho_n^2\varepsilon_n^{-2-\beta}=o\left(\rho_n^{1-\beta/2}\right)
\end{equation}
due to $(\ref{eps2})$ for $p=2$ and $(\ref{eps3})$, note that $b_n=o(\rho_n^2)$, we obtain
$\sup_{0\leq s\leq t}|\Gamma_{2,s}^{(1)\prime\prime}|
=o_p\left( b_n^{1/4}\rho_n^{1/2-\beta/2}\right) 
+ o_p\left( \rho_n^{1/2-\beta/4}\eta_n^{r/4}\right)
=o_p\left( \rho_n^{1-\beta/2}\right) 
+ o_p\left( \eta_n^{r/2}\right).$
Consequently, we conclude that
\begin{equation}\label{tphygamma2(1)}
\sup_{0\leq s\leq t}|\Gamma_{2,s}^{(1)}|
=o_p\left( b_n^{1/4}\right) 
+ o_p\left( \rho_n^{1-\beta/2}\right) 
+ o_p\left( \eta_n^{r/2}\right).
\end{equation}
Next estimate $\Gamma_{2,t}^{(2)}$. We decompose it further as
\begin{align*}
\Gamma_{2,t}^{(2)}
&=\frac{1}{(\psi_{HY}k_n)^2}\sum_{i,j:\bar{R}^\vee(i,j)\leq t}\bar{\Theta}^{1}(\widehat{\mathcal{I}})^i_t\bar{M}^{2}(\widehat{\mathcal{J}})^j_t\bar{K}^{i j}\\
&\hphantom{=}\times \left(1-1_{\{ N^{2}(\bar{J}^j)_t\neq 0\}\cup\{\tilde{N}^{2}(\bar{J}^j)_t\neq 0\}\cup\{|\overline{\mathsf{Z}}^{1}(\widehat{\mathcal{I}})^i|^2>\varrho_n^{1}[i]\}\cup\{|\bar{L}^{1}(\widehat{\mathcal{I}})^i_t|^2> 4\varrho_n^{1}[i]\}\cup\{|\overline{\mathsf{Z}}^{2}(\widehat{\mathcal{J}})^j|^2>\varrho_n^{2}[j]\}\cup\{|\bar{L}^{2}(\widehat{\mathcal{J}})^j_t|^2> 4\varrho_n^{2}[j]\}\}}\right)\\
&=:\Gamma_{2,t}^{(2)\prime}+\Gamma_{2,t}^{(2)\prime\prime}.
\end{align*}
Lemma \ref{ThetaM} yields $\sup_{0\leq s\leq t}|\Gamma_{2,s}^{(3)\prime}|=o_p\left( b_n^{1/4}\right)+o_p\left( \rho_n^{1-\beta/2}\right)$. On the other hand, since $\bar{\Theta}^{1}(\widehat{\mathcal{I}})^i_t\bar{M}^{2}(\widehat{\mathcal{J}})^j_t\bar{K}^{ij}_t=\bar{K}^{ij}_-\bar{\Theta}^{1}(\widehat{\mathcal{I}})^i_-\bullet\bar{M}^{2}(\widehat{\mathcal{J}})^j_t+\bar{K}^{ij}_-\bar{M}^{2}(\widehat{\mathcal{J}})^j_-\bullet\bar{\Theta}^{1}(\widehat{\mathcal{I}})^i_t$ by integration by parts and Lemma 4.3 of \cite{Koike2012phy}, we have
\begin{align*}
E\left[|\bar{\Theta}^{1}(\widehat{\mathcal{I}})^i_t\bar{M}^{2}(\widehat{\mathcal{J}})^j_t\bar{K}^{i j}_t|^2\right]
\lesssim k_n^2\bar{r}_n\varepsilon_n^{-2(\beta-1)_+}E\left[\sum_{p=1}^{k_n-1}|\widehat{I}^{i+p}(t)|^2\bar{K}^{ij}_t\right]
\end{align*}
by [A2], the optional sampling theorem and the inequality  $|\Theta^{1}(\widehat{I}^p)_t|^2\lesssim\varepsilon_n^{-2(\beta-1)_+}|\widehat{I}^p(t)|^2$ and $(\ref{absmod2})$. Therefore, $(\ref{sumbarK})$ and [SA6] yield
\begin{align*}
E\left[\sum_{i,j:\bar{R}^\vee(i,j)\leq t}|\bar{\Theta}^{1}(\widehat{\mathcal{I}})^i_t\bar{M}^{2}(\widehat{\mathcal{J}})^j_t\bar{K}^{i j}|^2\right]
\lesssim k_n^4\bar{r}_n\varepsilon_n^{-2(\beta-1)_+}b_n\lesssim k_n^2\bar{r}_n\varepsilon_n^{-2(\beta-1)_+}
\end{align*}
and thus an argument similar to the above yields
\begin{align*}
\sup_{0\leq s\leq t}|\Gamma_{2,s}^{(2)\prime\prime}|
=&o_p\left( k_n^{-1/2}\bar{r}_n^{1/2}\varepsilon_n^{-(\beta-1)_+}k_n^{1/2}\rho_n^{-\beta/4}\right) 
+ O_p\left( k_n^{-1/2}\bar{r}_n^{1/2}\varepsilon_n^{-(\beta-1)_+}b_n^{-1/2}\eta_n^{r/4}\right)\\
=&o_p\left(\rho_n^{1-\beta/2}\right)+o_p\left(\eta_n^{r/2}\right).
\end{align*}
Consequently, we conclude that 
$\sup_{0\leq s\leq t}|\Gamma_{2,s}^{(2)}|=o_p\left( b_n^{1/4}\right) +o_p\left(\eta_n^{r/2}\right) + o_p\left( \rho_n^{1-\beta/2}\right).$
Combining this result with $(\ref{tphygamma2(1)})$, we conclude
\begin{equation}\label{tphygamma2}
\sup_{0\leq s\leq t}|\Gamma_{2,s}|=o_p\left( b_n^{1/4}\right) +o_p\left(\eta_n^{r/2}\right) + o_p\left( \rho_n^{1-\beta/2}\right).
\end{equation}
Similarly we can also show that
\begin{equation}\label{tphygamma3}
\sup_{0\leq s\leq t}|\Gamma_{3,s}|=o_p\left( b_n^{1/4}\right) +o_p\left(\eta_n^{r/2}\right) + o_p\left( \rho_n^{1-\beta/2}\right).
\end{equation}

%%%%%%%%%%%%%%%%%%%%%%%%%%%%%%%%%%%%%%%%%%%%%%%%%%%%%%%%%%
%                       gamma4
%%%%%%%%%%%%%%%%%%%%%%%%%%%%%%%%%%%%%%%%%%%%%%%%%%%%%%%%%%

Now we deal with $\Gamma_{4,t}$. Since $\bar{L}^{1}(\widehat{\mathcal{I}})^i_t=\bar{\mathfrak{L}}^{1}(\widehat{\mathcal{I}})^i_t$ on $\{\tilde{N}^{2}(\bar{J}^j)_t=0\}$, we can decompose the target quantity as
\begin{align*}
\Gamma_{4,t}
&=\frac{1}{(\psi_{HY}k_n)^2}\sum_{i,j:\bar{R}^\vee(i,j)\leq t}\bar{\mathfrak{L}}^{1}(\widehat{\mathcal{I}})^i_t\overline{\zeta}^{2}(\widehat{\mathcal{J}})^j\bar{K}^{i j}\\
&\hphantom{=}\times \left(1-1_{\{ N^{2}(\bar{J}^j)_t\neq 0\}\cup\{\tilde{N}^{2}(\bar{J}^j)_t\neq 0\}\cup\{|\overline{\mathsf{Z}}^{1}(\widehat{\mathcal{I}})^i|^2>\varrho_n^{1}[i]\}\cup\{|\bar{L}^{1}(\widehat{\mathcal{I}})^i_t|^2> 4\varrho_n^{1}[i]\}\cup\{|\overline{\mathsf{Z}}^{2}(\widehat{\mathcal{J}})^j|^2>\varrho_n^{2}[j]\}\cup\{|\bar{L}^{2}(\widehat{\mathcal{J}})^j_t|^2> 4\varrho_n^{2}[j]\}\}}\right)\\
&=:\Gamma_{4,t}^{(1)}+\Gamma_{4,t}^{(2)}.
\end{align*}
By Lemma \ref{XiM} and Lemma \ref{ThetaM} we have $\sup_{0\leq s\leq t}|\Gamma_{4,s}^{(1)}|=o_p(b_n^{1/4})+o_p(\rho_n^{1-\beta/2})$. On the other hand, by the Lipschitz continuity of $g$, $(\ref{sumbarK})$ and [SN$^\flat_2$], we have
\begin{equation*}
E_{0}\left[\sum_{i,j:\bar{R}^\vee(i,j)\leq t}|\bar{\mathfrak{L}}^{1}(\widehat{\mathcal{I}})^i_t|^2|\overline{\zeta}^{2}(\widehat{\mathcal{J}})^j|^2\bar{K}^{i j}\right]
\lesssim\sum_{i=1}^{\infty}|\bar{\mathfrak{L}}^{(1)}(\widehat{\mathcal{I}})^i_t|^2,
\end{equation*}
hence the Schwarz inequality, $(\ref{sumbarK})$, Lemma \ref{barhatN*}, Lemma \ref{barlarge}, $(\ref{sumU})$ and $(\ref{eps3})$ yield
\begin{align*}
\sup_{0\leq s\leq t}|\Gamma_{4,s}^{(2)}|
=&o_p\left( k_n^{-1}\varepsilon_n^{1-\beta/2}\varphi_\beta(\varepsilon_n)^{1/2}k_n^{1/2}\rho_n^{-\beta/4}\right) 
+ O_p\left( k_n^{-1}\varepsilon_n^{1-\beta/2}\varphi_\beta(\varepsilon_n)^{1/2}b_n^{-1/2}\eta_n^{r/4}\right).
\end{align*}
Since we have $(\ref{eps4})$ due to $(\ref{eps2})$ for $p=2$ and $(\ref{eps3})$, note that $b_n=o(\rho_n^2)$, we obtain
\begin{align*}
\sup_{0\leq s\leq t}|\Gamma_{4,s}^{(2)}|
=&o_p\left( b_n^{1/4}\rho_n^{1/2-\beta/2}\right) 
+ o_p\left( \rho_n^{1/2-\beta/4}\eta_n^{r/4}\right)
=o_p\left( \rho_n^{1-\beta/2}\right) 
+ o_p\left( \eta_n^{r/2}\right).
\end{align*}
Consequently, we obtain
\begin{equation}\label{tphygamma4}
\sup_{0\leq s\leq t}|\Gamma_{4,s}|
=o_p\left( b_n^{1/4}\right) +o_p\left( \rho_n^{1-\beta/2}\right) 
+ o_p\left( \eta_n^{r/2}\right).
\end{equation}

%%%%%%%%%%%%%%%%%%%%%%%%%%%%%%%%%%%%%%%%%%%%%%%%%%%%%%%%%
%                 tphyIII4(3)est
%%%%%%%%%%%%%%%%%%%%%%%%%%%%%%%%%%%%%%%%%%%%%%%%%%%%%%%%%%

By $(\ref{tphygamma1})$, $(\ref{tphygamma2})$ and $(\ref{tphygamma3})$, we obtain
$\sup_{0\leq s\leq t}|\mathbb{III}_{4,s}^{(3)}|=o_p\left( b_n^{1/4}\right) +o_p\left( \rho_n^{1-\beta/2}\right) 
+ o_p\left( \eta_n^{r/2}\right).$
Consequently, we obtain
\begin{equation}\label{tphyIII4}
\sup_{0\leq s\leq t}|\mathbb{III}_{4,s}|
=o_p\left( b_n^{1/4}\right) +o_p\left( \rho_n^{1-\beta/2}\right) 
+ o_p\left( \eta_n^{r/2}\right).
\end{equation}

Note that $\eta_n=o(1)$ and $\beta\geq 0$, $(\ref{tphyIII1})$, $(\ref{tphyIII2})$, $(\ref{tphyIII3})$ and $(\ref{tphyIII4})$ yield
\begin{equation}\label{tphyIII}
\sup_{0\leq s\leq t}|\mathbb{III}_{s}|
=o_p\left( b_n^{1/4}\right) + o_p\left(\eta_n^{r/2}\right) + o_p\left(\rho_n^{1-\beta/2}\right).
\end{equation}

%%%%%%%%%%%%%%%%%%%%%%%%%%%%%%%%%%%%%%%%%%%%%%%%%%%%%%
%                          tphyIV
%%%%%%%%%%%%%%%%%%%%%%%%%%%%%%%%%%%%%%%%%%%%%%%%%%%%%%

(d) By symmetry, we obtain
\begin{equation}\label{tphyIV}
\sup_{0\leq s\leq t}|\mathbb{IV}_{s}|
=o_p\left( b_n^{1/4}\right) + o_p\left(\eta_n^{r/2}\right) + o_p\left(\rho_n^{1-\beta/2}\right).
\end{equation}

%%%%%%%%%%%%%%%%%%%%%%%%%%%%%%%%%%%%%%%%%%%%%%%%%%%
%                          tphyV
%%%%%%%%%%%%%%%%%%%%%%%%%%%%%%%%%%%%%%%%%%%%%%%%%%%%%

(e) Finally we consider $\mathbb{V}$. We decompose it as
\begin{align*}
\mathbb{V}_t
&=\frac{1}{(\psi_{HY}k_n)^2}\sum_{i,j:\bar{R}^\vee(i,j)\leq t}\bar{L}^{1}(\widehat{\mathcal{I}})^i_t\bar{L}^{2}(\widehat{\mathcal{J}})^j_t\bar{K}^{i j}1_{\{\overline{\mathsf{Z}}^{1}(\widehat{\mathcal{I}})^i|^2\leq\varrho_n^{1}[i] ,\overline{\mathsf{Z}}^{2}(\widehat{\mathcal{J}})^j|^2\leq\varrho_n^{2}[j]\}}\\
&\hphantom{=\sum_{i,j:S^i\vee T^j\leq t}}\times \left(1_{\{\bar{L}^{1}(\widehat{\mathcal{I}})^i_t|^2> 4\varrho_n^{1}[i],\bar{L}^{2}(\widehat{\mathcal{J}})^j_t|^2> 4\varrho_n^{2}[j]\}}+1_{\{ \bar{L}^{1}(\widehat{\mathcal{I}})^i_t|^2> 4\varrho_n^{1}[i], \bar{L}^{2}(\widehat{\mathcal{J}})^j_t|^2\leq 4\varrho_n^{2}[j]\}}\right.\\
&\hphantom{=\sum_{i,j:S^i\vee T^j\leq t}\times(}\left.+1_{\{\bar{L}^{1}(\widehat{\mathcal{I}})^i_t|^2\leq 4\varrho_n^{1}[i], \bar{L}^{2}(\widehat{\mathcal{J}})^j_t|^2> 4\varrho_n^{2}[j]\}}+1_{\{ \bar{L}^{1}(\widehat{\mathcal{I}})^i_t|^2\leq 4\varrho_n^{1}[i] ,\bar{L}^{2}(\widehat{\mathcal{J}})^j_t|^2\leq 4\varrho_n^{2}[j]\}}\right)\\
&=:\mathbb{V}_{1,t}+\mathbb{V}_{2,t}+\mathbb{V}_{3,t}+\mathbb{V}_{4,t}.
\end{align*}

%%%%%%%%%%%%%%%%%%%%%%%%%%%%%%%%%%%%%%%%%%%%%%%%%%%%%%%%%%%%%%%%%%%%%%%%%%%%
%                              tphyV1-3
%%%%%%%%%%%%%%%%%%%%%%%%%%%%%%%%%%%%%%%%%%%%%%%%%%%%%%%%%%%%%%%%%%%%%%%%%%%%

By the Schwarz inequality, $(\ref{sumbarK})$ and Lemma \ref{lemL2Z} we have
$\sup_{0\leq s\leq t}|\mathbb{V}_{1,s}|=O_p\left(\eta_n^{\frac{r}{2}}\right).$
Moreover, by $(\ref{sumbarK})$ and Lemma \ref{lemLZ} we have
$\sup_{0\leq s\leq t}|\mathbb{V}_{2,s}|
=o_p\left(\eta_n^{r/4}\rho_n^{1/2-\beta/4}\right)$ and
$\sup_{0\leq s\leq t}|\mathbb{V}_{3,s}|
=o_p\left(\eta_n^{r/4}\rho_n^{1/2-\beta/4}\right).$
Furthermore, the Schwarz inequality, $(\ref{sumbarK})$ and Lemma \ref{matanian} yield
$\sup_{0\leq s\leq t}|\mathbb{V}_{4,s}|=o_p\left(\rho_n^{1-\beta/2}\right).$
Consequently, we obtain
\begin{equation}\label{tphyV}
\sup_{0\leq s\leq t}|\mathbb{V}_{s}|=O_p\left(\eta_n^{r/2}\right) + o_p\left(\rho_n^{1-\beta/2}\right).
\end{equation}

%%%%%%%%%%%%%%%%%%%%%%%%%%%%%%%%%%%%%%%%%%%%%%%%%%%%%%%%%%%%%%%%%%%%%%%%%
%                    consequent
%%%%%%%%%%%%%%%%%%%%%%%%%%%%%%%%%%%%%%%%%%%%%%%%%%%%%%%%%%%%%%%%%%%%%%%%%

Note that $\eta_n=o(1)$, $(\ref{tphy0})$, $(\ref{tphyI})$, $(\ref{tphyII})$, $(\ref{tphyIII})$, $(\ref{tphyIV})$ and $(\ref{tphyV})$ yield
\begin{equation*}
\sup_{0\leq s\leq t}|PTHY(\mathsf{Z}^{1}, \mathsf{Z}^{2})^n_s-PHY(\mathsf{X}^{1}, \mathsf{X}^{2})^n_s|
=o_p\left( b_n^{1/4}\right) + O_p\left(\eta_n^{\frac{r-2}{2}}\right) + o_p\left(\rho_n^{1-\beta/2}\right).
\end{equation*}  
Since $\eta_n=(k_n\rho_n)^{-1}=O(b_n^{1/2}\rho_n^{-1})$, we complete the proof of Theorem \ref{thmIAnoise}.
\end{proof}

%% file: pthy/pthy_propphy.tex
\section{Proof of Proposition \ref{propphy}}\label{proofpropphy}

By a localization procedure, we may assume that [SC1]-[SC2], [SA4], [SA6] and [SN$^\flat_2$] hold. In a similar manner we may also assume that $E^{1}=E^{2}=:E$ and that there is a non-negative bounded measurable function $\psi$ on $E$ such that
\begin{equation*}
\sup_{\omega\in\Omega ,t\in\mathbb{R}_+}|\delta^{l}(\omega, t,x)|\leq\psi(x)\quad\textrm{and}\quad\int_E \psi(x)^{2}F^{l}(\mathrm{d}x)<\infty ,\qquad l=1,2.
\end{equation*}

%%%%%%%%%%%%%%%%%%%%%%%%%%%%%%%%%%%%%%%%%%%%%%%%%%%%%%%%%%%%%%%%
%                 decomposition of PHY
%%%%%%%%%%%%%%%%%%%%%%%%%%%%%%%%%%%%%%%%%%%%%%%%%%%%%%%%%%%%%%%%%

Under the above assumption we can define the process $L^{\prime l}$ by $L^{\prime l}=\delta^l\star(\mu^l-\nu^l)$ for each $l=1,2$. Then, for each $l=1,2$ we have $Z^l=X^{\prime l}+L^{\prime l}$, where $X^{\prime l}_t=X^l_t-\int_0^t\int_E\kappa'(\delta^l(s,x))\mathrm{d}sF(\mathrm{d}x)$. Hence we can decompose the target quantity as
\begin{align*}
&PHY(\mathsf{Z}^{1},\mathsf{Z}^{2})^n_t\\
=&\frac{1}{(\psi_{HY}k_n)^2}\sum_{i,j: \bar{R}^\vee(i,j)\leq t}\left\{\overline{\mathsf{X}}^{\prime 1}(\widehat{\mathcal{I}})^i\overline{\mathsf{X}}^{\prime 2}(\widehat{\mathcal{J}})^j+\overline{\mathsf{X}}^{\prime 1}(\widehat{\mathcal{I}})^i\bar{L}^{\prime 2}(\widehat{\mathcal{J}})^j_t+\bar{L}^{\prime 1}(\widehat{\mathcal{I}})^i_t\overline{\mathsf{X}}^{\prime 2}(\widehat{\mathcal{J}})^j+\bar{L}^{\prime 1}(\widehat{\mathcal{I}})^i_t\bar{L}^{\prime 2}(\widehat{\mathcal{J}})^j_t\right\} \bar{K}^{i j}
\\
=:&\mathbb{I}_t+\mathbb{II}_t+\mathbb{III}_t+\mathbb{IV}_t,
\end{align*}
where
$\overline{\mathsf{X}}^{\prime1}(\widehat{\mathcal{I}})^i=\sum_{p=1}^{k_n-1}g\left (\frac{p}{k_n}\right)\left(\mathsf{X}^{\prime1}_{\widehat{S}^{i+p}}-\mathsf{X}^{\prime1}_{\widehat{S}^{i+p-1}}\right),$
$\overline{\mathsf{X}}^{\prime2}(\widehat{\mathcal{J}})^j=\sum_{q=1}^{k_n-1}g\left (\frac{q}{k_n}\right)\left(\mathsf{X}^{\prime2}_{\widehat{T}^{j+q}}-\mathsf{X}^{\prime2}_{\widehat{T}^{j+q-1}}\right)$
and $\mathsf{X}^{\prime1}_{\widehat{S}^{i}}=X^{\prime1}_{\widehat{S}^i}+U^1_{\widehat{S}^i}$, $\mathsf{X}^{\prime2}_{\widehat{T}^{j}}=X^{\prime2}_{\widehat{T}^j}+U^2_{\widehat{T}^j}$ for every $i,j$.

First, we can adopt an argument similar to the proof of Lemma \ref{XiM} for the proof of $\mathbb{II}_t=O_p(b_n^{1/4})$ and $\mathbb{III}_t=O_p(b_n^{1/4})$. Next, combining Lemma 4.2 of \cite{Koike2012phy} with an argument similar to the proof of Lemma \ref{XiM} we can show that $\mathbb{I}_t=[X^{\prime1},X^{\prime2}]_t+O_p(b_n^{1/4})$. Finally we consider $\mathbb{IV}_t$. By integration by parts we can decompose it as
\begin{align*}
\mathbb{IV}_t&=
\frac{1}{(\psi_{HY}k_n)^2}\sum_{i,j: \bar{R}^\vee(i,j)\leq t}\left\{\bar{L}^{\prime 1}(\widehat{\mathcal{I}})^i_-\bullet\bar{L}^{\prime 2}(\widehat{\mathcal{J}})^j_t+\bullet\bar{L}^{\prime 2}(\widehat{\mathcal{J}})^j_-\bullet\bar{L}^{\prime 1}(\widehat{\mathcal{I}})^i_t+[\bar{L}^{\prime 1}(\widehat{\mathcal{I}})^i,\bar{L}^{\prime 2}(\widehat{\mathcal{J}})^j]_t\right\}\bar{K}^{i j}\\
&=:\mathbb{IV}^{(1)}_t+\mathbb{IV}^{(4)}_t+\mathbb{IV}^{(3)}_t.
\end{align*}
By an argument similar to the proof of Lemma \ref{XiM} we can show that $\mathbb{IV}^{(1)}_t=O_p(b_n^{1/4})$ and $\mathbb{IV}^{(2)}_t=O_p(b_n^{1/4})$. On the other hand, since
\begin{align*}
[\bar{L}^{\prime1}(\widehat{\mathcal{I}})^i,\bar{L}^{\prime2}(\widehat{\mathcal{J}})^j]_t
=\sum_{p,q=0}^{k_n-1}g^n_pg^n_q(\widehat{I}^{i+p}_-\widehat{J}^{j+q}_-)\bullet[L^{\prime1},L^{\prime2}]_t
=\sum_{p=i}^{i+k_n-1}\sum_{q=i}^{j+k_n-1}g^n_{p-i}g^n_{q-j}(\widehat{I}^{p}_-\widehat{J}^{q}_-)\bullet[L^{\prime1},L^{\prime2}]_t,
\end{align*}
we obtain
\begin{align*}
\mathbb{IV}_t^{(3)}=\frac{1}{(\psi_{HY}k_n)^2}
\sum_{p,q=1}^{\infty}\left(\sum_{i=(p-k_n+1)\vee1}^p\sum_{j=(q-k_n+1)\vee1}^q g^n_{p-i}g^n_{q-j}\bar{K}^{ij}1_{\{\bar{R}^\vee(i,j)\leq t\}}\right)(\widehat{I}^{p}_-\widehat{J}^{q}_-)\bullet[L^{\prime1},L^{\prime2}]_t.
\end{align*}
Moreover, on $\{\widehat{I}^p\cap\widehat{J}^q\neq\emptyset\}$ we have $\widehat{S}^{p-1}<\widehat{T}^q$ and $\widehat{T}^{q-1}<\widehat{S}^p$, hence for $i\in\{(p-k_n+1)\vee1,\dots,p-1\}$ and $j\in\{(q-k_n+1)\vee1,\dots,q-1\}$ we have $\widehat{S}^i<\widehat{T}^{j+k_n-1}$ and $\widehat{T}^j<\widehat{S}^{i+k_n-1}$, so that $\bar{K}^{ij}=1$. Therefore, for $p,q\geq k_n$ we have
$\sum_{i=(p-k_n+1)\vee1}^p\sum_{j=(q-k_n+1)\vee1}^q g^n_{p-i}g^n_{q-j}\bar{K}^{ij}_t
=\left(\sum_{i=1}^{k_n-1}g_i^n\right)^2$
on $\{\widehat{I}^p\cap\widehat{J}^q\neq\emptyset\}$ because $g(0)=0$. Since $(\widehat{I}^{p}_-\widehat{J}^{q}_-)\bullet[L^{\prime1},L^{\prime2}]_t=1_{\{\widehat{I}^p\cap\widehat{J}^q\neq\emptyset\}}(\widehat{I}^{p}_-\widehat{J}^{q}_-)\bullet[L^{\prime1},L^{\prime2}]_t$ and $\frac{1}{(\psi_{HY}k_n)^2}\left(\sum_{i=1}^{k_n-1}g_i^n\right)^2=1+O(k_n^{-1})$ by the Lipschitz continuity of $g$, we obtain
\begin{align*}
[L^{\prime1},L^{\prime2}]_t
&=\frac{1}{(\psi_{HY}k_n)^2}\sum_{p,q=1}^\infty\left(\sum_{i=1}^{k_n-1}g_i^n\right)^2(\widehat{I}^{p}_-\widehat{J}^{q}_-)\bullet[L^{\prime1},L^{\prime2}]_t+O_p\left(k_n^{-1}\right)\\
&=\frac{1}{(\psi_{HY}k_n)^2}\sum_{p,q=1}^\infty\left(\sum_{i=(p-k_n+1)\vee1}^p\sum_{j=(q-k_n+1)\vee1}^q g^n_{p-i}g^n_{q-j}\bar{K}^{ij}_t\right)(\widehat{I}^{p}_-\widehat{J}^{q}_-)\bullet[L^{\prime1},L^{\prime2}]_t+O_p\left(k_n^{-1}\right)\\
&=:\mathbf{IV}^{(3)}_t+O_p\left(k_n^{-1}\right).
\end{align*}
Since we have
\begin{align*}
&\mathbb{IV}_t^{(3)}-\mathbf{IV}_t^{(3)}\\
=&\frac{1}{(\psi_{HY}k_n)^2}\sum_{p,q=1}^{\infty}\left(\sum_{i=(p-k_n+1)\vee1}^p\sum_{j=(q-k_n+1)\vee1}^q g^n_{p-i}g^n_{q-j}\bar{K}^{ij}_t1_{\{\bar{R}^\vee(i,j)> t\}}\right)(\widehat{I}^{p}_-\widehat{J}^{q}_-)\bullet[L^{\prime1},L^{\prime2}]_t+O_p(b_n^{1/2})\\
=&\frac{1}{(\psi_{HY}k_n)^2}\sum_{i,j=0}^\infty\bar{K}^{ij}_t 1_{\{\bar{R}^\vee(i,j)> t\}}\sum_{p,q=0}^{k_n-1}g^n_pg^n_q(\widehat{I}^{i+p}_-\widehat{J}^{j+q}_-)\bullet[L^{\prime1},L^{\prime2}]_t+O_p(b_n^{1/2}),
\end{align*}
we obtain
$\left|\mathbb{IV}_t^{(3)}-\mathbf{IV}_t^{(3)}\right|
\lesssim\frac{1}{k_n^2}\sum_{i,j=0}^\infty\bar{K}^{ij}_t 1_{\{\bar{R}^\vee(i,j)> t\}}(\bar{I}^{i}_-\bar{J}^{j}_-)\bullet\left|[L^{\prime1},L^{\prime2}]\right|_t+O_p(b_n^{1/2}),$
and thus the Kunita-Watanabe inequality and the inequality of arithmetic and geometric means yield
\begin{align*}
\left|\mathbb{IV}_t^{(3)}-\mathbf{IV}_t^{(3)}\right|
\lesssim\frac{1}{k_n^2}\sum_{i,j=0}^\infty\bar{K}^{ij}_t 1_{\{\bar{R}^\vee(i,j)> t\}}\left\{[L^{\prime1}](\bar{I}^{i})_t+[L^{\prime2}](\bar{J}^{j})_t\right\}+O_p(b_n^{1/2}).
\end{align*}
Since $\bar{K}^{ij}_t 1_{\{\bar{R}^\vee(i,j)> t\}}$ is $\mathcal{F}_{\widehat{S}^i\wedge\widehat{T}^j}$-measurable by Lemma \ref{barKssp}, we obtain
{\small \begin{align*}
E\left[\left|\sum_{i,j=0}^\infty\bar{K}^{ij}_t 1_{\{\bar{R}^\vee(i,j)> t\}}\left\{[L^{\prime1}](\bar{I}^{i})_t+[L^{\prime2}](\bar{J}^{j})_t\right\}\right|\right]
\lesssim E\left[\sum_{i,j=0}^\infty\bar{K}^{ij}_t 1_{\{\bar{R}^\vee(i,j)> t\}}\left\{\langle L^{\prime1}\rangle(\bar{I}^{i})_t+\langle L^{\prime2}\rangle(\bar{J}^{j})_t\right\}\right],
\end{align*}}
hence by [SK$_2$], [SA4] and $(\ref{estshift})$ we conclude that $\left|\mathbb{IV}_t^{(3)}-\mathbf{IV}_t^{(3)}\right|= O_p(k_n\bar{r}_n)+O_p(b_n^{1/2})=o_p(b_n^{1/4})$. Consequently, we obtain $\mathbb{IV}_t^{(3)}=[L^{\prime1},L^{\prime2}]_t+o_p(b_n^{1/4})$, and thus we complete the proof of the proposition because $[Z^1,Z^2]=[X^{\prime1},X^{\prime2}]+[L^{\prime1},L^{\prime2}]$.\hfill $\Box$

%% file: pthy/pthy_dependent.tex
\section{Proof of Proposition \ref{depcon}}\label{proofdepcon}

By a localization procedure, we may systematically replace the conditions [C1]-[C2], [A4], [A6], [N$^\flat_2$], [T] and [N$^\flat_r$] with [SC1]-[SC2], [SA4], [SA6], [SN$^\flat_2$], [ST] and [SN$^\flat_r$] respectively.

Set $\tilde{\lambda}^l_u=\sum_{v=u}^\infty\lambda^l_v$ for each $u\in\mathbb{Z}_+$ and $l=1,2$. We define the random variables $\tilde{\zeta}^1_{i}$ and $\tilde{\zeta}^2_{j}$ by $\tilde{\zeta}^1_i=\sum_{u=0}^i\tilde{\lambda}^1_{u+1}\zeta^1_{S^{i-u}}$ and $\tilde{\zeta}^2_j=\sum_{u=0}^i\tilde{\lambda}^2_{u+1}\zeta^2_{T^{j-u}}$ for every $i,j$. Then we have
\begin{align*}
\tilde{\zeta}^1_{i}-\tilde{\zeta}^1_{i-1}
=\sum_{u=0}^i\tilde{\lambda}^1_{u+1}\zeta^1_{S^{i-u}}-\sum_{u=1}^i\tilde{\lambda}^1_{u}\zeta^1_{S^{i-u}}
=\sum_{u=0}^i(\tilde{\lambda}^1_{u+1}-\tilde{\lambda}^1_{u})\zeta^1_{S^{i-u}}+\tilde{\lambda}_{0}\zeta^1_{S^i}
=-\sum_{u=0}^i\lambda_u\zeta^1_{S^{i-u}}+\tilde{\lambda}_{0}\zeta^1_{S^i},
\end{align*}
hence we obtain
\begin{equation}\label{BN}
\sum_{u=0}^i\lambda_u\zeta^1_{S^{i-u}}
=\tilde{\lambda}_{0}\zeta^1_{S^i}-(\tilde{\zeta}^1_{i}-\tilde{\zeta}^1_{i-1}).
\end{equation}
In the time series analysis, this relation is known as the Beveridge-Nelson decomposition. See \cite{BN1981} for details. Combining $(\ref{BN})$ with Abel's partial summation formula, we obtain
\begin{align*}
\sum_{p=0}^{k_n-1}\Delta(g)^n_p\left(\sum_{u=0}^{i+p}\lambda_u\zeta^1_{S^{i+p-u}}\right)
&=\tilde{\lambda}_{0}\overline{\zeta}^1(\widehat{\mathcal{I}})^i-\sum_{p=0}^{k_n-1}\Delta(g)^n_p(\tilde{\zeta}^1_{i+p}-\tilde{\zeta}^1_{i+p-1})\\
&=\tilde{\lambda}_{0}\overline{\zeta}^1(\widehat{\mathcal{I}})^i+\sum_{p=0}^{k_n-1}\{\Delta(g)^n_{p+1}-\Delta(g)^n_p\}(\tilde{\zeta}^1_{i+p}-\tilde{\zeta}^1_{i-1}).
\end{align*}
for every $i$. Similarly we can deduce
\begin{align*}
\sum_{q=0}^{k_n-1}\Delta(g)^n_q\left(\sum_{u=0}^{j+q}\lambda^2_u\zeta^2_{T^{j+q-u}}\right)
&=\tilde{\lambda}^2_{0}\overline{\zeta}^2(\widehat{\mathcal{J}})^j+\sum_{q=0}^{k_n-1}\{\Delta(g)^n_{q+1}-\Delta(g)^n_q\}(\tilde{\zeta}^2_{j+q}-\tilde{\zeta}^2_{j-1}),\\
\sum_{p=0}^{k_n-1}\Delta(g)^n_p\left(\sum_{u=0}^{i+p}\mu^1_u b_n^{-1/2}\underline{X}^1(I^{i+p-u})_t\right)
&=\tilde{\mu}^1_{0}\widetilde{\underline{X}^1}(\widehat{\mathcal{I}})^i_t+b_n^{-1/2}\sum_{p=0}^{k_n-1}\{\Delta(g)^n_{p+1}-\Delta(g)^n_p\}(\underline{X}^1(I^{i+p})_t-\underline{X}^1(I^{i-1})_t),\\
\sum_{q=0}^{k_n-1}\Delta(g)^n_q\left(\sum_{u=0}^{j+q}\mu^2_u b_n^{-1/2}\underline{X}^2(J^{j+q-u})_t\right)
&=\tilde{\mu}^2_{0}\widetilde{\underline{X}^2}(\widehat{\mathcal{J}})^j_t+b_n^{-1/2}\sum_{q=0}^{k_n-1}\{\Delta(g)^n_{q+1}-\Delta(g)^n_q\}(\underline{X}^2(J^{j+q})_t-\underline{X}^2(J^{j-1})_t)
\end{align*}
for every $i,j$, where $\tilde{\mu}^l_u=\sum_{v=u}^\infty\mu^l_v$ for each $u\in\mathbb{Z}_+$ and $l=1,2$.
Note that $\sum_{u=1}^{\infty}|\tilde{\lambda}^l_u|\leq \sum_{u=1}^{\infty}\sum_{v=u}^{\infty}|\lambda^l_v|=\sum_{v=1}^{\infty}v|\lambda^l_v|<\infty$ and $\sum_{u=1}^{\infty}|\tilde{\mu}^l_u|\leq \sum_{u=1}^{\infty}\sum_{v=u}^{\infty}|\mu^l_v|=\sum_{v=1}^{\infty}v|\mu^l_v|<\infty$ for $l=1,2$, by using the above formulas we can show that
\begin{align*}
\widehat{PHY}(\mathsf{Z}^1,\mathsf{Z}^2)^n_t-\widehat{PHY}(\tilde{\mathsf{Z}}^1,\tilde{\mathsf{Z}}^2)^n_t\to^p 0
\quad\mathrm{and}\quad
\widehat{PTHY}(\mathsf{Z}^1,\mathsf{Z}^2)^n_t-\widehat{PTHY}(\tilde{\mathsf{Z}}^1,\tilde{\mathsf{Z}}^2)^n_t\to^p 0
\end{align*}
as $n\to\infty$ for any $t\in\mathbb{R}_+$, where $\tilde{\mathsf{Z}}^1_{S^i}=Z^1_{S^i}+\tilde{\lambda}^1_0\zeta^1_{S^i}+\tilde{\mu}^1_0b_n^{-1/2}(\underline{X}^1_{S^i}-\underline{X}^1_{S^{i-1}})$ and $\tilde{\mathsf{Z}}^2_{T^j}=Z^2_{T^j}+\tilde{\lambda}^2_0\zeta^2_{T^j}+\tilde{\mu}^2_0b_n^{-1/2}(\underline{X}^2_{T^j}-\underline{X}^2_{T^{j-1}})$ for each $i,j$. Consequently, we complete the proof of the proposition due to Proposition \ref{propphy} and Corollary \ref{conpthy}.\hfill $\Box$

%% file: pthy/pthy_acknowledgements.tex
%%%%%%%%%%%%%%%%%%%%%%%%%%%%%%%%%%%%%%%%%%%%%%%%%%%%%%%%%%%%%%%%%%%%%
%               Acknowledgements
%%%%%%%%%%%%%%%%%%%%%%%%%%%%%%%%%%%%%%%%%%%%%%%%%%%%%%%%%%%%%%%%%%%%%

\section*{Acknowledgements}

\addcontentsline{toc}{section}{Acknowledgements}

I am indebted to Professor Nakahiro Yoshida for his encouragement to my research and valuable suggestions which led to considerable improvements of the paper.